\documentclass[a4paper,fleqn]{cas-sc}

\usepackage[numbers]{natbib}

\usepackage[utf8]{inputenc} 
\usepackage{accents}				
\usepackage{todonotes}
\usepackage{subcaption}
\usepackage{stmaryrd} 
\usepackage{amsthm} 
\usepackage[makeroom]{cancel} 
\usepackage{bookmark} 

\newcommand\numberthis{\addtocounter{equation}{1}\tag{\theequation}}

\newcommand{\norm}[1]{\left\lVert#1\right\rVert}
\newcommand{\mat}[1]{\underline{\mathbf{#1}}}
\newcommand\acclrvec[1]{\accentset{\,\leftrightarrow}{#1}}	
\newcommand{\blocktensor}[1]{\acclrvec{{\mathbf #1}}}	
\newcommand\blocktensorG[1]{\acclrvec{\boldsymbol #1}}	
\newcommand{\Nabla} {\vec{\nabla}}
\newcommand{\numfluxb}[1]{\hat{\mathbf{#1}} }
\newcommand{\numflux}[1]{\hat{{#1}} }
\newcommand\threeMatrix[1]{\underline{ #1}}				
\newcommand{\ec}{{\mathrm{EC}}}			
\newcommand{\es}{{\mathrm{ES}}}	
\newcommand{\tvdes}{{\mathrm{TVD-ES}}}			
\def\d{\mathrm{d}}
\newcommand{\minmod}{\mathrm{MINMOD}}
\newcommand{\sign}{\mathrm{sign}}

\newcommand{\bigpartialderiv}[2]{ \frac{\partial {#1}}{\partial {#2} } }

\newcommand\stateG[1]{\boldsymbol #1}			

\newcommand{\noncon}{\stateG{\Upsilon}}	

\newcommand{\entVar}{{\mathbf{v}}}
\newcommand{\scalEntVar}{{\mathbf{w}}}

\newcommand\state[1]{\mathbf{#1}}
\newcommand{\contSt}[1]{\tilde{\state{#1}}}

\newcommand{\supEuler}{{\mathrm{Euler}}}
\newcommand{\supMHD}{{\mathrm{MHD}}}
\newcommand{\supGLM}{{\mathrm{GLM}}}
\newcommand{\DG}{{\mathrm{DG}}}
\newcommand{\FV}{{\mathrm{FV}}}

\newcommand{\Jan}{\stateG{\Phi}}
\newcommand{\JanVec}{\blocktensorG{\Phi}}
\newcommand{\phiMHD}{\stateG{\phi}^\supMHD}

\newcommand{\phiGLM}{\blocktensorG{\phi}^\supGLM}
\newcommand{\phiGLMs}{\stateG{\phi}^\supGLM}

\newcommand{\vn}{\vec{\tilde{n}}}

\newcommand{\numnonconsD}[1]{ #1^{\Diamond} }

\newcommand{\elemDom}{K}	

\newcommand{\viscosity}{\mu_{\mathrm{NS}}}
\newcommand{\resistivity}{\mu_{\mathrm{R}}}						
\newcommand{\avg}[1]{\left\{\hspace*{-3pt}\left\{#1\right\}\hspace*{-3pt}\right\}}
\newcommand{\jump}[1]{\ensuremath{\left\llbracket #1 \right\rrbracket}}
\newcommand{\jumpS}[1]{\ensuremath{\left[\hspace*{-1pt}\left\llbracket #1 \right\rrbracket\hspace*{-1pt}\right]}}
\newcommand{\jumpR}[1]{\langle\hspace*{-2pt}\langle#1\rangle\hspace*{-2pt}\rangle}

\newtheorem{lemma}{Lemma}

\newtheorem{definition}{Definition}

\begin{document}

\let\WriteBookmarks\relax
\def\floatpagepagefraction{1}
\def\textpagefraction{.001}
\shorttitle{Entropy Stable DGSEM for MHD. Part II: Shock Capturing}
\shortauthors{Rueda-Ramírez et~al.}


\title [mode = title]{An Entropy Stable Nodal Discontinuous Galerkin Method for the resistive MHD Equations. Part II: Subcell Finite Volume Shock Capturing}

\author[1]{Andrés M. Rueda-Ramírez}[
                        orcid=0000-0001-6557-9162]
\cormark[1]
\ead{aruedara@uni-koeln.de}

\credit{Conceptualization, Methodology, Software, Validation, Formal analysis, Data Curation, Writing - Original Draft, Visualization}

\author[2]{Sebastian Hennemann}
\credit{Conceptualization, Methodology, Formal analysis, Writing - Original Draft}

\author[3]{Florian J. Hindenlang}
[orcid=0000-0002-0439-249X]
\credit{Methodology, Formal analysis, Writing - Original Draft}

\author[4]{Andrew R. Winters}
[orcid=0000-0002-5902-1522]
\credit{Methodology, Formal analysis, Writing - Original Draft}


\author[1]{Gregor J. Gassner}
[orcid=0000-0002-1752-1158]
\credit{Conceptualization, Methodology, Validation, Formal analysis, Data Curation, Writing - Original Draft}

\address[1]{Department of Mathematics and Computer Science, University of Cologne, Weyertal 86-90, 50931 Cologne, Germany}

\address[2]{German Aerospace Center (DLR), Linder Höhe, 51147 Cologne, Germany}

\address[3]{Max Planck Institute for Plasma Physics, Boltzmannstraße 2, 85748 Garching, Germany}

\address[4]{Department of Mathematics, Computational Mathematics, Linköping University, 581 83, Linköping, Sweden}

\cortext[cor1]{Corresponding author}

\begin{abstract}
The second paper of this series presents two robust entropy stable shock-capturing methods for discontinuous Galerkin spectral element (DGSEM) discretizations of the compressible magneto-hydrodynamics (MHD) equations.
Specifically, we use the resistive GLM-MHD equations, which include a divergence cleaning mechanism that is based on a generalized Lagrange multiplier (GLM).
For the continuous entropy analysis to hold, and due to the divergence-free constraint on the magnetic field, the GLM-MHD system requires the use of non-conservative terms, which need special treatment.

Hennemann et al. [\textit{"A provably entropy stable subcell shock capturing approach for high order split form DG for the compressible Euler equations"}. JCP, 2020] recently presented an entropy stable shock-capturing strategy for DGSEM discretizations of the Euler equations that blends the DGSEM scheme with a subcell first-order finite volume (FV) method.
Our first contribution is the extension of the method of Hennemann et al. to systems with non-conservative terms, such as the GLM-MHD equations.
In our approach, the advective and non-conservative terms of the equations are discretized with a hybrid FV/DGSEM scheme, whereas the visco-resistive terms are discretized only with the high-order DGSEM method.
We prove that the extended method is entropy stable on three-dimensional unstructured curvilinear meshes.
Our second contribution is the derivation and analysis of a second entropy stable shock-capturing method that provides enhanced resolution by using a subcell reconstruction procedure that is carefully built to ensure entropy stability.

We provide a numerical verification of the properties of the hybrid FV/DGSEM schemes on curvilinear meshes and show their robustness and accuracy with common benchmark cases, such as the Orszag-Tang vortex and the GEM (Geospace Environmental Modeling) reconnection challenge.
Finally, we simulate a space physics application: the interaction of Jupiter's magnetic field with the plasma torus generated by the moon Io.
\end{abstract}


\begin{highlights}
\item The entropy stable FV subcell shock-capturing method for the DGSEM is extended to compressible magnetohydrodynamics
\item An enhanced entropy stable higher-resolution FV subcell shock-capturing method is presented
\item The shock-capturing methods are validated and used for space physics applications
\end{highlights}


\maketitle


\section{Introduction}

The resistive magnetohydrodynamics (MHD) equations are of interest for instance in plasma physics, space physics, and geophysics, as they find applications in all those areas.
They describe the evolution of the mass, momentum, energy, and magnetic field of electrically conducting compressible fluids with a mixed hyperbolic-parabolic character that depends on the viscous and resistive properties of the medium.
The MHD equations have two important physical constraints that are not explicitly built into the partial differential equation (PDE).
The first one is the divergence-free condition on the magnetic field, $\Nabla \cdot \vec{B} = 0$, which rules out the existence of magnetic monopoles.
The second physical constraint is the second law of thermodynamics, which states that the thermodynamic entropy of a closed system can only increase or remain constant in time.
The entropy of an MHD system can only remain constant in the absence of diffusive effects if the solution is continuous.
In the presence of discontinuities, such as shocks, or viscous/resistive effects, the thermodynamic entropy increases over time.

Discontinuous Galerkin (DG) methods are a family of numerical schemes that offer an interesting and straight-forward way to construct discretizations with arbitrarily high-order accuracy by projecting the solution into high-order polynomial spaces.
High-order DG methods have made their way into the mainstream of Computational Fluid Dynamics because they are very robust when dealing with advection-dominated problems.
Furthermore, DG methods provide a compact stencil and hence a local character, a feature that makes them highly parallelizable and flexible for complex 3D unstructured grids \cite{Wang2013High,Cockburn2000,Hindenlang2012}.
Moreover, high-order DG methods offer flexibility to perform $h/p$ adaptation  \cite{Riviere2008,Kopriva2002,RuedaRamirez2019a,
RuedaRamirez2019}.

There exist two main stability issues in high-order discontinuous Galerkin methods.
The first one is the appearance of aliasing-driven instabilities due to insufficient integration, which may cause the simulations to crash, especially in very under-resolved flow fields (e.g. at high Reynolds numbers).
The second one is related to the emergence of spurious oscillations due to the Gibbs phenomenon when very steep gradients, or even discontinuities, are approximated with high-order polynomials.

In the first paper of this series, \citet{Bohm2018} presented a DGSEM discretization on Gauss-Lobatto points of the resistive MHD equations that takes care of the aliasing-driven instabilities with the use of a flux differencing representation of the fluxes and non-conservative terms, which has a dealiasing effect that stabilizes the numerical solution \cite{Winters2018}.
The flux differencing representation of the fluxes (also called split form since in some cases it corresponds to a split formulation of the advective PDE fluxes)  is possible since the DGSEM on Gauss-Lobatto nodes fulfills the summation-by-parts property
\cite{Fisher2013,Fisher2013a,Gassner2013,Carpenter2014}.
With a careful selection of the numerical fluxes, the split-form DGSEM scheme becomes provably entropy stable, i.e. consistent with the second law of thermodynamics, which provides additional nonlinear stability.
\citet{Bohm2018} complete their scheme using a divergence cleaning mechanism proposed by \citet{Derigs2018} that is based on a generalized Lagrange multiplier.

Even though entropy stability provides the high-order DGSEM discretization with enhanced robustness, it is insufficient to obtain strict nonlinear stability as the entropy analysis assumes positive density and pressure.
It has been observed that in the presence of strong discontinuities (e.g. shocks) the oscillations of the high-order polynomials can break the positivity condition and cause the scheme to crash.
Among other strategies, the oscillations can be controlled by adding artificial viscosity/resistivity to the scheme, or by using a subcell discretization that is more robust than the DG scheme.

Artificial dissipation smears out the shock fronts, and hence reduces the slope of the numerical solution, such that high-order DG polynomials can represent it.
Artificial dissipation is commonly applied in an element-local manner with the use of troubled cell indicators \cite{Persson2006,Klockner2011}.
An interesting approach developed by \citet{Fernandez2018} uses physics-based sensors to apply the artificial dissipation more locally for Navier-Stokes simulations.
This approach was extended to the compressible MHD equations by \citet{Ciuca2020}.
In fully hyperbolic problems, the addition of artificial diffusion operators can reduce the computational performance and requires the introduction of new nonphysical boundary conditions.
Besides, in the presence of very strong shocks, the artificial dissipation needed to stabilize the numerical solution might be high enough to significantly reduce the time-step size of time-explicit simulations.

To avoid the drawbacks of artificial dissipation techniques, Sonntag and Munz \cite{Sonntag2014,Sonntag2017,Sonntag2017a} proposed an interesting subcell FV shock-capturing method for DG discretizations of the Navier-Stokes equations.
The scheme relies on a troubled cell indicator and uses a hard switch to replace problematic DG elements with more robust FV subcells.
This approach was later extended to MHD by
\citet{Nunez-delaRosa2018}.
More general approaches have been suggested by \citet{Markert2020}, who proposed a continuous blending between DG schemes of different orders and a subcell FV method; and by 
\citet{Vilar2019}, who presented a so-called \textit{a posteriori} limitation procedure for DG that uses an underlying subcell FV scheme to control the monotonicity and positivity of the numerical solution.
More recently, subcell FV shock-capturing methods that satisfy entropy inequalities have been presented \cite{Hennemann2020,Pazner2020,Liu2018}.
An interesting approach, proposed by \citet{Hennemann2020} for the compressible Euler equations, retains entropy stability when blending a split-form DGSEM discretization with a co-located FV discretization in an element-local manner. 


In this paper, we take the entropy stable DGSEM discretization of the resistive GLM-MHD equations by \citet{Bohm2018} and construct two different entropy stable subcell schemes to make it robust to handle shocks.
The first contribution of this paper is the extension of the subcell approach of \citet{Hennemann2020} to the GLM-MHD equations, where we blend the high-order DGSEM discretization of the advective and non-conservative terms with a first-order subcell FV method that we modify from \cite{Derigs2018}, and use the high-order DGSEM scheme for the diffusive terms.
Our second contribution is the derivation of a subcell reconstruction procedure that enhances the resolution of the subcell FV shock-capturing method while retaining entropy stability.
We verify our methods, test them against typical benchmark cases, and compute a problem of space physics: the interaction of Jupiter's magnetosphere with the plasma torus generated by one of its moons, Io.

Although in this paper we derive discretization methods for the compressible GLM-MHD equations, the theory presented here is applicable to any hyperbolic-parabolic symmetrizable system with non-conservative terms.

The paper is organized as follows. 
In Section \ref{sec:NotationPhysics}, we briefly describe the notation and introduce the GLM-MHD system.
In section \ref{sec:HennemannGoesMHD} we extend the entropy stable subcell FV shock-capturing method of \citet{Hennemann2020} to systems of PDEs with non-conservative terms, such as the GLM-MHD equations. 
Next, in section \ref{sec:TVD-ES}, we derive the entropy stable shock-capturing scheme with enhanced resolution that is obtained with a subcell reconstruction procedure.
In Section \ref{sec:Indicator}, we describe the shock indicator that is used to determine where and in which amount the FV stabilization is added.
Finally, the numerical verification and validation of the methods is presented in Section \ref{sec:Results}.

\section{Notation and Governing Equations} \label{sec:NotationPhysics}

\subsection{Notation} \label{sec:notation}
We adopt the notation of \cite{Bohm2018,Gassner2018,RuedaRamirez2020} to work with vectors of different nature. 
Spatial vectors are noted with an arrow (e.g. $\vec{x}=(x,y,z) \in \mathbb{R}^3$), state vectors are noted in bold (e.g. $\state{u}=(\rho, \rho \vec{v}, \rho E, \vec{B}, \psi)^T$), and block vectors, which contain a state vector in every spatial direction, are noted as
\begin{equation}
\blocktensor{f} =
\begin{bmatrix}
\mathbf{f}_1 \\ 
\mathbf{f}_2 \\
\mathbf{f}_3
\end{bmatrix} =
\mathbf{f}_1 \hat{\imath} + \mathbf{f}_2 \, \hat{\jmath} + \mathbf{f}_3 \hat{k}.
\end{equation}

The gradient of a state vector is a block vector,
\begin{equation}
\Nabla \mathbf{q} =
\begin{bmatrix}
\partial_x \mathbf{q} \\
\partial_y \mathbf{q} \\
\partial_z \mathbf{q} 
\end{bmatrix}
=
\partial_x \mathbf{q} \, \hat{\imath} + 
\partial_y \mathbf{q} \, \hat{\jmath} +
\partial_z \mathbf{q} \hat{k},
\end{equation}
and the gradient of a spatial vector is defined as the transpose of the outer product with the del operator,
\begin{equation}
\Nabla \vec{v} := 
\left( \Nabla \otimes \vec{v} \right)^T = 
\left( \Nabla \vec{v}^T \right)^T = 
\begin{bmatrix}
\bigpartialderiv{v_1}{x} & \bigpartialderiv{v_1}{y} & \bigpartialderiv{v_1}{z} \\
\bigpartialderiv{v_2}{x} & \bigpartialderiv{v_2}{y} & \bigpartialderiv{v_2}{z} \\
\bigpartialderiv{v_3}{x} & \bigpartialderiv{v_3}{y} & \bigpartialderiv{v_3}{z}
\end{bmatrix}.
\end{equation}

We define the notation for the jump operator, arithmetic and logarithmic means between a left and right state, $a_L$ and $a_R$, as
\begin{equation}
\jump{a}_{(L,R)} := a_R-a_L, 
~~~~~~~~~ 
\avg{a}_{(L,R)} := \frac{1}{2}(a_L+a_R), 
~~~~~~~~ 
a^{\ln}_{(L,R)} := \jump{a}_{(L,R)}/\jump{\ln(a)}_{(L,R)}.
\label{means}
\end{equation}
A numerically stable procedure to evaluate the logarithmic mean is given in \cite{Ismail2009}.
Note that the jump operator defined here is not symmetric. For convenience, we define a symmetric jump operator that assumes ordered sub-indexes ($L$,$R$) as
\begin{equation} \label{eq:SymmetricJump}
\jumpS{a}_{(L,R)} := 
\begin{cases}
\jump{a}_{(L,R)} & \text{if } R \ge L, \\
\jump{a}_{(R,L)} & \text{otherwise.} 
\end{cases}
\end{equation}

\subsection{The Resistive GLM-MHD Equations} \label{sec:GLM-MHD}

\subsubsection{The System of Equations}

In this work, we use the variant of the resistive GLM-MHD equations that is consistent with the continuous entropy analysis of \citet{Derigs2018}.
The system of equations that governs the motion of compressible, visco-resistive plasmas reads 
\begin{equation} \label{eq:GLM-MHD}
\partial_t \mathbf{u} 
+ \Nabla \cdot \blocktensor{f}^a (\mathbf{u}) 
- \Nabla \cdot \blocktensor{f}^{\nu} (\mathbf{u}, \Nabla \mathbf{u})
+ \noncon(\mathbf{u}, \Nabla \mathbf{u})
= \state{0},
\end{equation}
with the state vector $\state{u} = (\rho, \rho \vec{v}, \rho E, \vec{B}, \psi)^T$, the advective flux $\blocktensor{f}^a$, the viscous flux $\blocktensor{f}^{\nu}$, the non-conservative term $\noncon$.
Here, $\rho$ is the density, $\vec{v} = (v_1, v_2, v_3)^T$ is the velocity, $E$ is the specific total energy, $\vec{B} = (B_1, B_2, B_3)^T$ is the magnetic field, and $\psi$ is the so-called \textit{divergence-correcting field}, a generalized Lagrange multiplier (GLM) that is added to the original MHD system to minimize the magnetic field divergence. 
These equations do not enforce the divergence-free condition exactly, $\Nabla \cdot \vec{B} = 0$, but they evolve towards a divergence-free state. 
For details see \cite{Munz2000,Dedner2002,Derigs2018}.

The advective flux contains the Euler, ideal MHD and GLM contributions,
\begin{equation}
\blocktensor{f}^a(\mathbf{u}) = \blocktensor{f}^{a,\supEuler} +\blocktensor{f}^{a,\supMHD}+\blocktensor{f}^{a,\supGLM}=
\begin{pmatrix} 
\rho \vec{v} \\[0.15cm]
\rho (\vec{v}\, \vec{v}^{\,T}) + p\threeMatrix{I} \\[0.15cm]
\vec{v}\left(\frac{1}{2}\rho \left\|\vec{v}\right\|^2 + \frac{\gamma p}{\gamma -1}\right)  \\[0.15cm]
\threeMatrix{0}\\ \vec{0}\\[0.15cm]
\end{pmatrix} +
\begin{pmatrix} 
\vec{0} \\[0.15cm]
\frac{1}{2 \mu_0} \|\vec{B}\|^2 \threeMatrix{I} - \frac{1}{\mu_0} \vec{B} \vec{B}^T \\[0.15cm]
\frac{1}{\mu_0} \left( \vec{v}\,\|\vec{B}\|^2 - \vec{B}\left(\vec{v}\cdot\vec{B}\right) \right) \\[0.15cm]
\vec{v}\,\vec{B}^T - \vec{B}\,\vec{v}^{\,T} \\ \vec{0}\\[0.15cm]
\end{pmatrix} +
\begin{pmatrix} 
\vec{0} \\[0.15cm]
\threeMatrix{0} \\[0.15cm]
\frac{c_h}{\mu_0} \psi \vec{B} \\[0.15cm]
c_h \psi \threeMatrix{I} \\ c_h \vec{B}\\[0.15cm]
\end{pmatrix},
\label{eq:advective_fluxes}
\end{equation}
where $p$ is the gas pressure, $\threeMatrix{I}$ is the $3\times 3$ identity matrix, $\mu_0$ is the permeability of the medium, and $c_h$ is the \textit{hyperbolic divergence cleaning speed}.
The visco-resistive flux is defined as
\begin{equation}
\blocktensor{f}^{\nu}(\mathbf{u},\Nabla \mathbf{u})  = 
\begin{pmatrix} 
\vec{0} \\[0.15cm] 
\threeMatrix{\tau} \\[0.15cm]
\threeMatrix{\tau}\vec{v} -\Nabla q - 
\frac{\resistivity}{\mu_0^2} \left( (\Nabla \times \vec{B} ) \times \vec{B}\right) \\[0.15cm]
\frac{\resistivity}{\mu_0} \left( (\Nabla \vec{B} )^T - \Nabla \vec{B}\right)\\[0.15cm]
\vec{0} \\[0.15cm]
\end{pmatrix}\,,
\label{eq:viscous_fluxes}
\end{equation}
where the viscous stress tensor reads
\begin{equation}
\threeMatrix{\tau} = \viscosity ((\Nabla \vec{v}\,)^T + \Nabla \vec{v}\,) - \frac{2}{3} \viscosity(\Nabla \cdot \vec{v}\,) \threeMatrix{I},
\label{stress}
\end{equation}
and the heat flux is defined as
\begin{equation}
\vec{\nabla} q = -\kappa \vec{\nabla} \left(\frac{p}{R \rho}\right).
\label{heat}
\end{equation}
The new constants, $\viscosity,\resistivity,\kappa,R \ge 0$, are the viscosity, resistivity of the plasma, thermal conductivity, and the universal gas constant, respectively.
In the case of vanishing viscosity, resistivity and conductivity, $\viscosity = \resistivity = \kappa = 0$, we recover the ideal GLM-MHD equations with $\blocktensor{f}^{\nu}=\mathbf{0}$.

We close the system with the (GLM) calorically perfect gas assumption,
\begin{equation}
p = (\gamma-1)\left(\rho  E - \frac{1}{2}\rho\left\|\vec{v}\right\|^2 - \frac{1}{2 \mu_0}\|\vec{B}\|^2 - \frac{1}{2 \mu_0}\psi^2\right),
\label{eqofstate}
\end{equation}
where $\gamma$ denotes the heat capacity ratio, and we compute the thermal conductivity supposing that the plasma has a constant Prandtl number (Pr),
\begin{equation}
\kappa = \frac{\gamma \viscosity R} {(\gamma - 1) \text{Pr}}.
\end{equation}  

The non-conservative term has two main components, $\noncon = \noncon^\supMHD + \noncon^\supGLM$, with
\begin{align}
\noncon^\supMHD &= (\Nabla \cdot \vec{B}) \phiMHD =  \left(\Nabla \cdot \vec{B}\right) 
\left( 0 \,,\, \mu_0^{-1} \vec{B} \,,\, \mu_0^{-1} \vec{v}\cdot\vec{B} \,,\,  \vec{v} \,    ,\, 0 \right)^T \,, \label{Powell}\\ 
\noncon^\supGLM &= \phiGLM \cdot \Nabla \psi \quad =   \phiGLMs_1 \,\frac{\partial \psi}{\partial x} + \phiGLMs_2 \frac{\partial \psi}{\partial y} + \phiGLMs_3 \frac{\partial \psi}{\partial z} \,,\label{NC_GLM}
\end{align}
where $\phiGLM$ is a block vector with
\begin{equation}\label{Galilean}
\phiGLMs_\ell = \mu_0^{-1} \left(0 \,,\, 0\,,\,0\,,\,0\,,\,  v_\ell \psi \,,\, 0\,,\,0\,,\,0\,,\, v_\ell \right)^T, \quad \ell = 1,2,3.
\end{equation} 
The first non-conservative term, $\noncon^\supMHD$, is the well-known Powell term \cite{Powell2001}, and the second non-conservative term, $\noncon^\supGLM$, results from Galilean invariance of the full GLM-MHD system \cite{Derigs2017}. 

 
We note that for a magnetic field with vanishing divergence, $\Nabla \cdot \vec{B} = 0$, i.e., in the continuous case, \eqref{eq:GLM-MHD} reduces to the visco-resistive MHD equations, which describe the conservation of mass, momentum, energy, and magnetic flux.

\subsubsection{Thermodynamic Properties of the System}

Making the physical assumption of positive density and pressure, $\rho,p>0$, we obtain a suitable, strictly convex entropy function for the ideal and the resistive GLM-MHD equations by dividing the thermodynamic entropy density by the constant $-(\gamma-1)$,
\begin{equation}
S(\state{u}) = - \frac{\rho s}{\gamma-1},
\label{entropy}
\end{equation}
where $S$ is our mathematical entropy and $s = \ln\left(p \rho^{-\gamma}\right)$ is the thermodynamic entropy.
From the entropy function, we define the entropy variables,
\begin{equation}
\entVar = \frac{\partial S}{\partial \state{u}} = \left(\frac{\gamma-s}{\gamma-1} - \beta \norm{\vec{v}}^2,~2\beta v_1,~2\beta v_2,~2\beta v_3,~-2\beta,~2\beta B_1,~2\beta B_2,~2\beta B_3,~2\beta \psi\right)^T,
\label{entrvars}
\end{equation}
with $\beta = \frac{\rho}{2p}$, a quantity that is proportional to the inverse temperature.

To analyze the thermodynamic properties of the MHD equations, let us first consider the homogeneous ideal GLM-MHD system, i.e., without the visco-resistive terms, 
%
\begin{equation} \label{eq:idealGLM-MHD}
\partial_t \mathbf{u} 
+ \Nabla \cdot \blocktensor{f}^a (\mathbf{u}) 
+ \noncon
= \state{0}.
\end{equation}
As was shown by \citet{Derigs2017}, if we contract \eqref{eq:idealGLM-MHD} with the entropy variables, we obtain the entropy conservation law if the solution is smooth,
\begin{equation}\label{eq:EntropyConservationLaw}
 \frac{\partial S}{\partial t} + \Nabla \cdot \vec{f}^{\,S} = 0,
\end{equation}
where $\vec{f}^{\,S} = \vec{v} S$ is the so-called entropy flux.

Furthermore, in the presence of discontinuities in the solution, and/or visco-resistive effects, the contraction of the resistive GLM-MHD equations with the entropy variables leads to an entropy inequality in the weak sense \cite{Bohm2018},
\begin{equation}\label{eq:EntropyInequalityWeak}
\int_{\Omega} \frac{\partial S}{\partial t} \d t 
+ \oint_{\partial \Omega} \left( \vec{f}^{S}
- \entVar^T \blocktensor{f}^{\nu}  \right) \cdot \vec{n} \d \sigma  \le 0,
\end{equation}
where the total mathematical entropy within any physical domain, $\Omega$, can only increase over time if it is transported into the domain through its boundaries,	$\partial \Omega$.
Equation \eqref{eq:EntropyInequalityWeak} is the mathematical description of the second law of thermodynamics.
We refer the reader to \cite{Bohm2018} for details about the derivation of \eqref{eq:EntropyInequalityWeak}.

Finally, we define the entropy flux potential to be
\begin{equation}\label{eq:entPotential}
\vec{\Psi} := \entVar^T \blocktensor{f}^a -\vec{f}^S +\theta \vec{B},
\end{equation}
where $\theta$ is the contraction of $\phiMHD$ from the Powell term \eqref{Powell} into entropy space,
\begin{equation}\label{eq:contractSource}
\theta = \entVar^T \phiMHD = 2\beta(\vec{v}\cdot\vec{B}). 
\end{equation}

\subsubsection{One-Dimensional MHD System}

To simplify the analysis of the GLM-MHD system and its numerical discretizations, we write a one-dimensional version of \eqref{eq:GLM-MHD},
\begin{equation} \label{eq:GLM-MHD_1D}
\bigpartialderiv {\mathbf{u}}{t}
+ \bigpartialderiv {\state{f}^a} {x}
- \bigpartialderiv {\state{f}^{\nu}} {x} 
+ \noncon_1
= \state{0},
\end{equation}
where the state variable is $\state{u} = (\rho, \rho \vec{v}, \rho E, \vec{B}, \psi)^T$, as before, the advective flux in $x$ is
\begin{equation}
\state{f}^a(\mathbf{u}) :=
\state{f}_1^{a,\supEuler} +\state{f}_1^{a,\supMHD}+\state{f}_1^{a,\supGLM}=
\begin{pmatrix} 
\rho v_1 \\[0.15cm]
\rho v_1^2 + p \\[0.15cm]
\rho v_1 v_2 \\[0.15cm]
\rho v_1 v_3 \\[0.15cm]
v_1 \left( \frac{1}{2} \rho \|\vec{v}\|^2 + \frac{\gamma p}{\gamma -1}\right) \\[0.15cm]
0 \\[0.15cm]
0 \\[0.15cm]
0 \\[0.15cm]
0 \\[0.15cm]
\end{pmatrix} +
\begin{pmatrix} 
0 \\[0.15cm]
\frac{1}{\mu_0} \left( \frac{1}{2} \|\vec{B}\|^2 - B_1 B_1 \right) \\[0.15cm]
- B_1 B_2 / \mu_0 \\[0.15cm]
- B_1 B_3 / \mu_0 \\[0.15cm]
\frac{1}{\mu_0} \left( v_1 \|\vec{B}\|^2 - B_1 \left(\vec{v}\cdot\vec{B}\right) \right) \\[0.15cm]
0 \\[0.15cm]
v_1 B_2 - v_2 B_1 \\[0.15cm]
v_1 B_3 - v_3 B_1 \\[0.15cm]
0 \\[0.15cm]
\end{pmatrix} +
\begin{pmatrix} 
0 \\[0.15cm]
0 \\[0.15cm]
0 \\[0.15cm]
0 \\[0.15cm]
\frac{c_h}{\mu_0} \psi B_1 \\[0.15cm]
c_h \psi  \\[0.15cm]
0 \\[0.15cm]
0 \\[0.15cm]
c_h B_1 \\[0.15cm]
\end{pmatrix},
\label{eq:advective_fluxes_1D}
\end{equation}
and the diffusive flux in $x$ reads
\begin{equation}
\state{f}^{\nu}\left(\mathbf{u},\bigpartialderiv {\state{u}}{x} \right)  = 
\begin{pmatrix} 
0 \\[0.15cm] 
\frac{4}{3} \viscosity \bigpartialderiv{v_1}{x}
\\[0.15cm]
\viscosity \bigpartialderiv{v_2}{x} 
\\[0.15cm]
\viscosity \bigpartialderiv{v_3}{x} 
\\[0.15cm] 
\viscosity \left(
\frac{4}{3} \bigpartialderiv{v_1}{x} v_1 + \bigpartialderiv{v_2}{x} v_2 + 
\bigpartialderiv{v_3}{x}  v_3
\right)
+ \bigpartialderiv{q}{x}
+ \frac{\resistivity}{\mu_0^2}
\left(
\bigpartialderiv{B_2}{x} B_2 + \bigpartialderiv{B_3}{x} B_3
\right)
 \\[0.15cm] 
0 \\[0.15cm] 
\frac{\resistivity}{\mu_0} \bigpartialderiv{B_2}{x} \\[0.15cm] 
\frac{\resistivity}{\mu_0} \bigpartialderiv{B_3}{x} \\[0.15cm] 
0 \\[0.15cm]
\end{pmatrix}\,.
\label{eq:viscous_fluxes_1D}
\end{equation}
Note that we removed the sub-index in the conservative fluxes to simplify the notation and improve the readability, $\state{f} \leftarrow \state{f}_1$, as this change does not produce ambiguity between the 1D and 3D notations.

The non-conservative term, $\noncon_1 = \noncon^{\supMHD}_1 + \noncon^{\supGLM}_1$, consists of the following two terms
\begin{equation}
\noncon^{\supMHD}_1 = \bigpartialderiv {B_1}{x} \phiMHD = \bigpartialderiv {B_1}{x}
\begin{pmatrix}
0 \\ \mu_0^{-1} B_1 \\ \mu_0^{-1} B_2 \\ \mu_0^{-1} B_3 \\ \mu_0^{-1} \vec{v} \cdot \vec{B} \\ v_1 \\ v_2 \\ v_3 \\  0
\end{pmatrix},
\ \ \ \ \
\noncon^{\supGLM}_1 = \bigpartialderiv {\psi}{x} \phiGLMs_1 = \bigpartialderiv {\psi}{x}
\begin{pmatrix}
0 \\ 0 \\ 0 \\ 0 \\ \mu_0^{-1} v_1 \psi \\ 0 \\ 0 \\ 0 \\ \mu_0^{-1} v_1
\end{pmatrix}.
\label{eq:NonCons_1D}
\end{equation}

Finally, the entropy flux potential in 1D is defined as
\begin{equation}\label{eq:entPotential_1D}
\Psi_1 := \entVar^T \state{f}^a - f^S +\theta B_1.
\end{equation}

For readability, in the remaining parts of this paper, we analyze and discretize the one-dimensional GLM-MHD equations, \eqref{eq:GLM-MHD_1D}. 
This can be done without loss of generality, as the spatial dimensions are decoupled in the GLM-MHD system \cite{Derigs2018}.
For completeness, however, we include relevant three-dimensional derivations in Appendix \ref{sec:3D_Derivations}.

\section{Entropy Stable FV/DGSEM Discretization} \label{sec:HennemannGoesMHD}

Following the approach developed by \citet{Hennemann2020}, we seek an entropy stable hybrid FV/DGSEM discretization of the GLM-MHD equations of the form
\begin{equation} \label{eq:BlendedScheme}
\Delta x_j \dot{\state{u}}_j = 
(1-\alpha) \state{F}_j^{a,\DG} 
+ \alpha \state{F}_j^{a,\FV} 
- \state{F}_j^{\nu,\DG},
~~~~~~
j=0, \ldots, N
\end{equation}
where $\Delta x_j$ is the FV subcell size in physical space at the degree of freedom $j$ of each element, $\dot{\state{u}}_j$ is the discrete time derivative  of the solution at the degree of freedom $j$, $\alpha \in [0,1]$ is an element-local blending coefficient, $\state{F}^{a,\FV}$ and $\state{F}^{a,\DG}$ are the discretizations of the advective and non-conservative terms with the FV method and the DGSEM, respectively, and $\state{F}^{\nu,\DG}$ is the discretization of the diffusive terms with the DGSEM.

Note that we propose a method that combines the advective and non-conservative terms of the low- and high-order methods, while we discretize the diffusive terms using only the high-order DGSEM.
This ansatz is valid from a numerical point of view, as the high gradients in the vicinity of a shock cause the diffusive terms to add an increased dissipation.
This additional dissipation contributes to the stabilization of the numerical solution.

The building blocks of \eqref{eq:BlendedScheme} are detailed in following sections.
First, we present the high-order DGSEM discretization of the visco-resistive GLM-MHD system in section \ref{sec:DGSEM}.
Next, we derive the first-order FV discretization of the ideal GLM-MHD system in Section \ref{sec:FV_GLM-MHD}.
Finally, in Section \ref{sec:EntropyStabilityFirstOrder} we show that our proposed hybrid FV/DGSEM is entropy stable.

All the derivations in this section are for the 1D GLM-MHD system for simplicity.
For completeness, however, we have included the derivations for the 3D GLM-MHD system on 3D curvilinear meshes in Appendix \ref{sec:3D_Derivations}.

\subsection{DGSEM Discretization of the Visco-Resistive GLM-MHD System} \label{sec:DGSEM}

\citet{Bohm2018} proposed an entropy stable DGSEM discretization of the resistive GLM-MHD equations.
To obtain it, we rewrite \eqref{eq:GLM-MHD_1D} as
\begin{equation} \label{eq:GLM-MHD_1D_homo}
\bigpartialderiv {\mathbf{u}}{t}
+ \bigpartialderiv {} {x} \state{f}^a (\state{u})
- \bigpartialderiv {} {x} 
\state{f}^{\nu} \left( \state{u}, \bigpartialderiv{\entVar}{x}  \right) 
+ \noncon_1 \left(\state{u},\Nabla \state{u} \right)
= 0,
\end{equation}
where the visco-resistive (diffusive) flux is rewritten to depend on the gradient of the entropy variables.
Following e.g. \citet{Bassi1997}, and \citet{Arnold2002}, \eqref{eq:GLM-MHD_1D_homo} can be rewritten as a first-order system,
\begin{align}  \label{eq:GLM-MHD_1D_homoSys1}
\begin{cases}
\bigpartialderiv {\mathbf{u}}{t}
&= 
- \bigpartialderiv {} {x} \state{f}^a (\state{u})
- \noncon_1 \left(\state{u},\Nabla \state{u} \right)
+ \bigpartialderiv {} {x}  \state{f}^{\nu} \left( \state{u}, \state{g}  \right),
\\
\bigpartialderiv{\entVar}{x} &= \state{g}. 
\end{cases}
\end{align}

To obtain the DGSEM-discretization of \eqref{eq:GLM-MHD_1D_homoSys1}, the simulation domain is tessellated into elements and all variables are approximated within each element by piece-wise Lagrange interpolating polynomials of degree $N$ on the Legendre-Gauss-Lobatto (LGL) nodes.
These polynomials are continuous in each element and may be discontinuous at the element interfaces.
Furthermore, \eqref{eq:GLM-MHD_1D_homoSys1} is multiplied by an arbitrary polynomial (test function) of degree $N$ and numerically integrated by parts inside each element of the mesh with an LGL quadrature rule of $N+1$ points on a reference element, $\xi \in [-1,1]$, to obtain
\begin{equation} \label{eq:DGSEM_MHD}
J \omega_j \dot{\state{u}}^{\DG}_j = \state{F}_j^{a,\DG} - \state{F}_j^{\nu,\DG},
\end{equation}
for each degree of freedom $j$ of each element. 
In \eqref{eq:DGSEM_MHD}, $\omega_j$ is the reference-space quadrature weight, $J$ is the geometry mapping Jacobian from reference space to physical space, which is constant within each element in the 1D discretization, $\state{F}_j^{a,\DG}$ is the discretization of the advective and non-conservative terms, and $\state{F}_j^{\nu,\DG}$ is the discretization of the diffusive term.

The advective and non-conservative terms are discretized using the split-form DGSEM.
The discretization for any element $\elemDom$ reads
%
\begin{align} \label{eq:DGSEMMHDadv2}
\state{F}_j^{a,\DG} =
& 
-2 \sum_{k=0}^N Q_{jk} \state{f}^{*}_{(j,k)} 
-  \sum_{k=0}^N Q_{jk} \Jan^{*}_{(j,k)} 
+ \delta_{jN} \left( \state{f}^a_N +\Jan_N \right)
- \delta_{j0} \left( \state{f}^a_0 + \Jan_0 \right)
& \bigg\rbrace &
{\state{F}_{\elemDom,j}^{a,\DG}}
\nonumber \\
&
- \delta_{jN} \left( \numfluxb{f}^a_{(N,R)} + \numnonconsD{\Jan}_{(N,R)} \right)
+ \delta_{j0} \left(\numfluxb{f}^a_{(0,L)} + \numnonconsD{\Jan}_{(0,L)} \right)
& \bigg\rbrace &
{\state{F}_{\partial \elemDom,j}^{a,\DG}},
\end{align}
where $Q_{jk}=\omega_j D_{jk}=\omega_j \ell'_k(\xi_j)$ is the SBP derivative matrix, defined in terms of the Lagrange interpolating polynomials, $\{ \ell_k \}_{k=0}^N$, $\state{f}^{*}_{(j,k)} = \state{f}^{*}(\state{u}_j,\state{u}_k)$ is the volume numerical two-point flux, $\numfluxb{f}^a_{(i,j)} = \numfluxb{f}^a(\state{u}_i,\state{u}_j)$ is the surface numerical flux, which accounts for the jumps of the solution across the cell interfaces, $\Jan^{*}_{(j,k)}$ is the volume numerical non-conservative term, and $\numnonconsD{\Jan}_{i,j}$ is the surface numerical non-conservative term.
$\state{F}_{\elemDom,j}^{a,\DG}$ gathers the terms that only depend on inner degrees of freedom, and $\state{F}_{\partial \elemDom,j}^{a,\DG}$ gathers the boundary terms that depend on outer and inner degrees of freedom.

We require the numerical fluxes to be conservative (i.e. symmetric),
\begin{equation} \label{eq:conservativeProp}
\numfluxb{f}^a (\state{u}_i,\state{u}_j) = \numfluxb{f}^a (\state{u}_j,\state{u}_i), 
~~~~
\state{f}^{*} (\state{u}_i,\state{u}_j) = \state{f}^{*} (\state{u}_j,\state{u}_i),
\end{equation}
and consistent,
\begin{equation} \label{eq:consistentProp}
\numfluxb{f}^a (\state{u}_j,\state{u}_j) = \state{f}^a ( \state{u}_j),
~~~~~~~
\state{f}^{*} (\state{u}_j,\state{u}_j) = \state{f}^a ( \state{u}_j).
\end{equation}

Note that the numerical non-conservative terms do not need to fulfill the symmetry property, \eqref{eq:conservativeProp}, as they are by definition non-conservative.
However, we require them to have the consistency property,
\begin{equation}
\numnonconsD{\Jan}_{(j,j)} = \Jan_j,
~~~~
{\Jan}^{*}_{(j,j)} = \Jan_j,
\end{equation}
where $\Jan := \phiMHD B_1 + \phiGLMs_1 \psi$.

The surface numerical non-conservative terms are defined as \cite{Derigs2018,Bohm2018}
\begin{align} \label{eq:DiamondFluxes}
\numnonconsD{\Jan}_{(j,j+1)} &= 
\numnonconsD{ \left( \phiMHD B_1 \right) }_{(j,j+1)} 
+
\numnonconsD{ \left( \phiGLMs_1 \psi \right) }_{(j,j+1)} 
\nonumber \\
&=
\avg{B_1}_{(j,j+1)} \phiMHD_j + \avg{\psi}_{(j,j+1)} \phiGLMs_{1,j}.
\end{align}
and the volume numerical non-conservative terms are defined as \cite{Bohm2018}
\begin{align} \label{eq:volNonCons_D1}
\Jan^{*}_{(j,k)} &= \noncon^{*\supMHD}_{(j,k)} + \noncon^{*\supGLM}_{(j,k)} \nonumber \\
&= \phiMHD_j B_{1,k} + \phiGLMs_{1,j} \psi_k.
\end{align}

Finally, the diffusive term is discretized using the standard DGSEM on Gauss-Lobatto nodes,
\begin{equation}
\state{F}_j^{\nu,\DG} = 
-\sum_{k=0}^N Q_{jk} \state{f}^{\nu}_{k} 
- \delta_{jN} \left( \numfluxb{f}^{\nu}_{(N,R)}-\state{f}^{\nu}_N \right)
+ \delta_{j0} \left( \numfluxb{f}^{\nu}_{(0,L)}-\state{f}^{\nu}_0 \right),
\end{equation}
where $\numfluxb{f}^{\nu}_{(L,R)}(\state{u}_L,\state{u}_R, \state{g}_L, \state{g}_R)$ is the diffusive numerical flux function, which fulfills the symmetry \eqref{eq:conservativeProp} and consistency \eqref{eq:consistentProp} properties.
Furthermore, the nodal values of the diffusive flux, $\state{f}^{\nu}_{k} = \state{f}^{\nu} (\state{u}_k,\state{g}_k)$, are evaluated with
\begin{equation} \label{eq:gradEqDGSEM}
J \omega_k \state{g}_k = 
\sum_{n=0}^N Q_{kn} \entVar_{n} 
+ \delta_{kN} \left( \numflux{\entVar}_{(N,R)}-\entVar_N \right)
- \delta_{k0} \left( \numflux{\entVar}_{(0,L)}-\entVar_0 \right),
\end{equation}
where $\numflux{\entVar}_{(L,R)}(\entVar_L,\entVar_R)$ is the numerical surface contribution of the entropy variables.
In this paper, we use method proposed by \citet{Bassi1997} (BR1) to compute $\numflux{\entVar}_{(L,R)}$ and $\numfluxb{f}^{\nu}_{(L,R)}$.
Note that the BR1 method preserves entropy stability for DGSEM discretizations of the Navier-Stokes \cite{Gassner2018} and the resistive GLM-MHD equations \cite{Bohm2018}, provided that the gradient equations use the entropy variables, as in \eqref{eq:gradEqDGSEM}.

\subsection{The Native LGL Finite Volume Discretization of the ideal GLM-MHD system} \label{sec:FV_GLM-MHD}

Following the strategy proposed by \citet{Hennemann2020}, we formulate a first-order Finite Volume method that can be seamlessly blended with the high-order DGSEM of Section \ref{sec:DGSEM}.
We call this method the \textit{Native LGL Finite Volume Approximation}, since the FV method uses a subcell grid that matches the LGL grid of the high-order DGSEM, where the subcell size is set as the quadrature weight times the mapping Jacobian, and the LGL nodal values are interpreted as subcell mean values. 
The subcell distribution within an element is illustrated in Figure \ref{fig:SubcellFV} for $N=5$.

\begin{figure}[htb]
\centering
\includegraphics[width=0.6\linewidth]{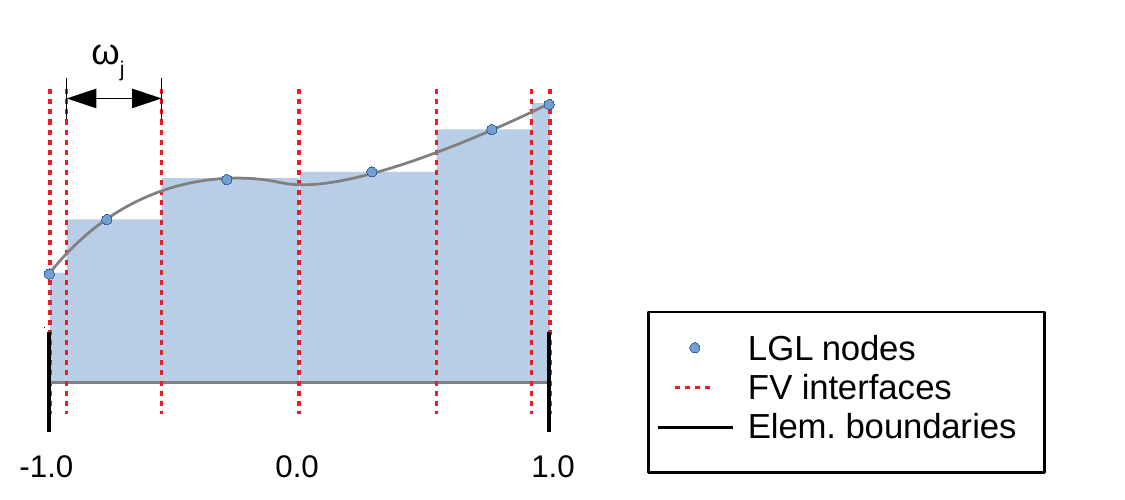}
\caption{Example of a DGSEM element with FV subcells ($N=5$). 
Note that we represent the element in reference space. 
In physical space, the subcell width is $J \omega_j$.}
\label{fig:SubcellFV}
\end{figure}

We use the entropy stable Finite Volume discretization of the ideal GLM-MHD system proposed by \citet{Derigs2018} in this non-uniform grid.
To do so, let us consider the ideal GLM-MHD system in one dimension,
\begin{equation} \label{eq:idalHomoGLM-MHD_1D}
\bigpartialderiv {\mathbf{u}}{t}
+ \bigpartialderiv {\state{f}^a} {x}
+ \noncon_1
= \state{0}.
\end{equation}

To obtain a first-order finite volume discretization of \eqref{eq:idalHomoGLM-MHD_1D}, we integrate over each subcell, and use integration by parts on the divergence term to obtain
%
%
\begin{equation} \label{eq:FVMHD}
J \omega_j \dot{\state{u}}^{\FV}_j
= \numfluxb{f}^a_{(j,j-1)} - \numfluxb{f}^a_{(j,j+1)}
- J \omega_j \noncon_{1,j} 
= \state{0},
\end{equation}
where the two-point numerical fluxes, $\numfluxb{f}^a_{(i,j)} = \numfluxb{f}^a(\state{u}_i,\state{u}_j)$ account for the jumps of the solution across the cell interfaces, and $\noncon_{1,j}$ is the discretization of the non-conservative terms in the cell $j$.

Following the strategy used by \citet{Chandrashekar2016}, and \citet{Derigs2018}, we discretize the non-conservative term, \eqref{eq:NonCons_1D}, using a central differencing scheme,
\begin{align} \label{eq:nonConsCentralDiff}
\noncon_{1,j} &=
\left( \noncon^{\supMHD}_1 + \noncon^{\supGLM}_1 \right)
\\
&\approx
\left( \frac{\avg{B_1}_{(j,j+1)} - \avg{B_1}_{(j-1,j)}} {J \omega_j}  \phiMHD_j +
\frac{\avg{\psi}_{(j,j+1)} - \avg{\psi}_{(j-1,j)}} {J \omega_j}  \phiGLMs_{1,j} \right).
\end{align}

Equation \eqref{eq:FVMHD} can be then rewritten as
\begin{equation} \label{eq:FVMHD_D}
J \omega_j \dot{\state{u}}^{\FV}_j  
= \numfluxb{f}^a_{(j,j-1)}
- \numfluxb{f}^a_{(j,j+1)}
+ \numnonconsD{\Jan}_{(j,j-1)}
- \numnonconsD{\Jan}_{(j,j+1)},
\end{equation}
where, in accordance with the central differencing discretization in \eqref{eq:nonConsCentralDiff}, the non-conservative terms are defined as in \eqref{eq:DiamondFluxes}.

The numerical fluxes and non-conservative terms on the element boundaries are evaluated with the left and right outer states,
\begin{align}
\numfluxb{f}^a_{(0,-1)} :=& \numfluxb{f}^a(\state{u}_0,\state{u}_L), &
\numnonconsD{\Jan}_{(0,-1)} :=& \numnonconsD{\Jan}(\state{u}_0,\state{u}_L),\\
\numfluxb{f}^a_{(N,N+1)} :=& \numfluxb{f}^a(\state{u}_N,\state{u}_R), &
\numnonconsD{\Jan}_{(N,N+1)} :=& \numnonconsD{\Jan}(\state{u}_N,\state{u}_R).
\end{align}
Therefore, it is possible to rewrite the native LGL FV discretization of the GLM-MHD system for any element $\elemDom$ as
\begin{align} \label{eq:FVMHD_BlendedForm}
\state{F}_j^{a,\FV} :=
J \omega_j \dot{\state{u}}^{\FV}_j  
=&  \delta_{j0} \left(\numfluxb{f}^a_{(0,L)} + \numnonconsD{\Jan}_{(0,L)} \right)
- \delta_{jN} \left(\numfluxb{f}^a_{(N,R)} + \numnonconsD{\Jan}_{(N,R)} \right)
& \bigg\rbrace &
{\state{F}_{\partial \elemDom,j}^{a,\FV}},
\nonumber\\
&
+ \left( 1-\delta_{j0} \right) 
\left( \numfluxb{f}^{a,\FV}_{(j,j-1)}
+ \numnonconsD{ \stateG{\Phi} }_{(j,j-1)}
\right)
- \left( 1-\delta_{jN} \right)
\left( \numfluxb{f}^{a,\FV}_{(j,j+1)}
+ \numnonconsD{ \stateG{\Phi} }_{(j,j+1)}
\right),
& \bigg\rbrace &
{\state{F}_{\elemDom,j}^{a,\FV}},
\end{align}
where we now allow using a different numerical flux function for the element boundaries, $\numfluxb{f}^{a}$, and for the subcell interfaces that lie within the element, $\numfluxb{f}^{a,\FV}$.


\subsection{Entropy Stability} \label{sec:EntropyStabilityFirstOrder}

In this section, we derive the entropy balance of the FV, DGSEM, and hybrid FV/DGSEM discretizations.
For simplicity, we start with the FV discretization.

\subsubsection{Entropy Balance of the Native LGL FV Discretization}

The numerical scheme is said to be entropy conservative \textit{semi-discretely} if it translates into a semi-discrete entropy conservation law when contracted with the entropy variables.
For instance, if we contract \eqref{eq:FVMHD_D} with the entropy variables on the left, 
\begin{equation} \label{eq:contractFVMHD_S}
J \omega_j \entVar_j^T \dot{\state{u}}_j  
= \entVar_j^T  \left(
  \numfluxb{f}^a_{(j,j-1)}
- \numfluxb{f}^a_{(j,j+1)}
+ \numnonconsD{ \stateG{\Phi} }_{(j,j-1)}
- \numnonconsD{ \stateG{\Phi} }_{(j,j+1)}
\right),
\end{equation}
we expect to obtain a semi-discrete entropy conservation law,
\begin{equation} \label{eq:FV_EClaw}
J \omega_j \dot{S}_j  
= \numflux{f}^S_{(j,j-1)}
- \numflux{f}^S_{(j,j+1)},
\end{equation}
where the numerical entropy flux, $\numflux{f}^S_{(i,j)}$, must fulfill the conservative \eqref{eq:conservativeProp} and consistency \eqref{eq:consistentProp} properties.

In the discretization of systems of conservation laws, semi-discrete entropy conservation can be enforced using Tadmor's condition for entropy conserving schemes \cite{tadmor1986minimum,tadmor1983entropy,Tadmor2003},
\begin{equation}
\jump{\entVar}_{(j,k)}^T \numfluxb{f}^a_{(j,k)} = \jump{\Psi}_{(j,k)}.
\end{equation}
However, since we are dealing with a system that has non-conservative terms, we need a generalized Tadmor's condition for entropy conserving schemes \cite{Bohm2018,Liu2018,Renac2019,Manzanero2020}, which can be written using the numerical non-conservative terms as
\begin{equation} \label{eq:TadmorNonCons}
\jump{\entVar}_{(j,k)}^T 
\numfluxb{f}^a_{(j,k)}
+ \entVar^T_{k} \numnonconsD{\stateG{\Phi}}_{(k,j)}
- \entVar^T_{j}   \numnonconsD{\stateG{\Phi}}_{(j,k)}
= \jump{\Psi}_{(j,k)}.
\end{equation}

We can fulfill Tadmor's generalized condition, \eqref{eq:TadmorNonCons}, with a correct combination of numerical non-conservative terms and numerical fluxes.
For instance, for this choice of the numerical non-conservative terms, \eqref{eq:DiamondFluxes}, \citet{Derigs2018} proposed an EC flux, which we detail in Appendix \ref{sec:ECflux}.
Using $\numfluxb{f}^a (\state{u}_L,\state{u}_R) = \state{f}_1^{\ec}(\state{u}_L,\state{u}_R)$, the first-order FV scheme is semi-discretely entropy conservative by construction and is virtually dissipation free.

Additional dissipation can be added to the scheme using the entropy conservative two-point flux and Lax-Friedrichs type dissipation,
\begin{equation} \label{eq:EC_LFdiss}
\numfluxb{f}^a (\state{u}_L,\state{u}_R) =  \state{f}_1^{\ec}(\state{u}_L,\state{u}_R) 
- \frac{1}{2} \mat{\mathcal{D}} \jumpS{\state{u}}_{(L,R)},
\end{equation}
where the use of the symmetric jump operator, \eqref{eq:SymmetricJump}, guarantees the fulfillment of the conservative property, \eqref{eq:conservativeProp}.
In this work, we use dissipation matrices of \textit{Roe-type},
\begin{equation}
\mat{\mathcal{D}} = \mat{R}~|\mat{\Lambda}| ~\mat{R}^{-1},
\end{equation}
where $\mat{R}$ is the matrix of right eigenvectors evaluated on a mean state, and $\mat{\Lambda}$ is a diagonal matrix with the eigenvalues of the flux.
Note that the Lax-Friedrichs and the Rusanov schemes can be written as a Roe-type operator, where $\mat{\Lambda}$ has the maximum eigenvalue in all its diagonal entries.

In accordance with the generalized Tadmor's condition for entropy conserving schemes, \eqref{eq:TadmorNonCons}, we define the numerical entropy flux and the entropy production.

\begin{definition}[Numerical entropy flux]
The numerical entropy flux from the degree of freedom $j$ to $k$ is defined as
\begin{equation} \label{eq:numEntFlux} 
\numflux{f}^S_{(j,k)} = 
\avg{\entVar}_{(j,k)}^T \numfluxb{f}^a_{(j,k)} 
+ \frac{1}{2} \entVar^T_j \numnonconsD{\stateG{\Phi}}_{(j,k)}
+ \frac{1}{2} \entVar^T_{k} \numnonconsD{\stateG{\Phi}}_{(k,j)}
- \avg{\Psi}_{(j,k)},
\end{equation}
which clearly fulfills the symmetric conservative property, \eqref{eq:conservativeProp}.
\end{definition}

\begin{definition}[Entropy production]
The entropy production on an interface between the degrees of freedom $j$ and $k$ is defined as
\begin{equation} \label{eq:EntropyDissipation}
r_{(j,k)} =
\jump{\entVar}_{(j,k)}^T 
\numfluxb{f}^a_{(j,k)}
+ \entVar^T_{k} \numnonconsD{\stateG{\Phi}}_{(k,j)}
- \entVar^T_{j}   \numnonconsD{\stateG{\Phi}}_{(j,k)}
- \jump{\Psi}_{(j,k)}.
\end{equation}
\end{definition}

Note that a scheme with zero entropy production fulfills \eqref{eq:TadmorNonCons}.
We are now ready to analyze the entropy behavior of the FV scheme that uses the selected numerical non-conservative terms and any numerical flux function.

\begin{lemma} \label{lemma:EntropyFV}
The semi-discrete entropy balance of the first-order native LGL FV discretization of the ideal and homogeneous GLM-MHD equations, \eqref{eq:FVMHD_D}, for each cell reads
\begin{equation} \label{eq:FV_ESinequality}
J \omega_j \dot{S}_j  
= \numflux{f}^S_{(j,j-1)}
- \numflux{f}^S_{(j,j+1)}
+ \frac{1}{2} \left( r_{(j-1,j)} + r_{(j,j+1)} \right).
\end{equation}
\end{lemma}

\begin{proof}
The proof of \eqref{eq:FV_ESinequality} is given in \cite{Derigs2018}.
However, for completeness, we include the proof consistent with the current notation and formulations in Appendix \ref{sec:Proof_FV_1D}.
\end{proof}


The original Lax-Friedrichs type dissipation, \eqref{eq:EC_LFdiss}, is not entropy stable for the MHD equations \cite{Derigs2017}.
We can obtain entropy stability if we approximate the jump of the state quantities with the jump of the entropy variables, $\jump{\state{u}}_{(j,k)} \approx \mat{\mathcal{H}} \jump{\entVar}_{(j,k)}$, where $\mat{\mathcal{H}} = \partial \state{u}/ \partial \entVar$ on a mean state.
We rewrite the entropy stable flux from \eqref{eq:EC_LFdiss} as
\begin{equation} \label{eq:EC_LFdiss2}
\state{f}_1^{\es}(\state{u}_L,\state{u}_R) := \state{f}_1^{\ec}(\state{u}_L,\state{u}_R) 
- \frac{1}{2} 
\underbrace{\mat{R}~ |\mat{\Lambda}| ~\mat{R}^{-1}~ \mat{\mathcal{H}} }_{\mat{\mathcal{D}}^{\entVar}}
\jumpS{\entVar}_{(L,R)}.
\end{equation}

The entropy production of the flux \eqref{eq:EC_LFdiss2} between the degrees of freedom $j$ and $k$ ($j<k$) can be computed according to \eqref{eq:EntropyDissipation} as
\begin{equation} \label{eq:LF_Diss2}
r_{(j,k)} = - \frac{1}{2} \jump{\entVar}_{(j,k)}^T \mat{\mathcal{D}}^{\entVar} \jump{\entVar}_{(j,k)},
\end{equation}
which is always negative if $\mat{\mathcal{D}}^{\entVar}$ is a symmetric positive definite matrix.

\citet{barth1999numerical} showed that there exists a diagonal matrix, $\mat{\mathcal{Z}}>0$, that relates the eigenvector matrix, $\mat{R}$, to the entropy Jacobian, $\mat{\mathcal{H}}$, such that 
\begin{equation}
\mat{\mathcal{H}} = \mat{R} ~\mat{\mathcal{Z}}~ \mat{R}^T.
\end{equation}
As a result, the dissipation matrix yields
\begin{align}
\mat{\mathcal{D}}^{\entVar} =  \mat{R}~|\mat{\Lambda}| ~ \mat{\mathcal{Z}}~ \mat{R}^T,
\end{align}
which is clearly symmetric positive definite by construction.

The derivation of the matrices $\mat{R}$ and $\mat{\mathcal{Z}}$ is not trivial.
We refer the reader to \cite{Derigs2018} for the derivation of the matrices for the GLM-MHD equations, and to \cite{Winters2017} for the derivations for the MHD equations.
Note that the eigensystem of the compressible MHD equations exhibits degeneracies \cite{Roe1996}.

\subsubsection{Entropy Balance of the High-Order DGSEM}

\begin{lemma} \label{lemma:EntropyDGSEM}
The semi-discrete entropy balance of the DGSEM discretization of the GLM-MHD equations, \eqref{eq:DGSEM_MHD}, integrating over an entire element, reads
\begin{equation} \label{eq:EntropyDGSEMLemma}
\sum_{j=0}^N \omega_j J \dot{S}_j = 
\underbrace{
\numflux{f}^S_{(0,L)} -\numflux{f}^S_{(N,R)} + \frac{1}{2} \left( r_{(L,0)} + r_{(N,R)} \right) + \sum_{j,k=0}^N Q_{jk} r_{(j,k)} 
}_{\dot{S}^a}
+ r^{\nu},
\end{equation}
where the numerical entropy flux and the entropy production are consistent with the FV definitions,  \eqref{eq:numEntFlux} and \eqref{eq:EntropyDissipation}, respectively, $\dot{S}^a$ gathers the entropy flux and production of the advective and non-conservative terms, and $r^{\nu}$ is the entropy production due to the diffusive terms.
\end{lemma}

\begin{proof}
The proof of Lemma \ref{lemma:EntropyDGSEM} can be found in \cite{Bohm2018}.
In Appendix \ref{sec:Proof_DGSEM_1D}, we summarize the proof in our own notation for the advective and non-conservative terms of the PDE, as these are the relevant terms for the present study.
The reader is referred to \cite{Bohm2018} for the proof that the diffusive terms of the semi-discrete system reduce to $r^{\nu} \le 0$ when contracted with the entropy variables and integrated over each element.

\end{proof}

As a consequence of Lemma \ref{lemma:EntropyDGSEM}, we can control the entropy behavior of the DGSEM discretization by selecting the volume and surface numerical fluxes.
If an entropy conserving flux is used for both the volume and surface numerical fluxes, the scheme is provably entropy conserving in its advective terms.

To obtain an entropy stable scheme, we can use an entropy conserving flux for the volume numerical fluxes and an entropy stable flux, such as \eqref{eq:EC_LFdiss2}, for the surface numerical fluxes.
We remark that the use of \eqref{eq:EC_LFdiss2} for the volume numerical fluxes produces an unpredictable behavior of the entropy balance since the second last term of \eqref{eq:EntropyDGSEMLemma} is weighted with the SBP operator, $\mat{Q}$, which can have positive and negative values.

\subsubsection{Entropy Balance of the Hybrid FV/DGSEM scheme} \label{sec:FirstOrder}

In \cite{Hennemann2020}, Hennemann et al. proved that a hybrid scheme that blends the DGSEM with the native LGL FV approximation is entropy stable for systems of conservation laws, given an appropriate choice of the numerical flux functions.
In this section, we provide a generalization of that proof that holds for DGSEM discretizations of non-conservative systems that are blended with first- and higher-order subcell Finite Volume discretizations.

\begin{lemma} \label{lemma:WillItBlend?}
The semi-discrete entropy balance of a discretization scheme for a non-conservative system that is obtained by blending two schemes at the element level,
\begin{equation}  \label{eq:Blending}
J \omega_j \dot{\state{u}}_j
= (1 - \alpha) \state{F}_j^{\DG} + \alpha \state{F}_j^{\FV}, ~~~ \forall j \in [0,N],
\end{equation}
where $\alpha$ is an  element-local blending coefficient, and both of the schemes are of the form
\begin{equation} \label{eq:SchemeToBlend}
\state{F}^i_j = \delta_{j0} \left(\numfluxb{f}^a_{(0,L)} + \numnonconsD{\Jan}_{(0,L)} \right)
- \delta_{jN} \left(\numfluxb{f}^a_{(N,R)} + \numnonconsD{\Jan}_{(N,R)} \right)
+ \state{F}^i_{\elemDom,j},
~~~~~~~~~~~~
i=\DG, \FV,
\end{equation}
with $\state{F}^i_{\elemDom,j}$ being any discretization terms that depend on the inner states of the element, is 
\begin{equation}
\sum_{j=0}^N J \omega_j \dot S_j=
\underbrace{
\numflux{f}^S_{(0,L)} -\numflux{f}^S_{(N,R)} + \frac{1}{2} \left( r_{(L,0)} + r_{(N,R)} \right) 
}_{\dot{S}_{\partial \elemDom}}
+ (1-\alpha) \dot{S}^{\DG}_{\elemDom} + \alpha \dot{S}^{\FV}_{\elemDom},
\end{equation}
where the terms gathered under $\dot{S}_{\partial \elemDom}$ are the numerical entropy flux, \eqref{eq:numEntFlux}, and the entropy production, \eqref{eq:EntropyDissipation}, on the boundaries of the element, which are intrinsic to the choice of the surface numerical flux function and the surface non-conservative term, and $\dot{S}^i_{\elemDom}$ is the entropy production of the scheme $i$ inside the element, which only depends on inner states.

\end{lemma}

\begin{proof}

The entropy balance within an element for a scheme $i$ of the form \eqref{eq:SchemeToBlend} reads
\begin{equation} \label{eq:EntropyBlended1}
\sum_{j=0}^N J \omega_j \dot S^i_j=
\sum_{j=0}^N \entVar^T \state{F}^i_j =
\entVar^T_0  \numfluxb{f}^a_{(0,L)} 
+ \entVar^T_0  \numnonconsD{\Jan}_{(0,L)}
- \entVar^T_N  \numfluxb{f}^a_{(N,R)} 
- \entVar^T_N  \numnonconsD{\Jan}_{(N,R)} 
+ \sum_{j=0}^N \entVar^T \state{F}^i_{\elemDom,j}.
\end{equation}
Following the strategy used in the proofs of Lemmas \ref{lemma:EntropyFV} and \ref{lemma:EntropyDGSEM}, we sum and subtract boundary terms to obtain
\begin{align} \label{eq:EntropyBlended}
\sum_{j=0}^N J \omega_j \dot S^i_j=&
\entVar^T_0  \numfluxb{f}^a_{(0,L)} 
+ \entVar^T_0  \numnonconsD{\Jan}_{(0,L)}
- \entVar^T_N  \numfluxb{f}^a_{(N,R)} 
- \entVar^T_N  \numnonconsD{\Jan}_{(N,R)} 
+ \sum_{j=0}^N \entVar^T \state{F}^i_{\elemDom,j} \nonumber\\
&
+ \frac{1}{2} \left(
\entVar^T_L  \numfluxb{f}^a_{(L,0)} 
+ \entVar^T_L  \numnonconsD{\Jan}_{(L,0)} 
- \Psi_L 
- \entVar^T_R  \numfluxb{f}^a_{(R,N)} 
- \entVar^T_R  \numnonconsD{\Jan}_{(R,N)} 
+ \Psi_R
\right)
- \Psi_0 + \Psi_N
\nonumber\\
&
- \frac{1}{2} \left(
\entVar^T_L  \numfluxb{f}^a_{(L,0)} 
+ \entVar^T_L  \numnonconsD{\Jan}_{(L,0)} 
- \Psi_L 
- \entVar^T_R  \numfluxb{f}^a_{(R,N)} 
- \entVar^T_R  \numnonconsD{\Jan}_{(R,N)} 
+ \Psi_R
\right)
+ \Psi_0 - \Psi_N
\nonumber\\
=&
\numflux{f}^S_{(0,L)} -\numflux{f}^S_{(N,R)} + \frac{1}{2} \left( r_{(L,0)} + r_{(N,R)} \right) 
+ \dot{S}^i_{\elemDom},
\end{align}
where we used the definition of the numerical entropy flux, \eqref{eq:numEntFlux}, and the entropy production, \eqref{eq:EntropyDissipation}.
The entropy production in the volume, $\dot{S}^i_{\elemDom}$, only depends on inner degrees of freedom and can be written as
\begin{equation}
\dot{S}^i_{\elemDom} = 
\sum_{j=0}^N \entVar^T \state{F}^i_{\elemDom,j} + \Psi_0 - \Psi_N.
\end{equation}

It is now easy to see that the entropy balance of the blended scheme, \eqref{eq:Blending}, after contracting with the entropy variables and integrating over the element reads
\begin{equation}
\sum_{j=0}^N J \omega_j \dot S_j=
\numflux{f}^S_{(0,L)} -\numflux{f}^S_{(N,R)} + \frac{1}{2} \left( r_{(L,0)} + r_{(N,R)} \right) 
+ (1-\alpha) \dot{S}^{\DG}_{\elemDom} + \alpha \dot{S}^{\FV}_{\elemDom}.
\end{equation}

\end{proof}

The main consequence of Lemma \ref{lemma:WillItBlend?} is that the resulting scheme is semi-discretely \textit{entropy consistent} with the blended schemes.
In other words, if both blended schemes are entropy conservative,
\begin{equation}
\frac{1}{2} \left( r_{(L,0)} + r_{(N,R)} \right) 
+ (1-\alpha) \dot{S}^{\DG}_{\elemDom} + \alpha \dot{S}^{\FV}_{\elemDom} = 0,
\end{equation}
the resulting scheme is entropy conservative.
This is true even if the blending coefficient, $\alpha$, is different for every element since, in this case, the entropy balance is independent of $\alpha$.
Furthermore, the resulting scheme is entropy stable if one of the blended schemes is entropy stable and the other one is entropy conservative or entropy stable.\\

The proposed hybrid scheme, \eqref{eq:BlendedScheme}, reads
\begin{equation} \label{eq:BlendedScheme2}
\underbrace{J \omega_j}_{\Delta x_j} \dot{\state{u}}_j = 
(1-\alpha) \state{F}_j^{a,\DG} 
+ \alpha \state{F}_j^{a,\FV} 
- \state{F}_j^{\nu,\DG},
\end{equation}

Since both $\state{F}_j^{a,\FV}$, \eqref{eq:FVMHD_BlendedForm}, and $\state{F}_j^{a,\DG}$, \eqref{eq:DGSEMMHDadv2}, are of the form \eqref{eq:SchemeToBlend}, we can use Lemma \ref{lemma:WillItBlend?} to analyze the entropy behavior of the resulting scheme.
Gathering the results from the Lemmas \ref{lemma:EntropyFV} and \ref{lemma:EntropyDGSEM}, the entropy balance within an element for the scheme \eqref{eq:BlendedScheme2} reads
\begin{equation}
\sum_{j=0}^N J \omega_j \dot S_j=
\numflux{f}^S_{(0,L)} -\numflux{f}^S_{(N,R)} + \frac{1}{2} \left( r_{(L,0)} + r_{(N,R)} \right) 
+
(1-\alpha) \underbrace{ 
\sum_{j,k=0}^N Q_{jk} r^{\DG}_{(j,k)}
}_{\dot{S}^{\DG}_{\elemDom}}
+
\alpha \underbrace{ 
\sum_{j=0}^{N-1} r^{\FV}_{(j,j+1)}
}_{\dot{S}^{\FV}_{\elemDom}}
+ r^{\nu},
\end{equation}
where $r_{(\cdot,\cdot)}$ is the entropy production associated with the numerical flux that is used on the element boundaries by the DGSEM and FV methods, $\numflux{f}^{S}_{(\cdot,\cdot)}$ is the entropy flux across the element boundaries, $r^{\DG}_{(j,k)}$ is the entropy production associated with the volume numerical flux of the DGSEM method, $\state{f}^{*}_{(j,k)}$, and $r^{\FV}$ is the entropy production associated with $\numfluxb{f}^{a,\FV}$.
We can choose these fluxes independently to control the entropy behavior of the scheme.

\section{Enhancing the Resolution of the FV/DGSEM blended scheme} \label{sec:TVD-ES}

A first-order FV scheme can be very sensitive to changes in cell sizes.
As a result, the LGL subcell distribution can cause the mesh to be imprinted on the solution (see e.g. \cite{Sonntag2014} and the \nameref{sec:Results} section).
In order to mitigate this phenomenon, in this section we propose higher-order FV schemes that use a reconstruction procedure, which can be seamlessly blended with the high-order DGSEM to obtain an entropy stable method.
The main focus here is to improve the discretization of the conservative terms.

Several reconstruction procedures are available in the FV literature \cite{VanLeer1974,Coquel1996,Coquel2006,Fjordholm2012}.
Most of these methods, instead of a piece-wise constant solution, assume a higher-order reconstructed solution in the FV cells, with which the numerical flux functions are evaluated.
Unfortunately, when the numerical flux is evaluated on a reconstructed solution, it is complicated and expensive, if not impossible, to guarantee entropy stability, especially if we want to blend the FV scheme with a high-order DG scheme.

We focus on the framework developed by \citet{Fjordholm2012}, as it can be used to construct provably entropy stable FV methods using inexpensive total variation diminishing (TVD) reconstructions.
\citet{Fjordholm2012} constructs arbitrarily high-order entropy stable schemes, where the stencil grows with the order of accuracy.
These schemes have been successfully used for the Euler \cite{Biswas2018} and MHD equations \cite{Chandrashekar2016}.

We use the second-order scheme of Fjordholm et al., which we will call the \textit{TVD-ES} scheme in the remaining parts of the paper, as it uses a TVD reconstruction and preserves the entropy stability.
The numerical flux of the TVD-ES scheme between the FV subcells $L$ and $R$ reads
\begin{equation} \label{eq:SchemeFjord}
\state{f}_1^{\tvdes}(\state{u}_L,\state{u}_R) := \state{f}_1^{\ec}(\state{u}_L,\state{u}_R) 
- \frac{1}{2} 
\mat{\mathcal{D}}^{\entVar}
\jumpR{\entVar}_{(L,R)},
\end{equation}
where the EC flux is evaluated on the mean (not reconstructed) states of the subcells, $\state{u}_L$ and $\state{u}_R$, and the new jump operator, $\jumpR{\entVar}_{(L,R)}$, denotes the symmetric jump in the reconstructed entropy variables.
To ensure entropy stability, the reconstruction must be done such that
\begin{equation} \label{eq:ReconsJump}
\jumpR{\entVar}_{(L,R)} = 
\left( \mat{R}^T \right)^{-1} \mat{B}^{\es} ~ \mat{R}^T \jumpS{\entVar}_{(L,R)},
\end{equation}
where the reconstruction procedure is defined by the diagonal matrix $\mat{B}^{\es}>0$.
We remark that \citet{Winters2016} showed that the EC flux is second-order accurate if, and only if, the mesh is regular.
However, even for the irregular meshes that we consider in this paper, the reconstructed dissipation operator mitigates the mesh imprinting that is observed with the original Native LGL Finite Volume approximation, as we show in the \nameref{sec:Results} section.

\citet{Fjordholm2012} defined the \textit{scaled entropy variables} as $\scalEntVar = \mat{R}^T \entVar$, such that \eqref{eq:ReconsJump} yields
\begin{equation} \label{eq:ReconsJumpScal}
\jumpR{\scalEntVar}_{(L,R)} = \mat{B}^{\es}  \jumpS{\scalEntVar}_{(L,R)}.
\end{equation}

Since $\mat{B}^{\es}$ is a positive diagonal matrix, the reconstruction procedure must guarantee component-wise that the symmetric jump in the reconstructed scaled entropy variables has the same sign as the symmetric jump in the scaled entropy variables.
In other words, the reconstruction must fulfill the so-called \textit{sign property} of the scaled entropy variables.
According to \citet{Fjordholm2012}, and \citet{Biswas2018}, the \textit{minmod} limiter is the only symmetric TVD limiter that fulfills the sign property.

Using Definition \ref{eq:EntropyDissipation}, it is easy to show that the entropy production between the degrees of freedom $j$ and $k$ ($j<k$) associated with \eqref{eq:SchemeFjord} is
\begin{align} \label{eq:EntProdFjod}
r_{(j,k)} =& - \frac{1}{2} \jump{\entVar}_{(j,k)}^T \mat{\mathcal{D}}^{\entVar} \jumpR{\entVar}_{(j,k)} \nonumber\\
=& - \frac{1}{2} \jump{\entVar}_{(j,k)}^T
\underbrace{ 
\mat{R}~|\mat{\Lambda}| ~ \mat{\mathcal{Z}}~ {\mat{R}^T}
}_{ \mat{\mathcal{D}}^{\entVar} }
\underbrace{ 
{\left( \mat{R}^T \right)^{-1}} \mat{B}^{\es} ~ \mat{R}^T \jump{\entVar}_{(j,k)}}_{\jumpR{\entVar}_{(j,k)}} \nonumber\\
=& - \frac{1}{2} \jump{\entVar}_{(j,k)}^T
\mat{R}~
|\mat{\Lambda}| ~ \mat{\mathcal{Z}}~ \mat{B}^{\es}
~\mat{R}^T \jump{\entVar}_{(j,k)}.
\end{align}
In other words, the scheme is always entropy stable by construction, as entropy is either conserved or dissipated.
We remark that it is not necessary to compute the inverse of the transpose of the eigenvalue matrix, $\mat{R}^T$, to evaluate $\mat{\mathcal{D}}^{\entVar} \jumpR{\entVar}_{(j,k)}$.

To obtain an entropy stable scheme that blends the DGSEM with the method of Fjordholm et al., we propose a reconstruction procedure that, again, interprets the nodal DG values as subcell mean values, and ensures that the subcell FV scheme is of the form \eqref{eq:SchemeToBlend}, such that Lemma \ref{lemma:WillItBlend?} holds.
We make the FV boundary values match with the DGSEM boundary values, and reconstruct the scaled entropy variables inside each subcell with a two-point symmetric \textit{minmod} limiter in reference space for irregular meshes, such that the sign property is preserved.

In our proposed reconstruction, the $l^{\text{th}}$ component of the scaled entropy variables inside every subcell is given by
\begin{equation} \label{eq:minmod_ini}
\tilde{\scalEntVar}_j^l (\xi) = \scalEntVar_j^l + (\xi - \xi_j) \Theta_j^l,
\end{equation}
and the slope is computed for the inner subcells with the \textit{minmod} function as 
\begin{equation}
\Theta_j^l = \mathrm{MINMOD} \left( \frac{\scalEntVar^l_{j+1} - \scalEntVar^l_{j}} {\xi_{j+1} - \xi_j} , 
\frac{\scalEntVar^l_{j} - \scalEntVar^l_{j-1}}{\xi_{j} - \xi_{j-1}} \right).
\end{equation}

We have several alternatives for the subcells that lie on the boundary of the element.
For instance,
\begin{itemize}
\item we can assume them to have piece-wise constant values,
\begin{equation} \label{eq:TVDrecons_NoRecons}
\Theta_0^l = \Theta_N^l = 0,
\end{equation}

\item we can use a central slope.
\begin{equation} \label{eq:TVDrecons_Central}
\Theta_0^l =\frac{\scalEntVar^l_{1} - \scalEntVar^l_{0}} {\xi_{1} - \xi_0}, ~~~~
\Theta_N^l =\frac{\scalEntVar^l_{0} - \scalEntVar^l_{N-1}} {\xi_{N} - \xi_{N-1}},
\end{equation}

\item or we can use the neighbor information to compute the slope with the \textit{minmod} limiter,
\begin{equation} \label{eq:TVDrecons_Bound}
\Theta_0^l = \minmod \left( \frac{\scalEntVar^l_{1} - \scalEntVar^l_{0}} {\xi_{1} - \xi_0} , 
\frac{\scalEntVar^l_{1} - \scalEntVar^{\underline{l}}_{N}}{\xi_{1} - \xi_{0}} \right), ~~~~
\Theta_N^l = \minmod \left( \frac{\scalEntVar^l_{0} - \scalEntVar^l_{N-1}} {\xi_{N} - \xi_{N-1}} , 
\frac{\scalEntVar^{\underline{l}}_{0} - \scalEntVar^l_{N-1}}{\xi_{N} - \xi_{N-1}} \right),
\end{equation}
where $\scalEntVar^{\underline{l}}_{N}$ is the $l^{\mathrm{th}}$ value of $\scalEntVar$ for the node $N$ of the element on the left, $\scalEntVar^{\underline{l}}_{0}$ is the $l^{\mathrm{th}}$ value of $\scalEntVar$ for the node $0$ of the element on the right.
\end{itemize} 

The \textit{minmod} function is defined as
\begin{equation} \label{eq:minmod_end}
\minmod(a,b) = 
\begin{cases}
\sign(a) \min(|a|,|b|) & \mathrm{if}~\sign(a) = \sign(b), \\
0 & \mathrm{otherwise}.
\end{cases}
\end{equation}

A schematic representation of the reconstruction procedure is given in Figure \ref{fig:TVD_Recons} for $N=5$ and a reconstruction of the boundary subcells that uses the neighbor information.
Note that the FV and DGSEM solutions concur on the element boundaries.

Note that \eqref{eq:TVDrecons_Bound} needs the same connectivity between elements as the DGSEM method.
As a consequence, the method detailed here can be implemented in a DGSEM code by only replacing the volume integral and without any additional MPI communication (the MPI footprint is the same as in the DGSEM).

\begin{figure}[htb]
\centering
\includegraphics[width=0.7\linewidth]{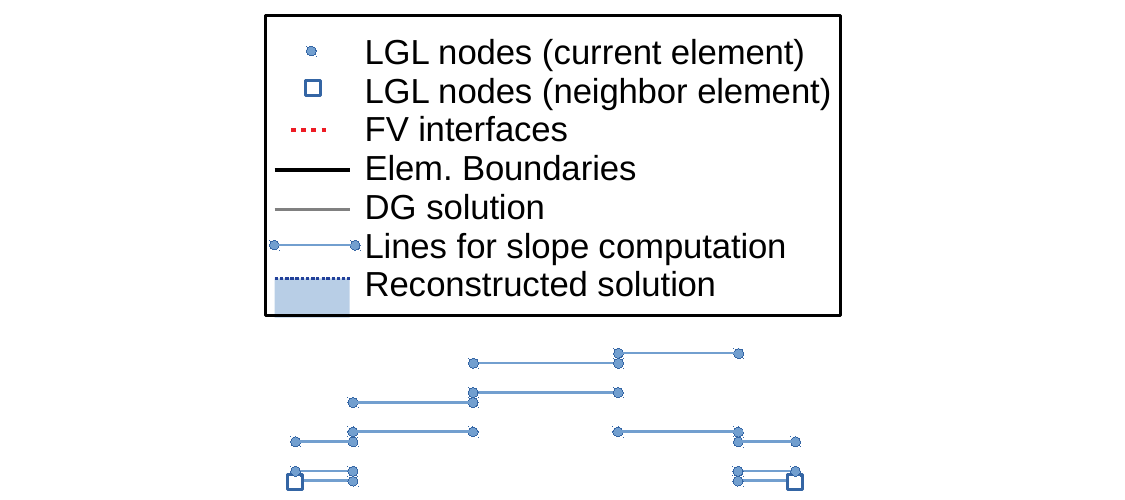}
\includegraphics[width=0.7\linewidth]{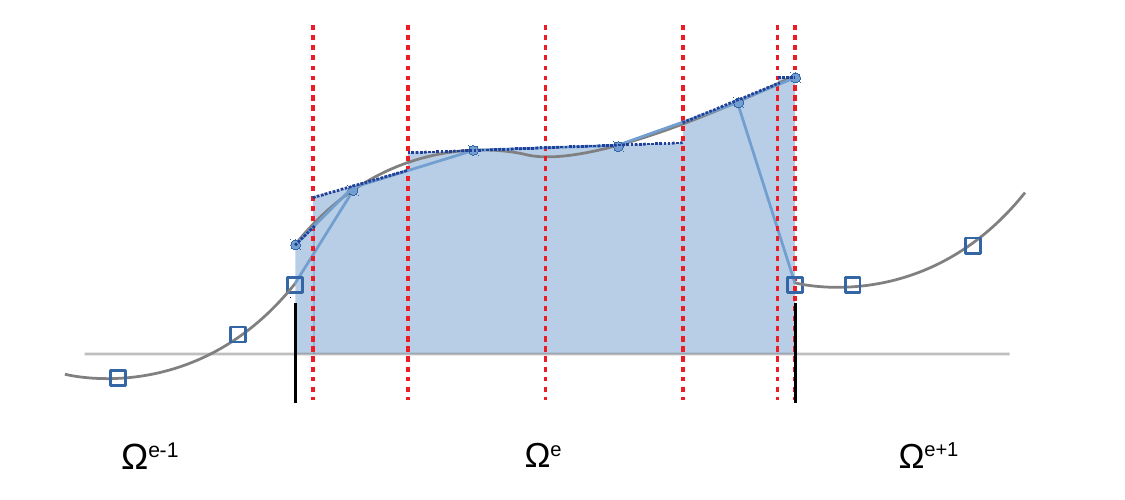}
\caption{Schematic representation of the node connectivities used to compute the slopes for each subcell interface (top) and example of a reconstructed solution using the \textit{minmod} limiter (bottom) for a DGSEM element with FV subcells ($N=5$).}
\label{fig:TVD_Recons}
\end{figure}

\section{Shock Indicator} \label{sec:Indicator}

The methods described above can be used with any troubled cell indicator.
For simplicity, we use the shock sensor introduced by \citet{Persson2006} that compares the modal energy of the highest polynomial modes of an \textit{indicator quantity} with its overall modal energy.
We transform our indicator quantity, $\epsilon$, from a (collocated) nodal representation to a hierarchical modal representation with Legendre polynomials,
\begin{equation}
\epsilon (\xi)= 
\underbrace{\sum_{j=0}^N {\epsilon}_j \ell_j (\xi) }_{\mathrm{nodal}} =
\underbrace{\sum_{j=0}^N \hat{\epsilon}_j \tilde{L}_j (\xi)}_{\mathrm{modal}},
\end{equation}
where $\{\hat{\epsilon}\}_{j=0}^N$ are the modal coefficients and $\{\tilde{L}_j\}_{j=0}^N$ are the Legendre polynomials.

We compute for each DG element how much energy is contained in the highest modes relative to the total energy of the polynomial as follows
\begin{equation}
\mathbb{E} = \max\left(\frac{\hat{\epsilon}_N^2}{\sum_{j=0}^{N} \hat{\epsilon}_j^2}, \frac{\hat{\epsilon}_{N-1}^2}{\sum_{j=0}^{N-1} \hat{\epsilon}_j^2}\right),
\end{equation}
where we use the highest and second highest mode to avoid odd/even effects when approximating element local functions.

Initially, we define the blending coefficient as
\begin{equation}
\tilde \alpha = \frac{1}{1+\exp \left( \frac{-s}{\mathbb{T}}(\mathbb{E}-\mathbb{T})\right)},
\end{equation}
where the so-called sharpness, $s=9.21024$, is selected as in \cite{Hennemann2020} to obtain $\alpha(\mathbb{E}=0)=0.0001$, and the so-called threshold is computed as
\begin{equation}
\mathbb{T}(N)=0.5 \cdot 10^{-1.8 (N+1)^{0.25}},
\end{equation}
based on \cite{Hennemann2020} and motivated by the discussion in \cite{Persson2006} about the spectral decay of the modes that is proportional to $1/N^4$.

The final blending coefficient is computed as
\begin{equation}
\alpha = 
\begin{cases}
0 & \mathrm{if}~ \tilde \alpha < \alpha_{\min}, \\
\tilde \alpha & \mathrm{if}~ \tilde \alpha_{\min} \le \alpha \le \alpha_{\max},  \\
\alpha_{max} & \mathrm{if}~ \tilde \alpha > \alpha_{\max},
\end{cases}
\end{equation}
where we choose $\alpha_{\min}=0.01$ as a way to improve the computational efficiency of the method in regions where only small limiting is needed, and $\alpha_{\max}=1$ to be able to deal with strong shocks.

The modal indicator described above might not be optimal for systems with moving shocks.
For instance, we have observed that the indicator may switch on and off depending on the relative position of the shocks with the element boundaries. 
To avoid large oscillations of the blending coefficient in time, we perform a time relaxation such that the blending coefficient in time $i+1$ is set to 
\begin{equation}
\alpha^{i+1}  = \max \{ \alpha^{i+1}, 0.7 \alpha^i \},
\end{equation}
unless otherwise explicitly stated.
Furthermore, to avoid large jumps in the blending coefficient from element to element, unless otherwise explicitly stated, we perform two space propagation sweeps such that for each element
\begin{equation}
\alpha  = \max_E \{ \alpha, 0.7 \alpha_E \},
\end{equation}
where $\alpha_E$ denotes the blending coefficient of any neighbor element.

\section{Numerical Results} \label{sec:Results}

In this section, we present the numerical validation of the methods presented in the paper and use them to solve well-known benchmark tests and applications.
For simplicity, we use the entropy stable version of the Rusanov scheme in all the cases that need dissipation in the numerical fluxes, i.e. \eqref{eq:EC_LFdiss2} where $\mat{\Lambda}$ is a diagonal matrix with the largest advective eigenvalue in all the nonzero entries.
As will be seen, the ES Rusanov solver is enough to obtain sharp shock profiles with the high-order DGSEM/FV method that we use.
Furthermore, we use the EC flux of \citet{Derigs2018} (see Appendix \ref{sec:ECflux}) for the DGSEM volume flux and, for consistency, we use the same numerical flux in the FV method for the element boundaries and the element interior, $\numfluxb{f}^a = \numfluxb{f}^{a,\FV}$.

In all cases, the time integration is performed with the fourth-order Strong Stability-Preserving Explicit Runge-Kutta (SSPRK) method of five stages introduced by \citet{Spiteri2002}, and the blending coefficient, $\alpha$, is computed before every Runge-Kutta stage.

The time-step size is computed as \cite{Krais2019}
\begin{equation}
\Delta t = \min \left( 
\frac{\text{CFL} \, \beta^a(N) \Delta x}{\lambda^a_{\max} (2N+1)}
,
\frac{\text{CFL}^{\nu} \, \beta^{\nu}(N) \Delta x^2}{\lambda^{\nu}_{\max} (2N+1)^2}
 \right)
\end{equation}
where CFL, CFL$^{\nu} \le 1$ are the advective and diffusive CFL numbers, $\lambda^a_{\max}$ and $\lambda^{\nu}_{\max}$ are the largest advective and diffusive eigenvalues, respectively, $\Delta x$ is the element size, and $\beta^a$ and $\beta^{\nu}$ are proportionality coefficients that are derived for the SSPRK from numerical experiments such that CFL, CFL$^{\nu} \le 1$ must hold to obtain a (linear) CFL-stable time step for all polynomial degrees.
As a "conservative" approach, we use $\text{CFL}=\text{CFL}^{\nu}=0.5$.

Furthermore, the hyperbolic divergence cleaning speed, $c_h$, is selected at each time step as the maximum value that retains CFL-stability, and we use $\mu_0=1$ as the magnetic permeability of the medium and $\gamma=5/3$ as the heat capacity ratio.

All the FV/DGSEM simulations presented in this section were computed with the 3D open-source code FLUXO (\url{www.github.com/project-fluxo}).
The 2D simulations were computed with 2D extruded meshes with one-element in the $z$-direction.

\subsection{Numerical Verification of the Schemes}

The goal of this test is to numerically validate that the method is indeed free-stream-preserving and EC/ES for general 3D meshes. 
We use a 3D heavily warped mesh adapted from \cite{chan2019efficient}.
We start with the cube $\Omega=[0,3]^3$ with $10^3$ elements and apply the transformation
\begin{equation}
X(\xi,\eta,\zeta) = (x,y,z): \Omega \rightarrow f(\Omega)
\end{equation}
such that
\begin{align}
y &= \eta + 
\frac{1}{8} L_y
\cos \left( \frac{3}{2} \pi \frac{2\xi - L_x}{L_x} \right) 
\cos \left( \frac{\pi}{2} \frac{2\eta- L_y }{L_y} \right)
\cos \left( \frac{\pi}{2} \frac{2\zeta- L_z }{L_z} \right), \\
x &= \xi + 
\frac{1}{8} L_x 
\cos \left( \frac{\pi}{2} \frac{2\xi- L_x}{L_x} \right)
\cos \left( 2 \pi \frac{2y- L_y }{L_y} \right)
\cos \left( \frac{\pi}{2} \frac{2\zeta- L_z}{L_z} \right), \\
z &= \zeta + 
\frac{1}{8} L_z 
\cos \left( \frac{\pi}{2} \frac{2x- L_x}{L_x} \right)
\cos \left( \pi \frac{2y- L_y }{L_y}  \right)
\cos \left( \frac{\pi}{2} \frac{2\zeta- L_z}{L_z} \right),
\end{align}
where $L_x = L_y = L_z = 3$. 
The mesh, which can be seen in Figure \ref{fig:warped3Dmesh}, was generated with the HOPR package \cite{hindenlang2015mesh} with a geometry mapping degree $N_{\mathrm{geo}}=4$.
All boundaries are set to periodic.

\begin{figure}[htb]
\centering
\includegraphics[trim=500 0 500 0,clip,width=0.45\linewidth]{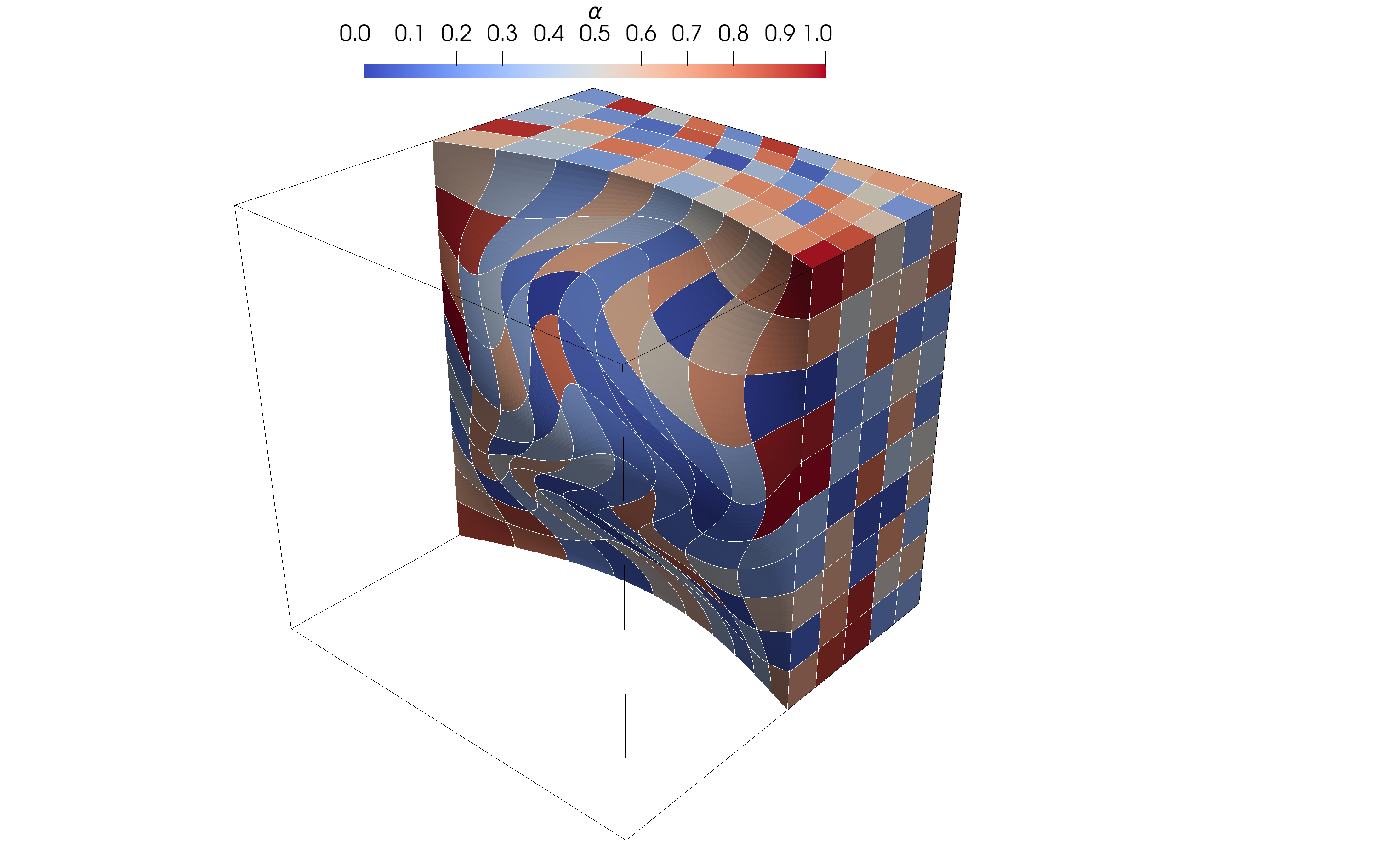}\\

\caption{Slice cut visualization of the heavily warped mesh and initial random blending coefficients for the free-stream-preservation test.}
\label{fig:warped3Dmesh}
\end{figure}

Free-stream preservation (FSP) must hold because, as is shown by \citet{Hennemann2020}, it is necessary to ensure entropy conservation and stability.
To test FSP, the initial condition is set to a uniform flow, $\state{u}(t=0) = \state{u}_{\mathrm{FSP}}$, given in Table \ref{tab:statesBlast}.
%
The blending function is selected randomly in each element of the domain at each Runge-Kutta stage (see Figure \ref{fig:warped3Dmesh}).
Furthermore, the spatial propagation and time relaxation of the blending coefficient were deactivated for this experiment, such that the random blending coefficients are not affected.

Table \ref{tab:3DcurvedFSP} summarizes the results for the free-stream preservation test.
These results were obtained using $N=4$ and $\mathrm{CFL}=0.1$.
Similar results can be obtained with other polynomial degrees and CFL numbers.
The $\mathbb{L}^2$ norm is computed as
\begin{equation}
\norm{u}_{\mathbb{L}^2} = 
\left( \frac{\int^N_{\Omega} u^2 \d \vec{x}}{\int^N_{\Omega} \d \vec{x}} \right)^{\frac{1}{2}},
\end{equation}
where the superscript $N$ on the integral denotes the approximation of it with a quadrature rule with $N+1$ points per element and direction.

The second column of Table \ref{tab:3DcurvedFSP} shows the mean rate of change of all state quantities at $t=0$ for the entropy conservative (EC), entropy stable (ES) and TVD-reconstructed entropy stable (TVD-ES) surface numerical fluxes.
The value is the same for the three choices of the numerical flux since, in the absence of jumps in the solution, the dissipation term is equal to zero.
As expected, the rate of change of the state variables is near machine precision at the beginning of the simulation.

The third, fourth and fifth columns of Table \ref{tab:3DcurvedFSP} show the mean deviation from the initial condition of all state quantities at $t=1$ for the EC, ES and TVD-ES surface numerical fluxes, respectively.
The three different choices of the surface numerical flux preserve the free stream with errors near machine precision.
The error of the ES schemes is lower than the one of the EC scheme since the extra dissipation that they provide smoothens away the deviations from the free stream.

\begin{table}[hb]
\centering
\begin{tabular}{c|c|ccc}
 & $\displaystyle \norm{\bigpartialderiv{u}{t}(t=0)}_{\mathbb{L}^2}$ & \multicolumn{3}{c}{$\displaystyle \norm{u(t=1) - u(t=0)}_{\mathbb{L}^2}$ } \\ \hline
$u$ & EC,ES,TVD-ES & EC & ES & TVD-ES \\ \hline
$\rho$ & 
$1.59\cdot 10^{-13}$ & $2.27\cdot 10^{-13}$ & $4.25\cdot 10^{-15}$ & $4.91\cdot 10^{-15}$ \\
$\rho v_1$ & 
$9.85\cdot 10^{-13}$ & $2.63\cdot 10^{-13}$ & $1.28\cdot 10^{-14}$ & $1.39\cdot 10^{-14}$ \\
$\rho v_2$ & 
$8.90\cdot 10^{-13}$ & $2.98\cdot 10^{-13}$ & $1.33\cdot 10^{-14}$ & $1.46\cdot 10^{-14}$ \\
$\rho v_3$ & 
$9.93\cdot 10^{-13}$ & $3.22\cdot 10^{-13}$ & $1.39\cdot 10^{-14}$ & $1.50\cdot 10^{-14}$ \\
$\rho E$ & 
$8.73\cdot 10^{-13}$ & $2.09\cdot 10^{-13}$ & $2.30\cdot 10^{-14}$ & $2.48\cdot 10^{-14}$ \\
$B_1$ & 
$1.55\cdot 10^{-13}$ & $2.46\cdot 10^{-13}$ & $7.67\cdot 10^{-15}$ & $8.81\cdot 10^{-15}$ \\
$B_2$ & 
$1.78\cdot 10^{-13}$ & $2.77\cdot 10^{-13}$ & $8.21\cdot 10^{-15}$ & $9.41\cdot 10^{-15}$ \\
$B_3$ & 
$1.59\cdot 10^{-13}$ & $2.99\cdot 10^{-13}$ & $8.16\cdot 10^{-15}$ & $9.41\cdot 10^{-15}$ \\
$\psi$ & 
$5.98\cdot 10^{-13}$ & $3.61\cdot 10^{-14}$ & $9.54\cdot 10^{-15}$ & $9.88\cdot 10^{-15}$ \\
\hline
\end{tabular}
\caption{Mean rate of change of the state variables at $t=0$ and their absolute deviation from the initial state at $t=1$ for the uniform flow computed with $\mathrm{CFL}=0.1.$}
\label{tab:3DcurvedFSP}
\end{table}

Now, to test entropy conservation and stability we initialize a weak magnetic blast in the same heavily warped domain. 
We use the same setup as \citet{Bohm2018}, where the initial condition is obtained as a blend of two states,
\begin{equation}
\state{u}(t=0) = \frac{\state{u}_{\mathrm{inner}}+ \lambda \state{u}_{\mathrm{outer}}}{1+\lambda}, \ \
\lambda = \exp \left[ \frac{5}{\delta_0} \left( r - r_0 \right) \right], \ \
r = \norm{\vec{x} - \vec{x}_c},
\end{equation}
where $\vec{x}_c = (1.5,1.5,1.5)^T$ is the center of the blast, $r_0 = 0.3$ is the distance to the center of the blast, $\delta_0 = 0.1$ is the approximate distance in which the two states are blended, and the inner and outer states are given in Table \ref{tab:statesBlast}.

\begin{table}[htbp!]
	\centering
		\begin{tabular}{c|ccccccccc}
			     & $\rho$ 	& $v_1$	& $v_2$	& $v_3$	& $p$ 	& $B_1$	& $B_2$	& $B_3$	& $\psi$ \\\hline
			$\state{u}_{\mathrm{FSP}}$ & $1.0$	& $0.1$	& $-0.2$	& $0.3$	&$1.0$	& $1.0$	& $1.0$	& $1.0$	& $0.0$ \\
			$\state{u}_{\mathrm{inner}}$ & $1.2$	& $0.1$	& $0.0$	& $0.1$	&$0.9$	& $1.0$	& $1.0$	& $1.0$	& $0.0$ \\
			$\state{u}_{\mathrm{outer}}$ & $1.0$	& $0.2$	& $-0.4$& $0.2$	&$0.3$	& $1.0$	& $1.0$	& $1.0$	& $0.0$ \\\hline 
		\end{tabular}
		\caption{Primitive states for the FSP, entropy conservation and entropy stability tests.}
\label{tab:statesBlast}
\end{table}

In the remaining part of the results section, the modal shock indicator of Section \ref{sec:Indicator} will be used.
For this particular test, we use $\epsilon = p$ as the shock indicator quantity.
The blast triggers the shock-capturing method, as can be seen in Figure \ref{fig:Blast_EC}, which illustrates the pressure and the blending coefficient for the entropy conservation test with $N=6$ at $t=0.5$.
The solution is clearly distorted because the EC flux does not add any dissipation to the numerical scheme.

\begin{figure}[htb]
\centering
\includegraphics[trim=500 0 500 0,clip,width=0.45\linewidth]{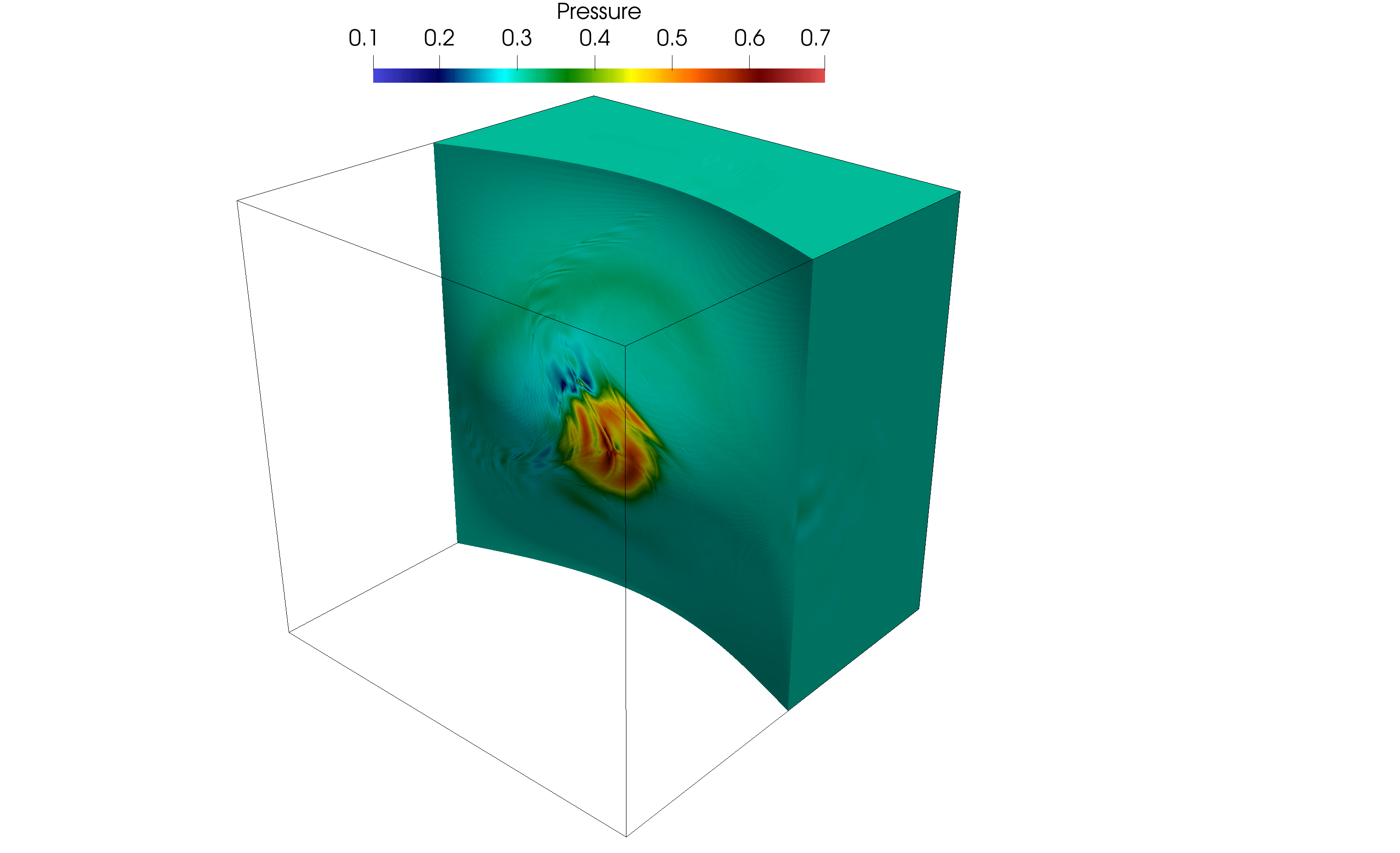}
\includegraphics[trim=500 0 500 0,clip,width=0.45\linewidth]{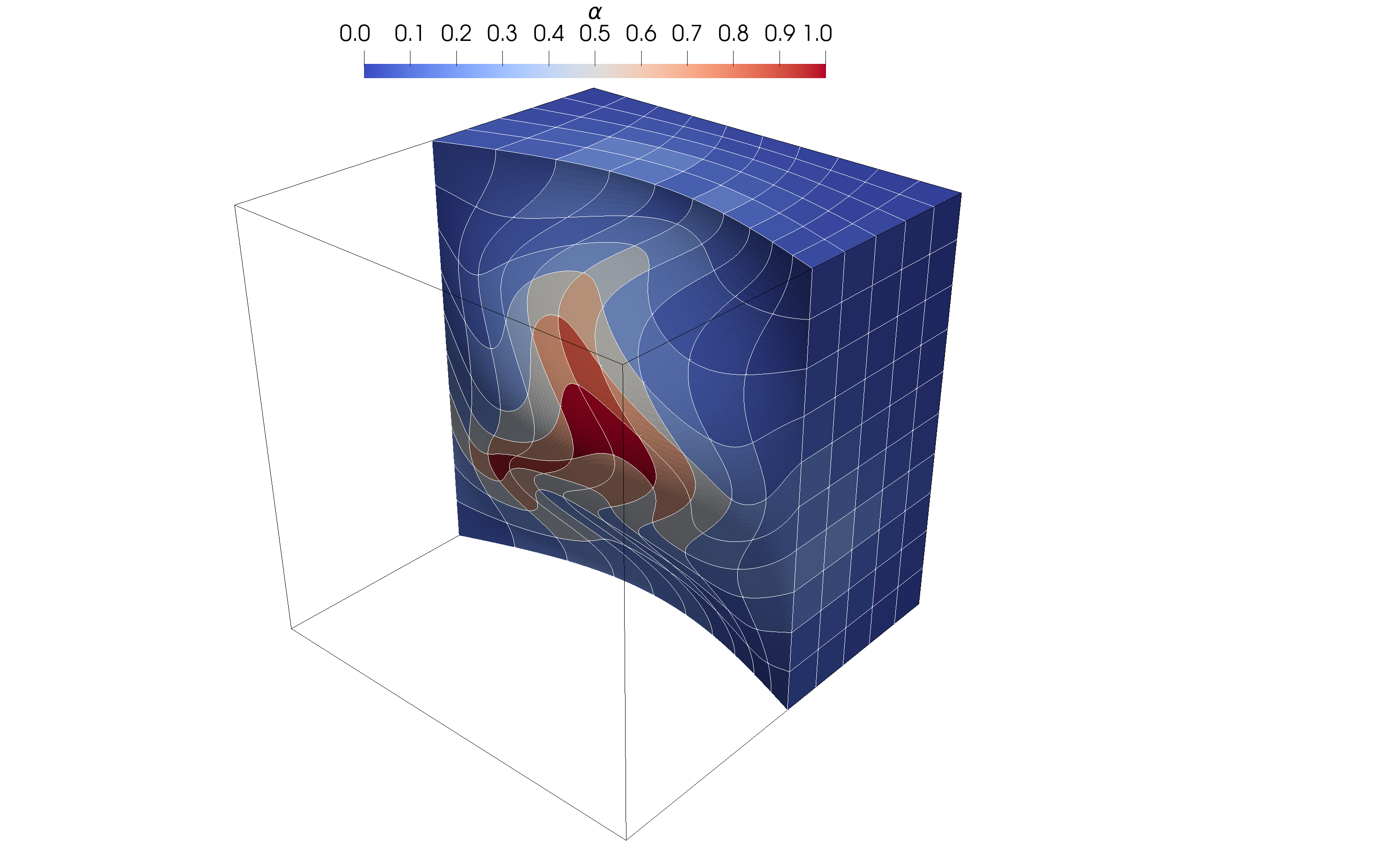}\\

\caption{Pressure and blending coefficient distribution on a slice cut for the entropy conservation test of the soft magnetic blast at $t=0.5$ ($\mathrm{CFL}=0.1$).}
\label{fig:Blast_EC}
\includegraphics[width=0.6\linewidth]{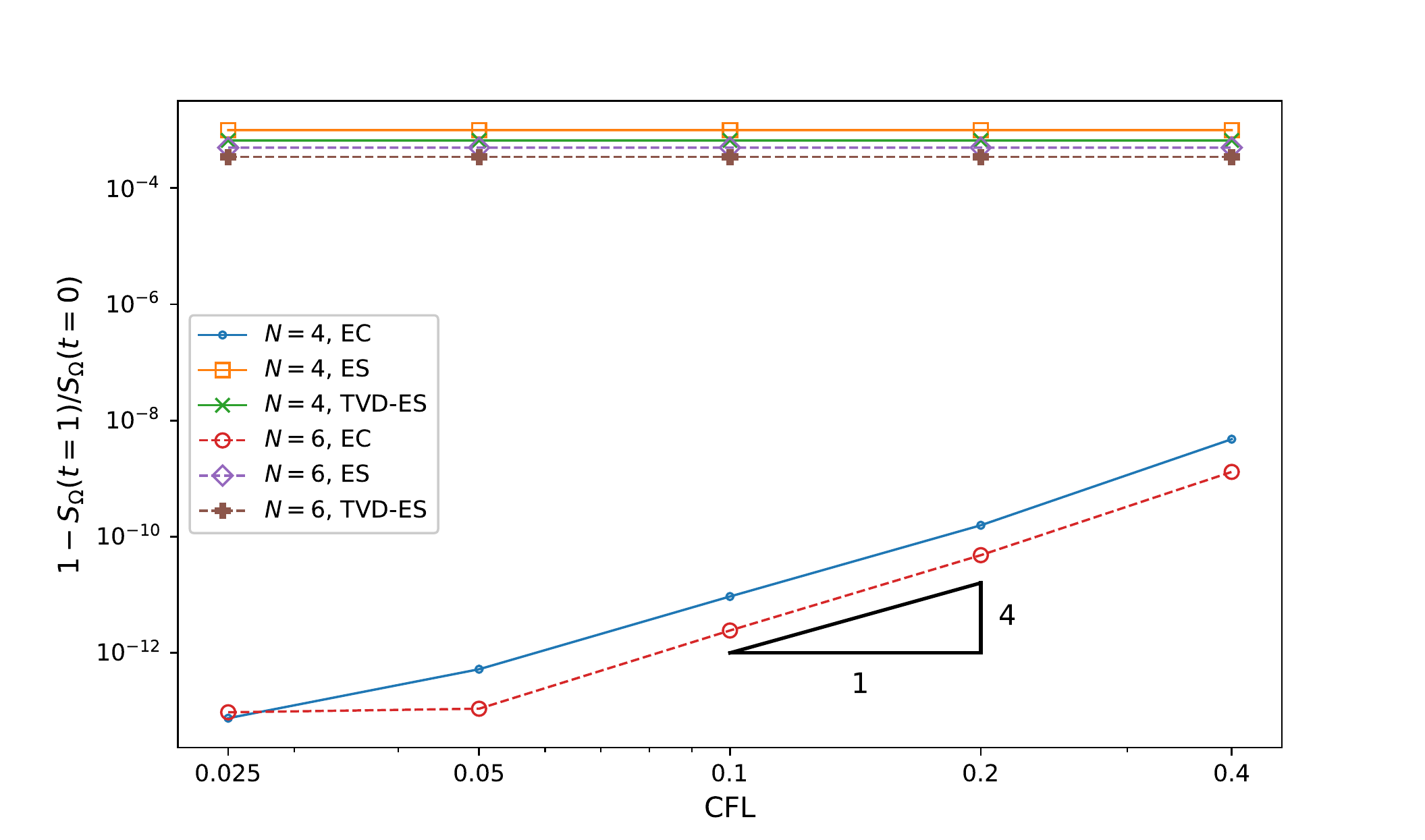}

\caption{Log-log plot of the entropy change from the initial entropy, $S_{\Omega}(t=0)$, to $S_{\Omega}(t=1)$ as a function of the CFL number for the different schemes.}
\label{fig:Blast_Entropy}
\end{figure}

Figure \ref{fig:Blast_Entropy} shows the total entropy change throughout the simulation for the EC, ES and TVD-ES schemes with $N=4$ and $N=6$.
The total entropy in the domain is computed as
\begin{equation}
S_{\Omega} = \int^N_{\Omega} S \d \vec{x}.
\end{equation}
In Section \ref{sec:FirstOrder}, we proved that EC schemes are entropy conservative at the semi-discrete level.
However, the time-integration scheme adds a non-zero entropy dissipation that depends on the time-step size.
As can be seen in Figure \ref{fig:Blast_Entropy}, the entropy dissipation of the EC scheme converges to zero with fourth-order accuracy (down to machine precision) as the time-step size is reduced.
Furthermore, it can be seen that the ES and TVD-ES schemes show entropy stability.

\subsection{Orszag-Tang Vortex}

This 2D case was originally proposed by \citet{orszag1979small} and is widely used to test the robustness of MHD codes \cite{Ciuca2020,Chandrashekar2016,Derigs2018,Stone2008}.
Starting from a smooth initial condition, this case evolves to a complex shock pattern with multiple shock-shock interactions and the transition to supersonic/transonic MHD turbulence.

We use the same setup as in \cite{Ciuca2020,Chandrashekar2016}. 
The simulation domain is $\Omega = [0,1]^2$ with a Cartesian grid and periodic boundary conditions, and the initial condition is set to
\begin{align*}
\rho(x,y,t=0) &= \frac{25}{36 \pi},
&  p(x,y,t=0) &= \frac{5}{12 \pi}, \\
 v_1(x,y,t=0) &= - \sin (2 \pi y), 
&v_2(x,y,t=0) &=   \sin (2 \pi x), \\
 B_1(x,y,t=0) &= -\frac{1}{\sqrt{4 \pi}} \sin (2 \pi y), 
&B_2(x,y,t=0) &= -\frac{1}{\sqrt{4 \pi}} \sin (4 \pi x),
\end{align*}
which fulfills $\Nabla \cdot \vec{B} = 0$ and gives a sound speed $a=1$.

\begin{figure}[htb]
\centering
\includegraphics[trim=600 1400 600 50,clip,width=0.5\linewidth]{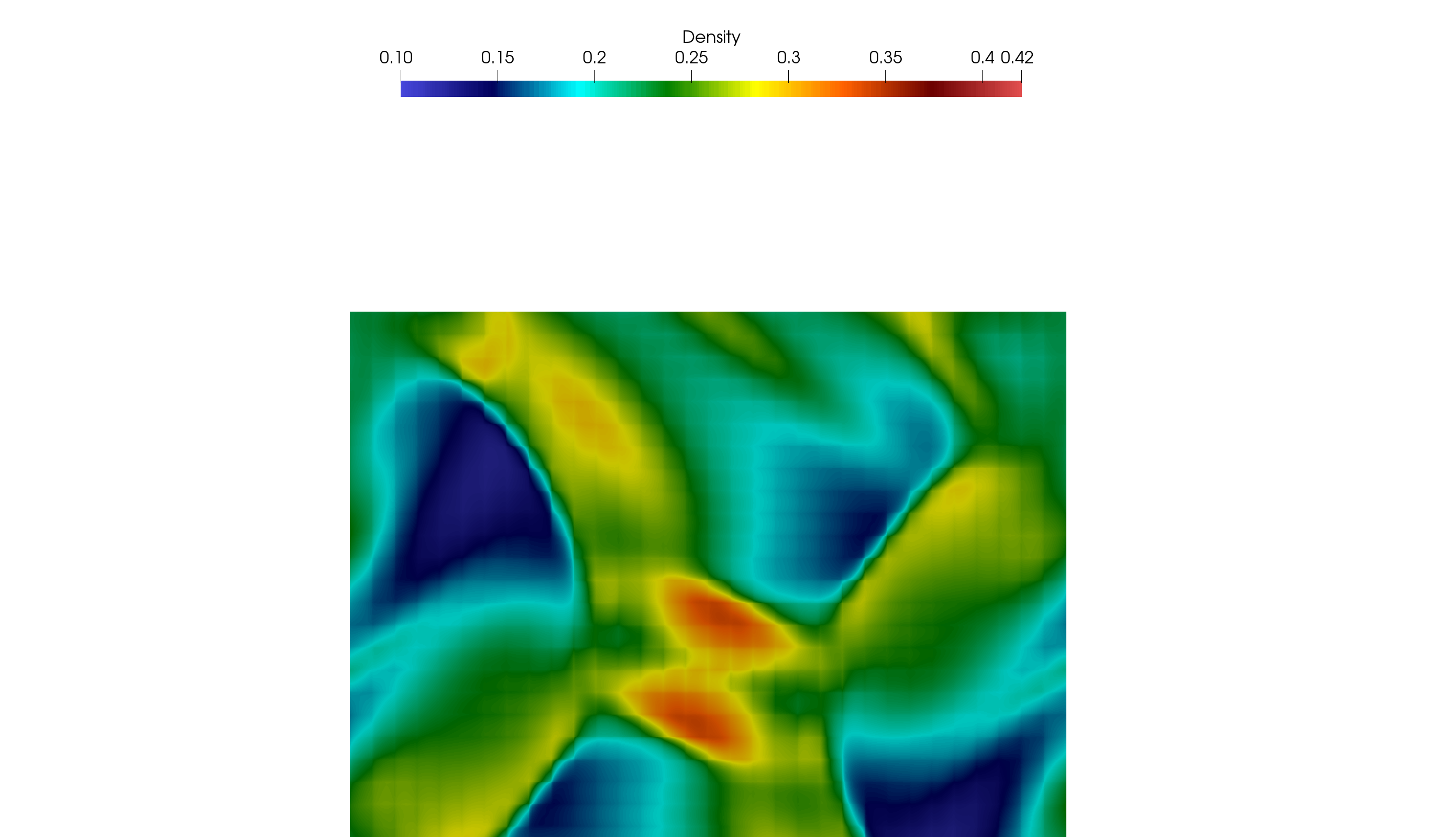}\\
\begin{subfigure}[b]{0.32\linewidth}
	\includegraphics[trim=700 100 700 100,clip,width=\linewidth]{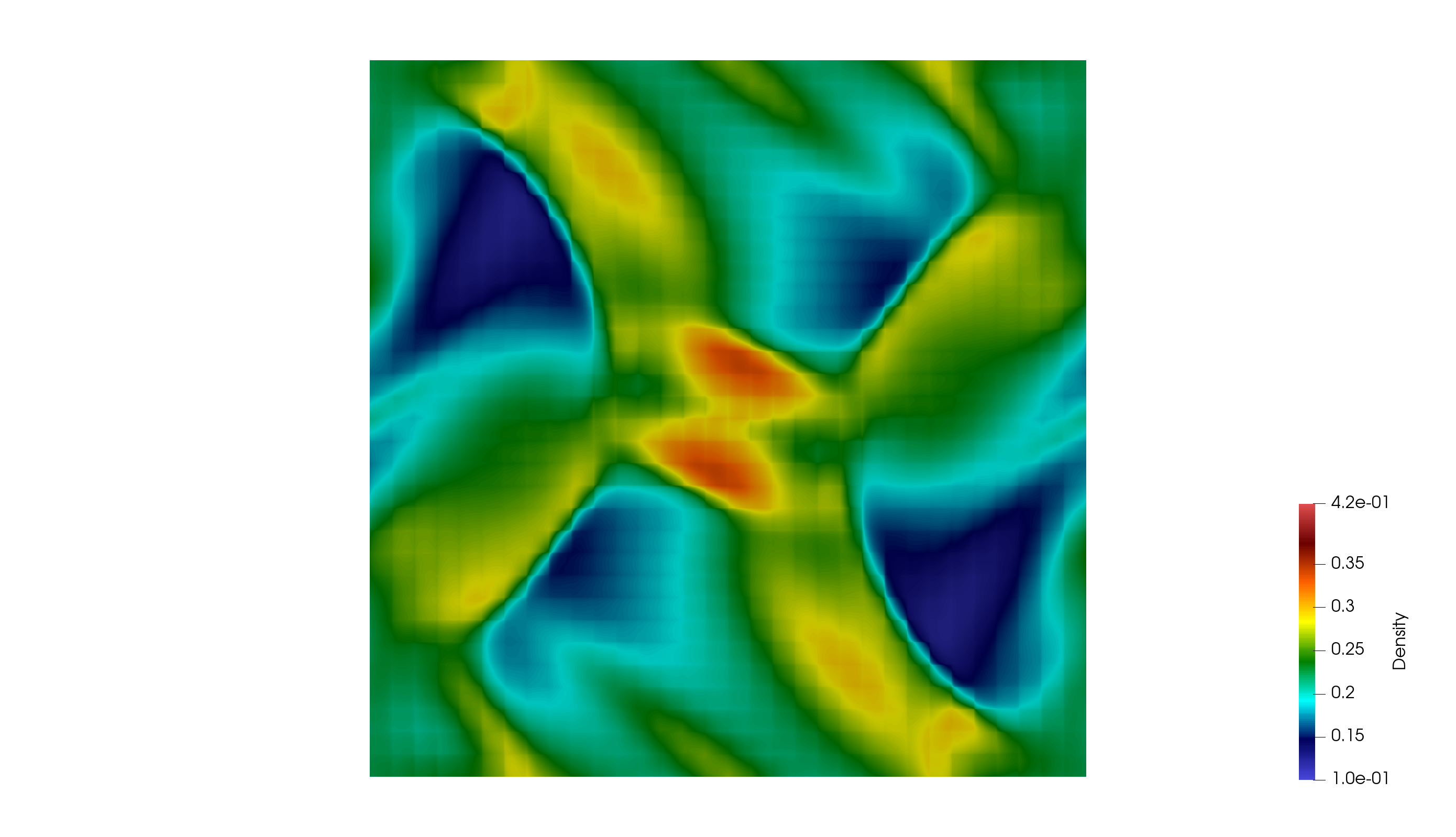}
	\caption{First-order FV}
\end{subfigure}
\begin{subfigure}[b]{0.32\linewidth}
	\includegraphics[trim=700 100 700 100,clip,width=\linewidth]{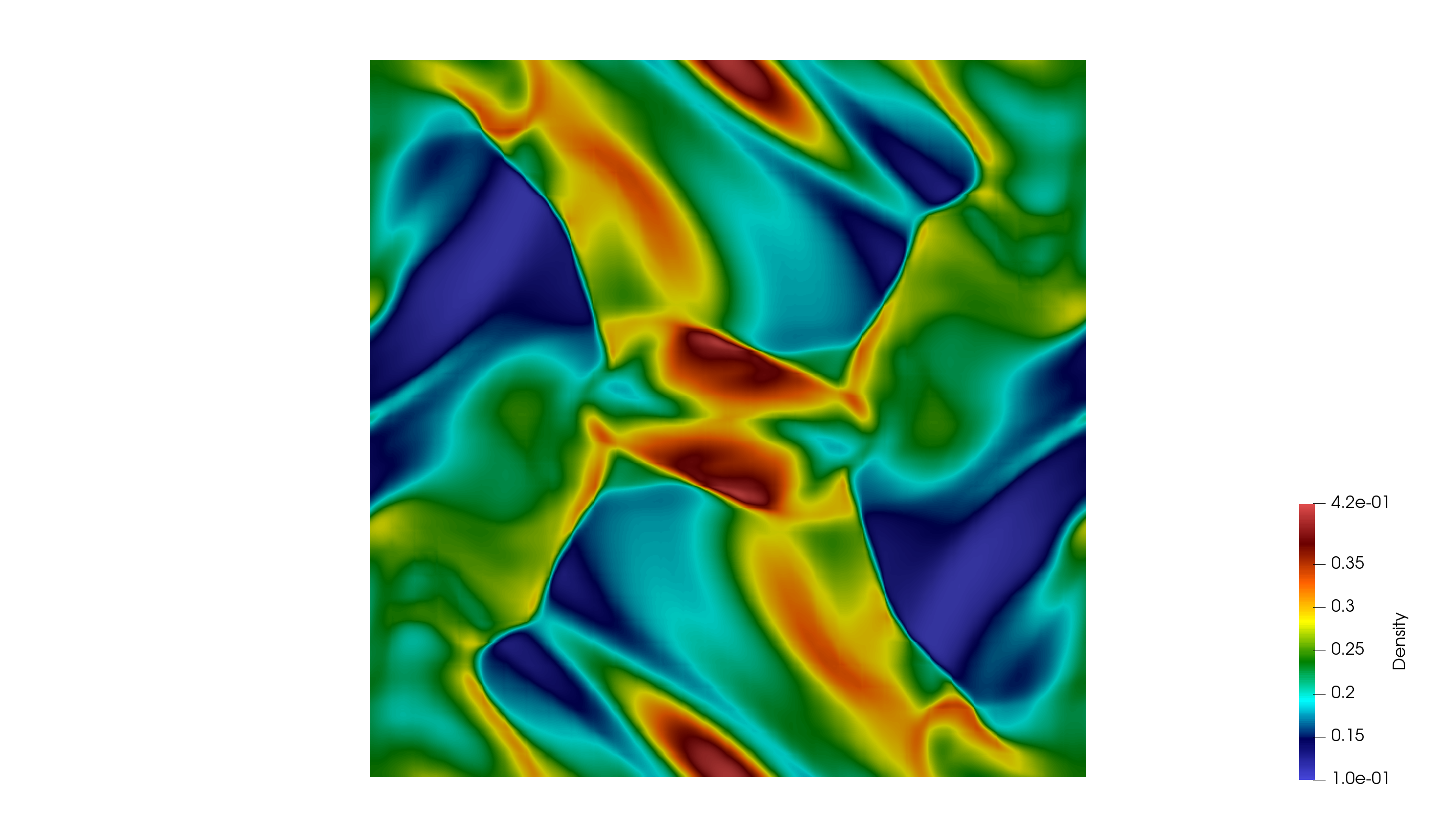}
	\caption{TVD-ES (no boundary reconstruction)}
\end{subfigure}
\\
\begin{subfigure}[b]{0.32\linewidth}
	\includegraphics[trim=700 100 700 100,clip,width=\linewidth]{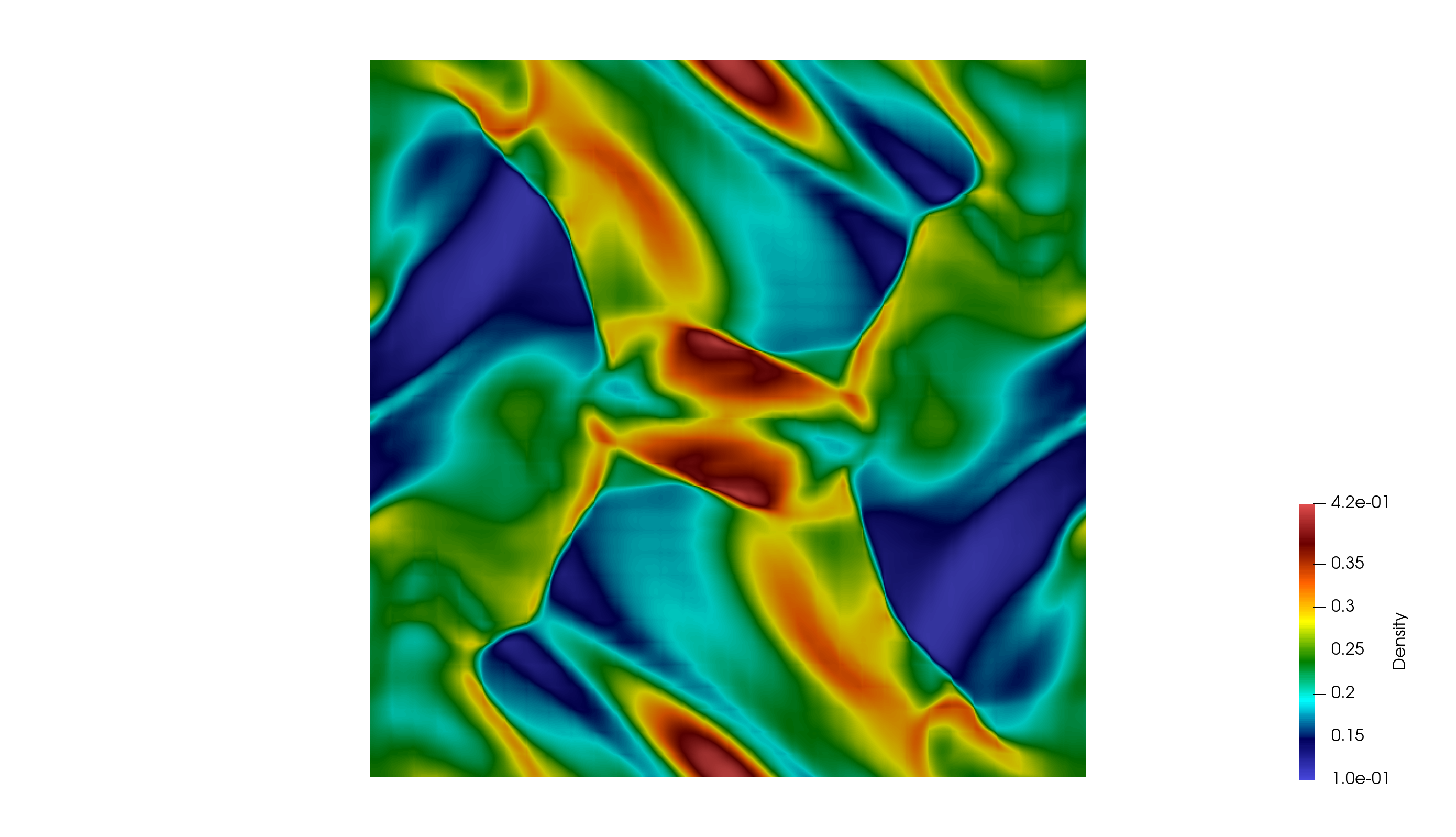}
	\caption{TVD-ES (central reconstruction)}
\end{subfigure}
\begin{subfigure}[b]{0.32\linewidth}
	\includegraphics[trim=700 100 700 100,clip,width=\linewidth]{figs/02_Orszag-Tang/DOF0256N7_t_0-50_Density_Alpha1_FjodES_neighborRecons.png}
	\caption{TVD-ES (neighbor reconstruction)}
\end{subfigure}
\caption{Density patterns at $t=0.5$ obtained with $256^2$ degrees of freedom and $N=7$ using the pure FV first-order method and different variants of the pure FV TVD-ES method. To obtain a pure FV discretization, we set $\alpha = 1$.}
\label{fig:OT_DensityFirstVsTVD}
\end{figure}

We solve this problem until $t=1$ with $256^2$, $512^2$ and $1024^2$ degrees of freedom, with the polynomial degrees $N=3$ and $N=7$, using the first-order and the TVD-ES shock-capturing methods introduced in Sections \ref{sec:HennemannGoesMHD} and \ref{sec:TVD-ES}, respectively.

We first study how both shock capturing methods perform in the pure FV limit, i.e. $\alpha=1$.
Figure \ref{fig:OT_DensityFirstVsTVD} shows a comparison of the the pure FV first-order method with different variants of the pure FV TVD-ES: (b) the scheme that does not use a reconstruction on the subcell boundaries \eqref{eq:TVDrecons_NoRecons}, (c) the scheme that uses a central reconstruction on the subcell boundaries \eqref{eq:TVDrecons_Central}, and (c) the scheme that uses the neighbor elements' state to reconstruct the solution on the subcell boundaries \eqref{eq:TVDrecons_Bound}.
For this test, we use $256^2$ DOFs, $N=7$.
The TVD-ES method provides an increased resolution and a reduction of the artifacts that the first-order FV method causes.
Moreover, it is worth pointing out that the scheme that does not use a boundary reconstruction procedure delivers the best results.
For this reason, we use the TVD-ES method without boundary reconstruction in the remaining parts of the results section.

\begin{figure}[htb!]
	\centering
	\includegraphics[trim=450 0 450 1050,clip,width=0.36\linewidth]{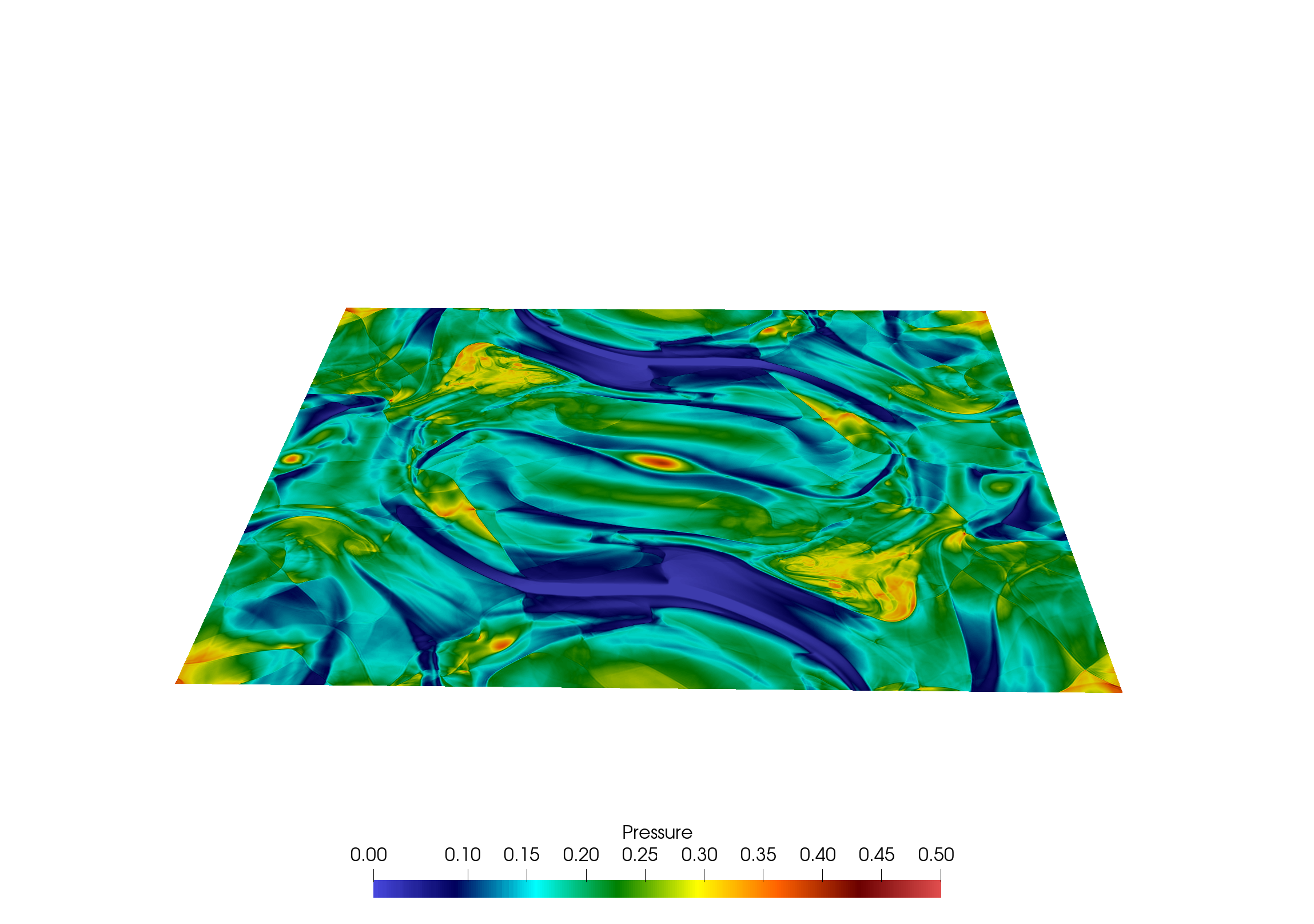}
	\includegraphics[trim=600 0 600 1400,clip,width=0.36\linewidth]{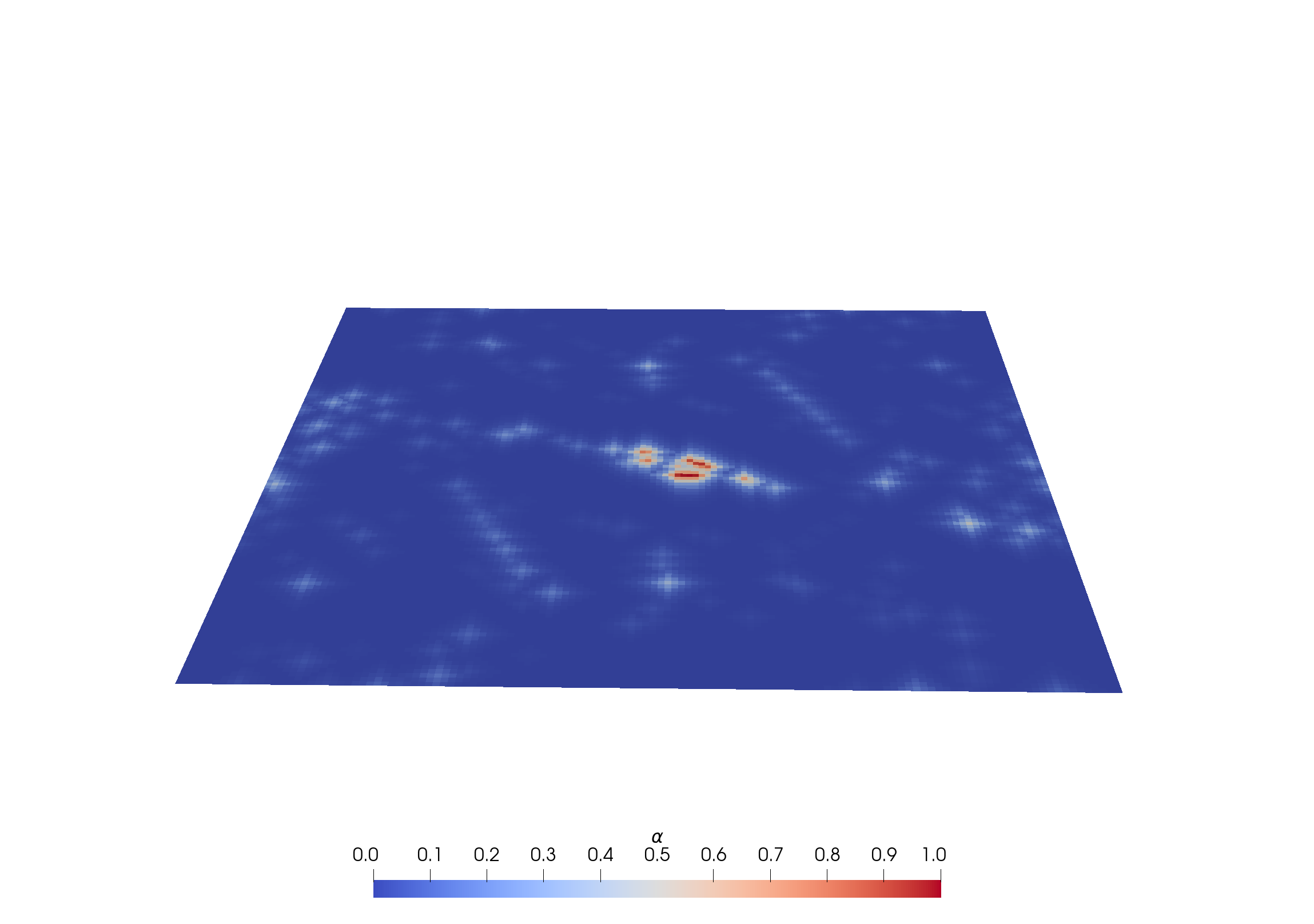}
		\includegraphics[trim=700 100 700 100 ,clip,width=0.36\linewidth]{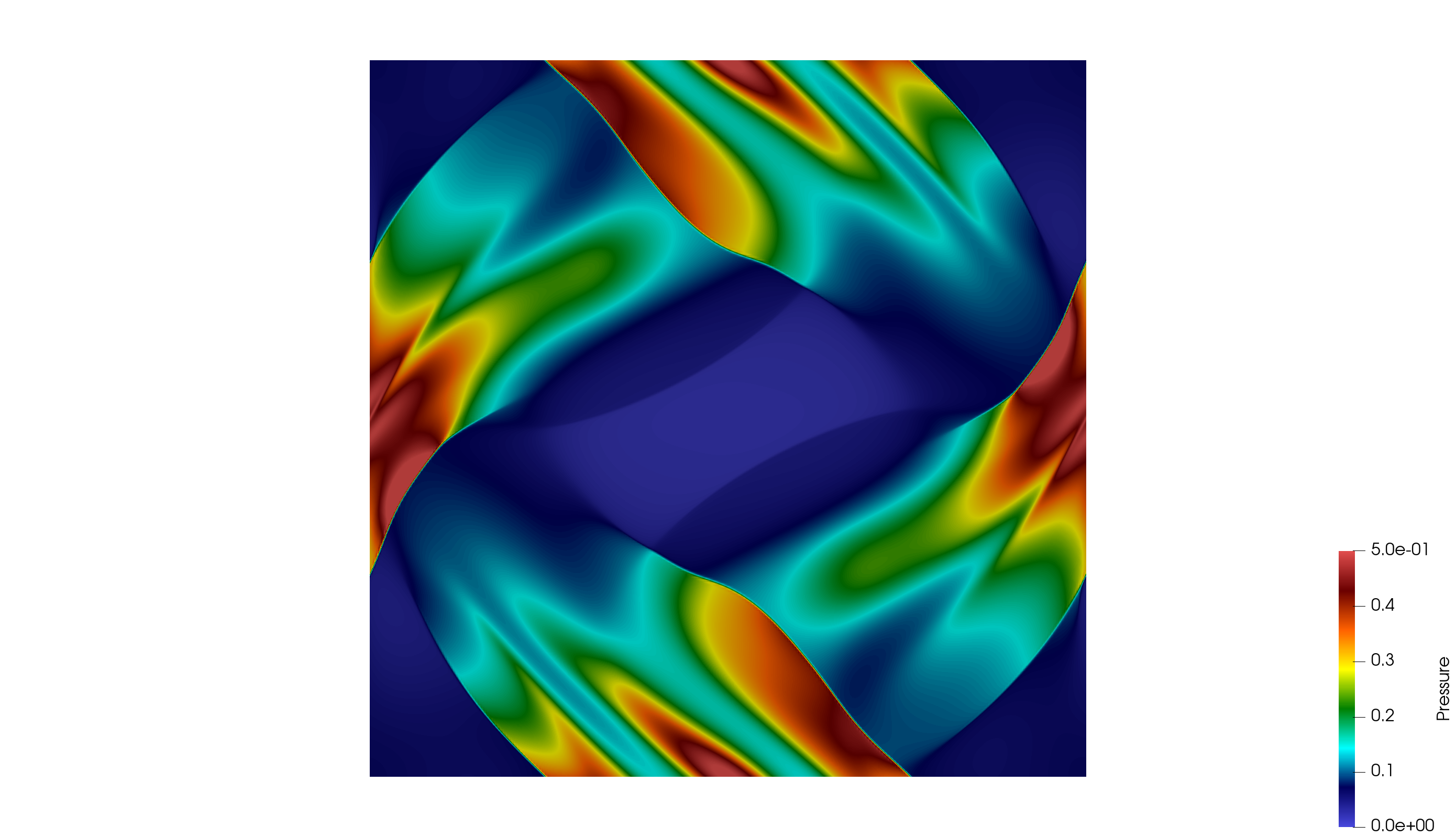}
		\includegraphics[trim=700 100 700 100,clip,width=0.36\linewidth]{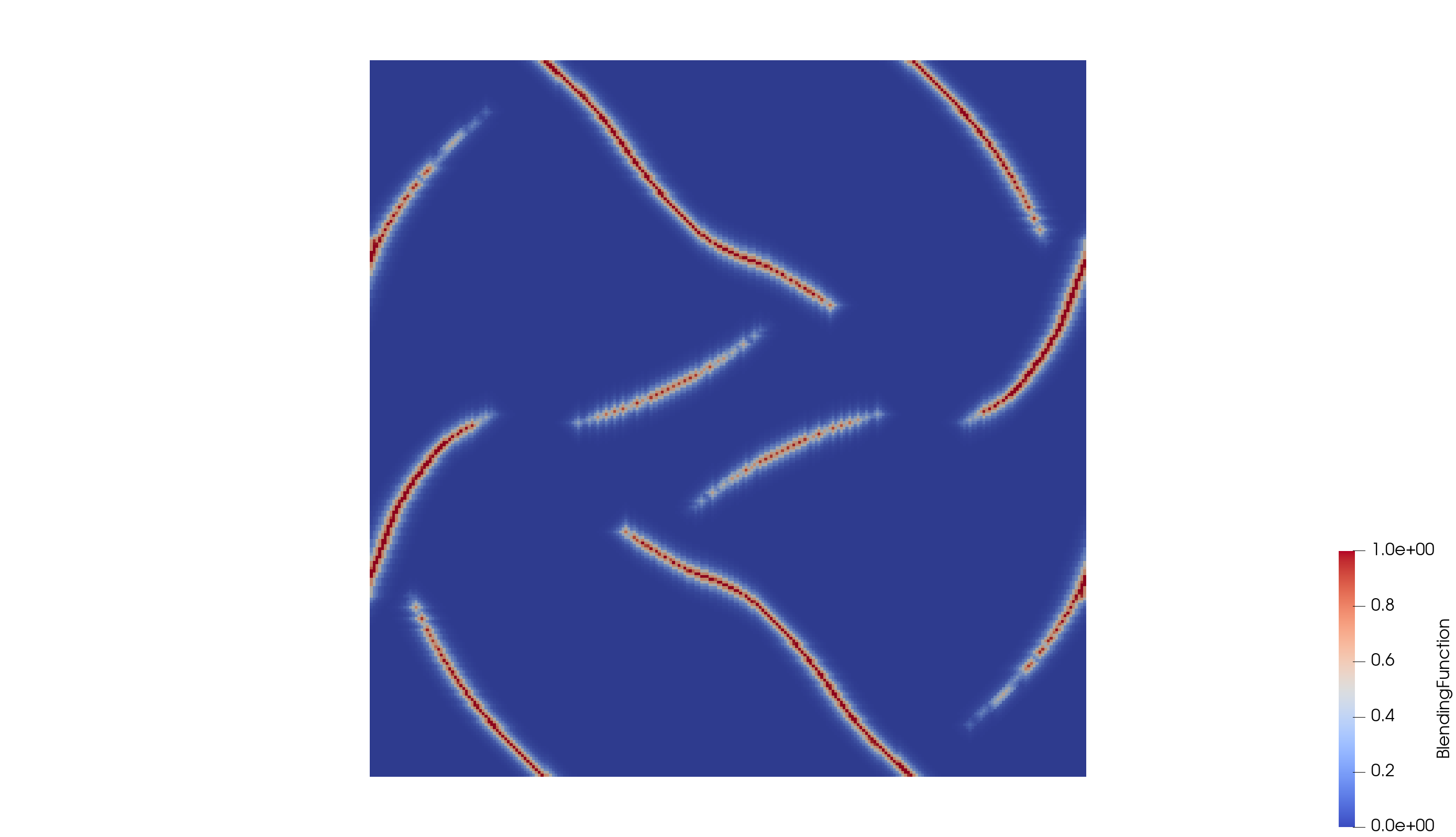}
		\includegraphics[trim=700 100 700 100,clip,width=0.36\linewidth]{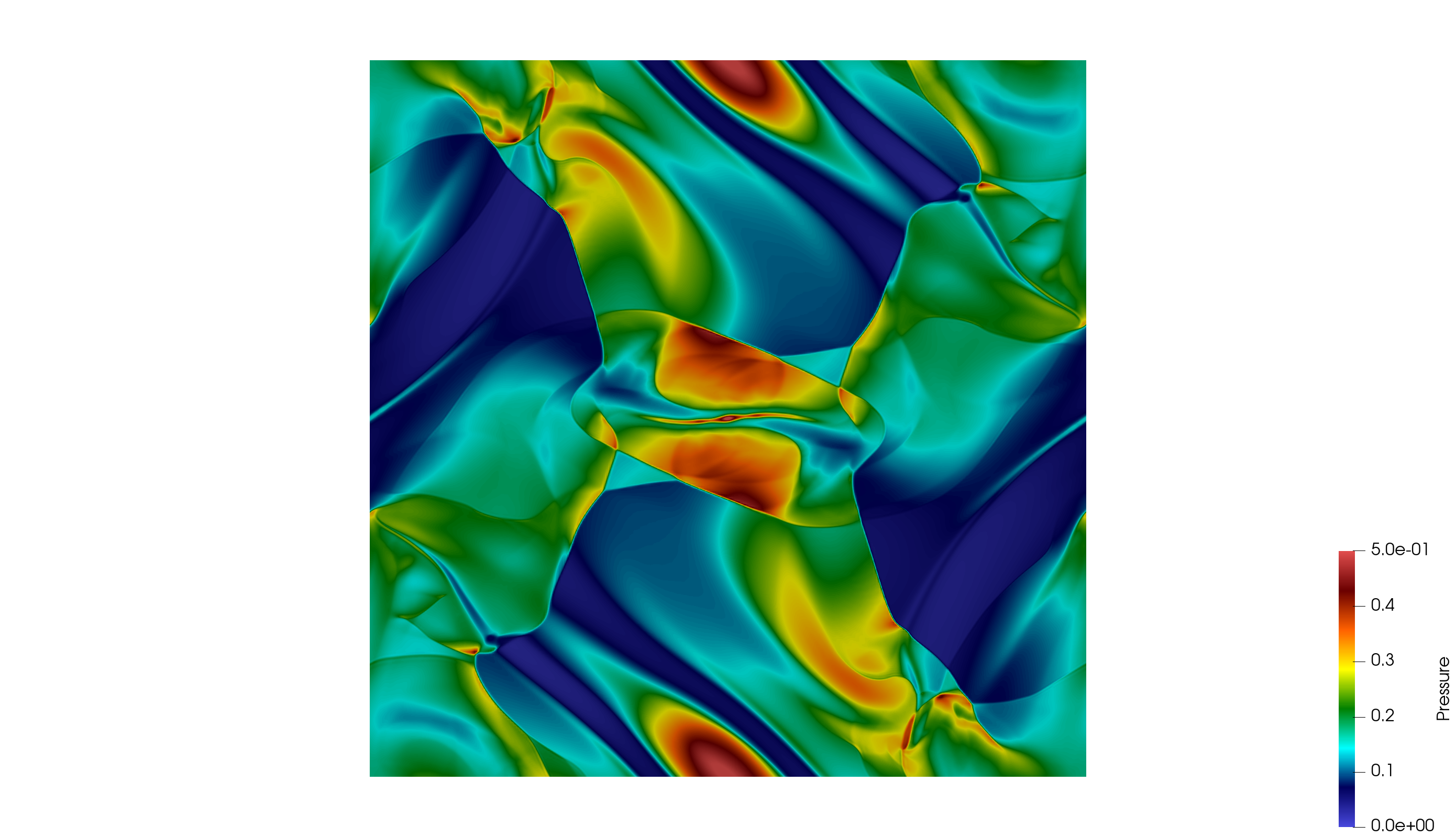}
		\includegraphics[trim=700 100 700 100,clip,width=0.36\linewidth]{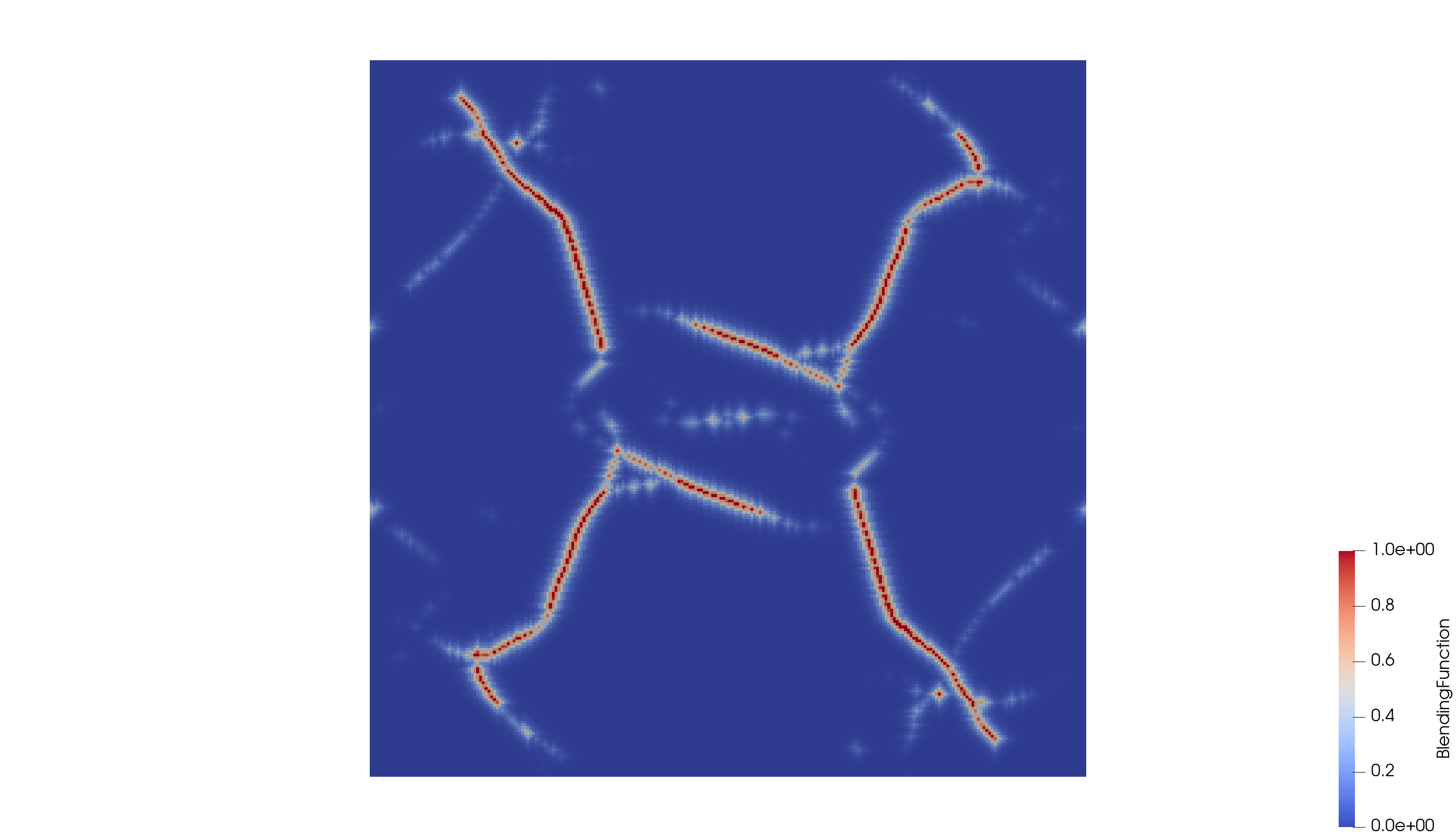}
		\includegraphics[trim=700 100 700 100,clip,width=0.36\linewidth]{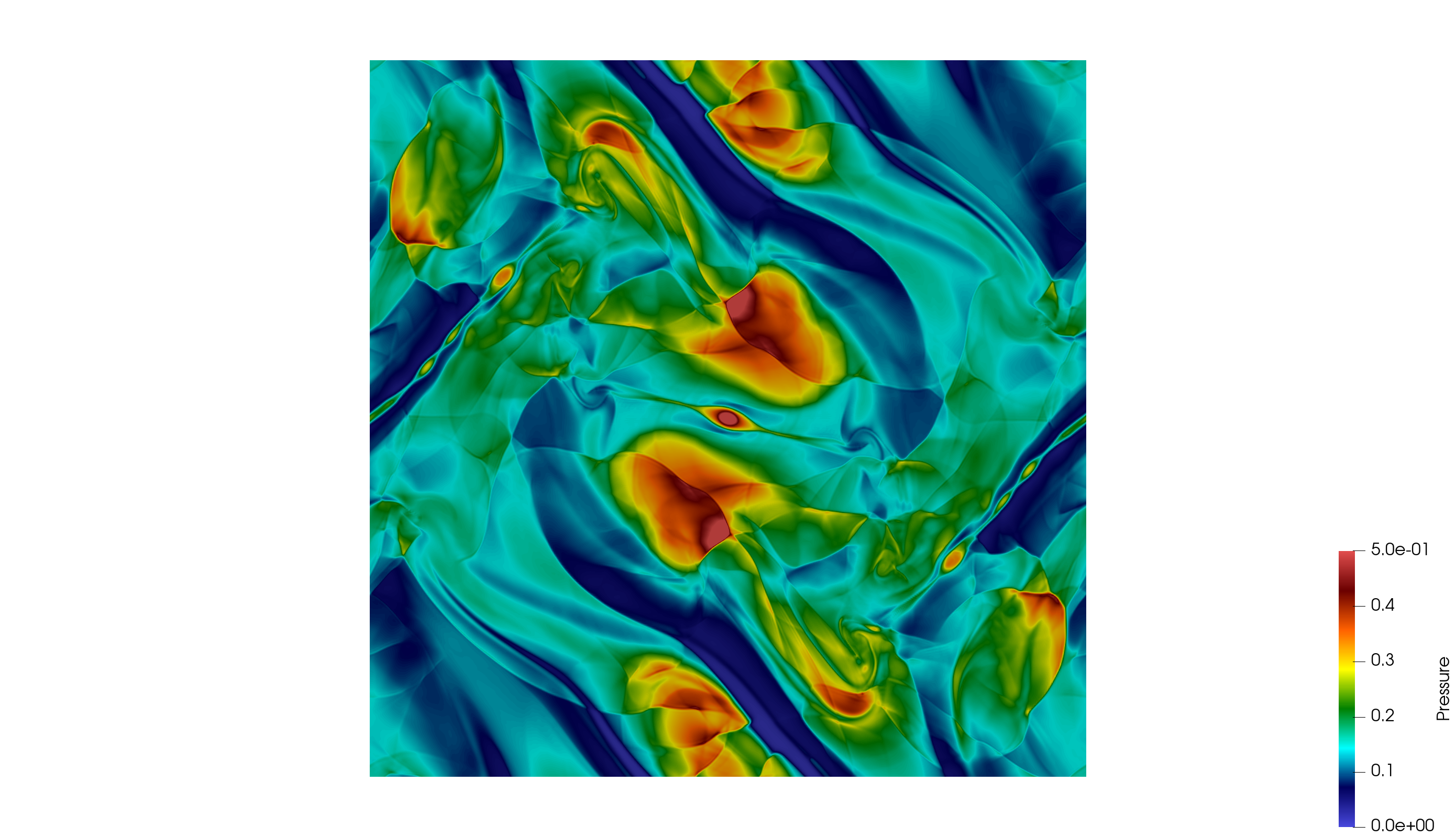}
		\includegraphics[trim=700 100 700 100,clip,width=0.36\linewidth]{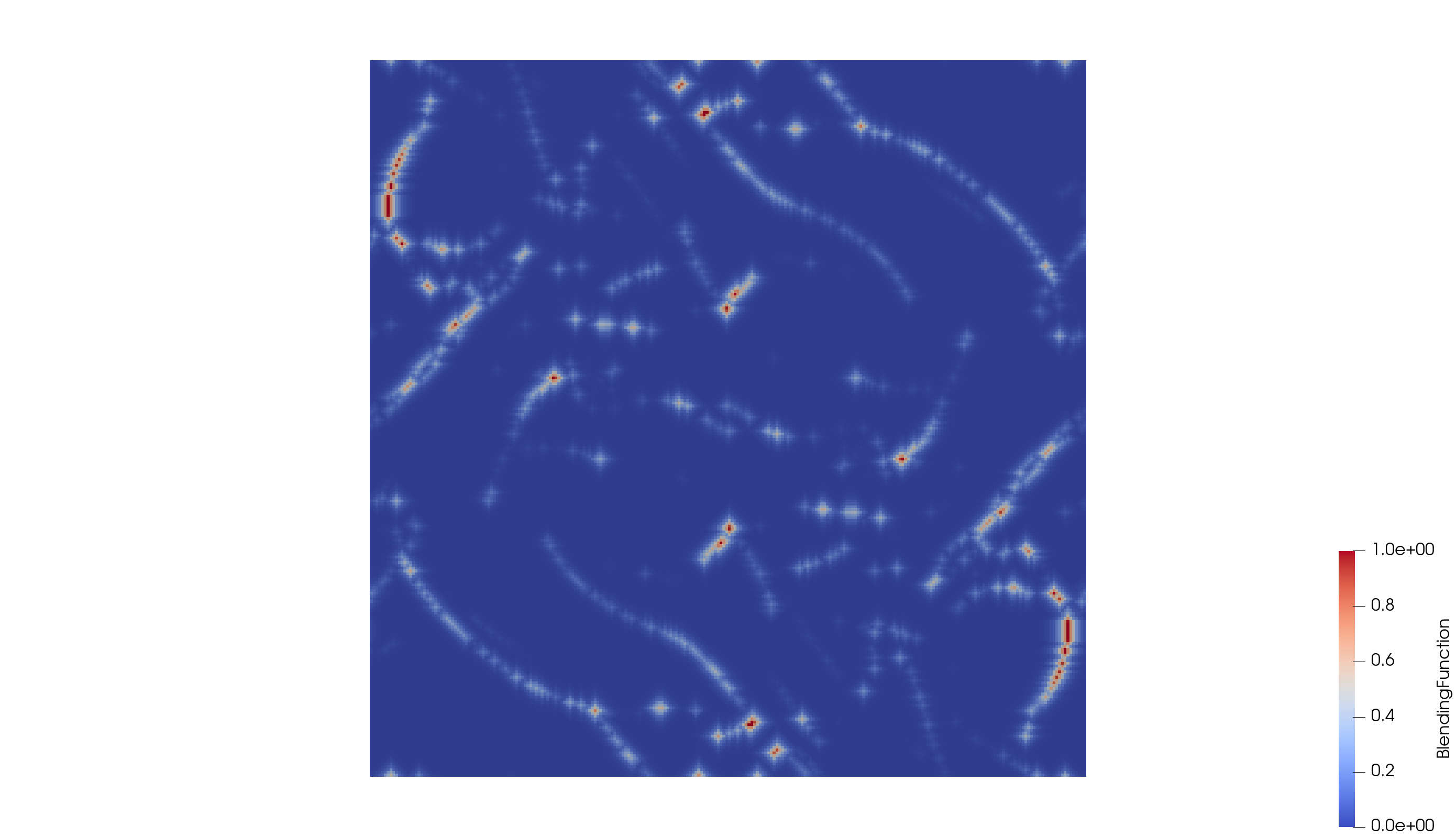}
	\caption{Evolution of the Orszag-Tang vortex problem with the TVD-ES shock capturing method. 
	We show the pressure and the blending coefficient for the simulation with $1024^2$ degrees of freedom and $N=3$ for $t=0.25$ (top), $t=0.50$ (middle), $t=0.75$ (bottom).
	The blending coefficient is computed with the indicator described in Section \ref{sec:Indicator} using the gas pressure as indicator quantity, $\epsilon=p$.}
	\label{fig:OT1024_N3}
\end{figure}

We now study the performance of the scheme that blends the subcell FV method with our high-order DGSEM discretization.
To do that, we use the shock indicator described in Section \ref{sec:Indicator} with the gas pressure as indicator quantity, $\epsilon=p$, as it showed to provide enough robustness for the simulations of this test.

Figure \ref{fig:OT1024_N3} shows the evolution of the pressure and the blending coefficient, $\alpha$, for the simulation that uses the TVD-ES method with $1024^2$ DOFs and $N=3$.
The initially smooth solution quickly develops shocks that travel freely through the domain, as can be seen at $t=0.25$.
The shock indicator locates the presence of shocks and applies an appropriate amount of localized stabilization.
As the shocks meet the periodic boundaries, they re-enter the domain and start to interact with each other, as is evident at $t=0.5$.
As the simulation advances, the shock-shock interaction produces zones with increased vorticity and mixing, which lead to the transition to MHD turbulence.

\begin{figure}[htb!]
	\centering
	\includegraphics[trim=450 0 450 1050,clip,width=0.36\linewidth]{figs/02_Orszag-Tang/Legend_Pressure.png}
	\includegraphics[trim=600 0 600 1400,clip,width=0.36\linewidth]{figs/02_Orszag-Tang/Legend_Alpha.png}
		\includegraphics[trim=700 100 700 100 ,clip,width=0.36\linewidth]{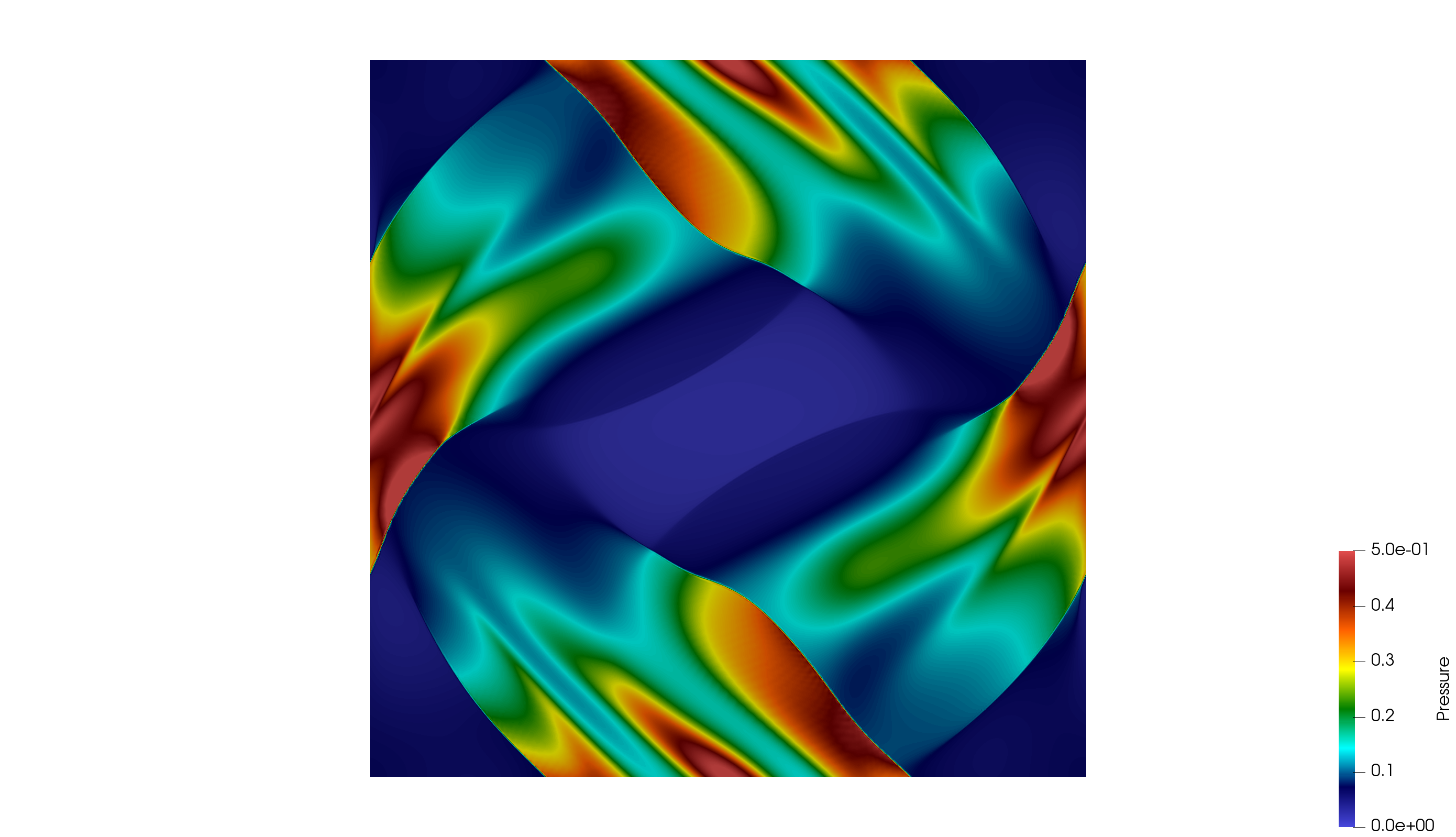}
		\includegraphics[trim=700 100 700 100 ,clip,width=0.36\linewidth]{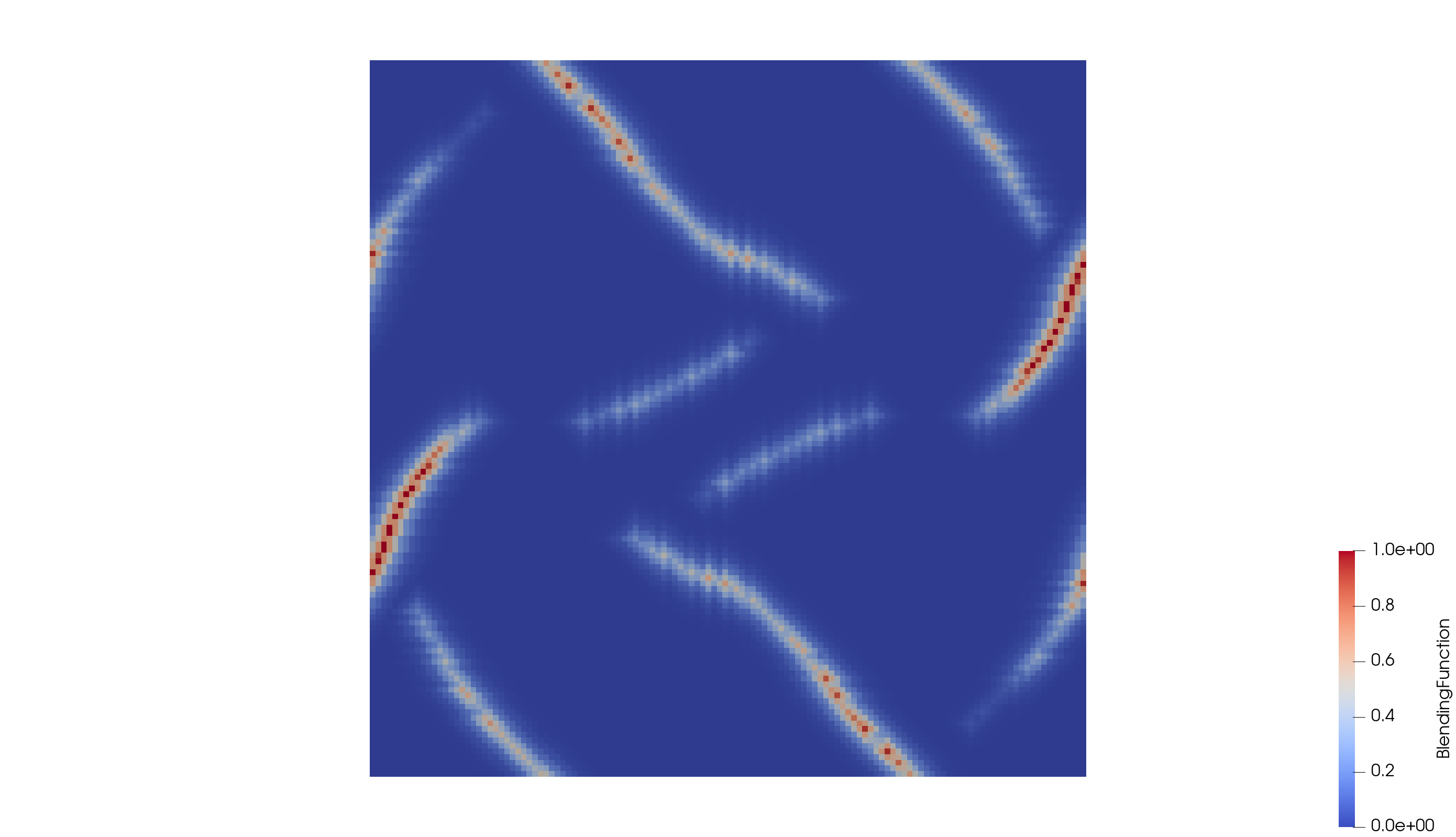}
		\includegraphics[trim=700 100 700 100 ,clip,width=0.36\linewidth]{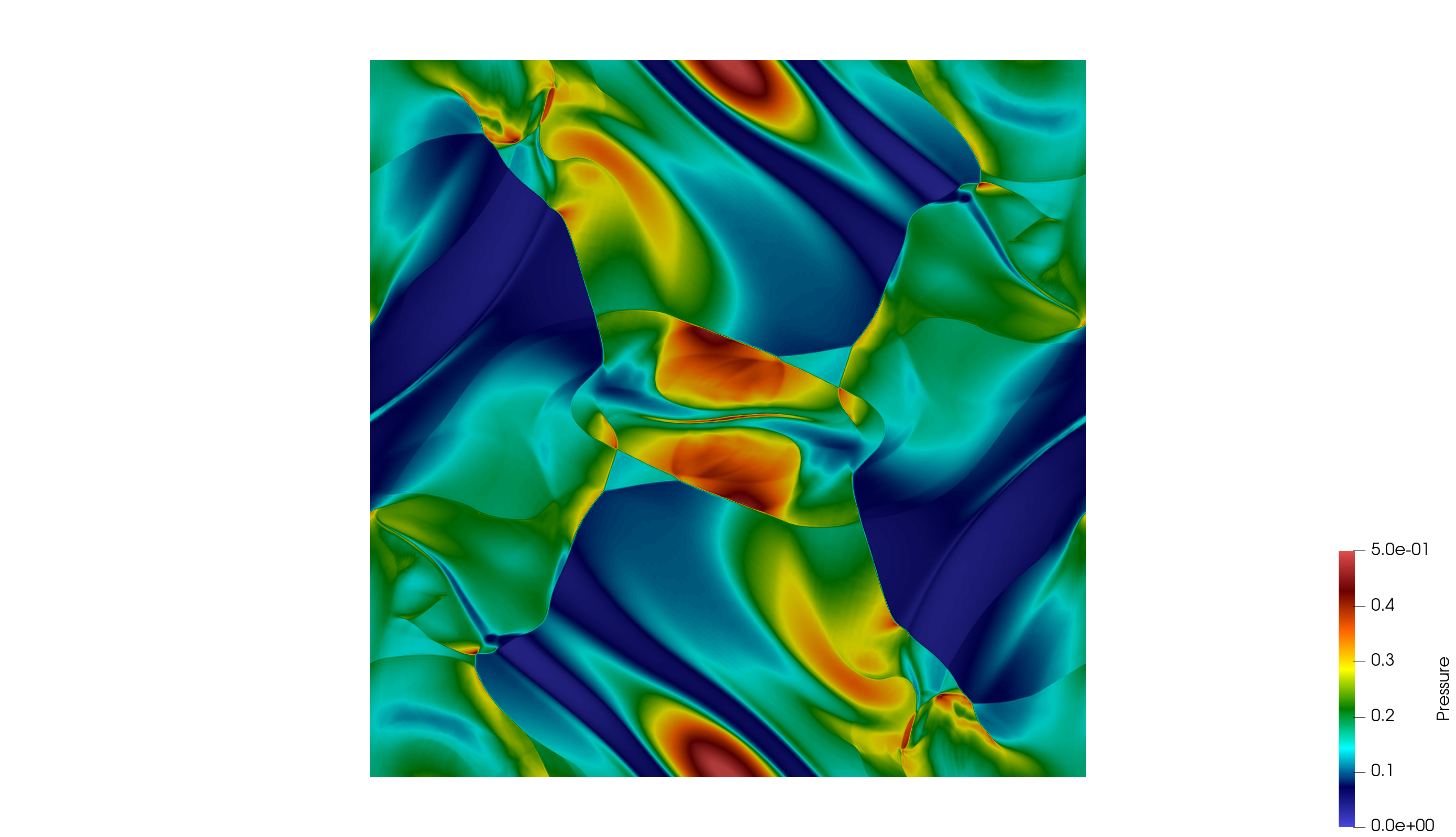}
		\includegraphics[trim=700 100 700 100 ,clip,width=0.36\linewidth]{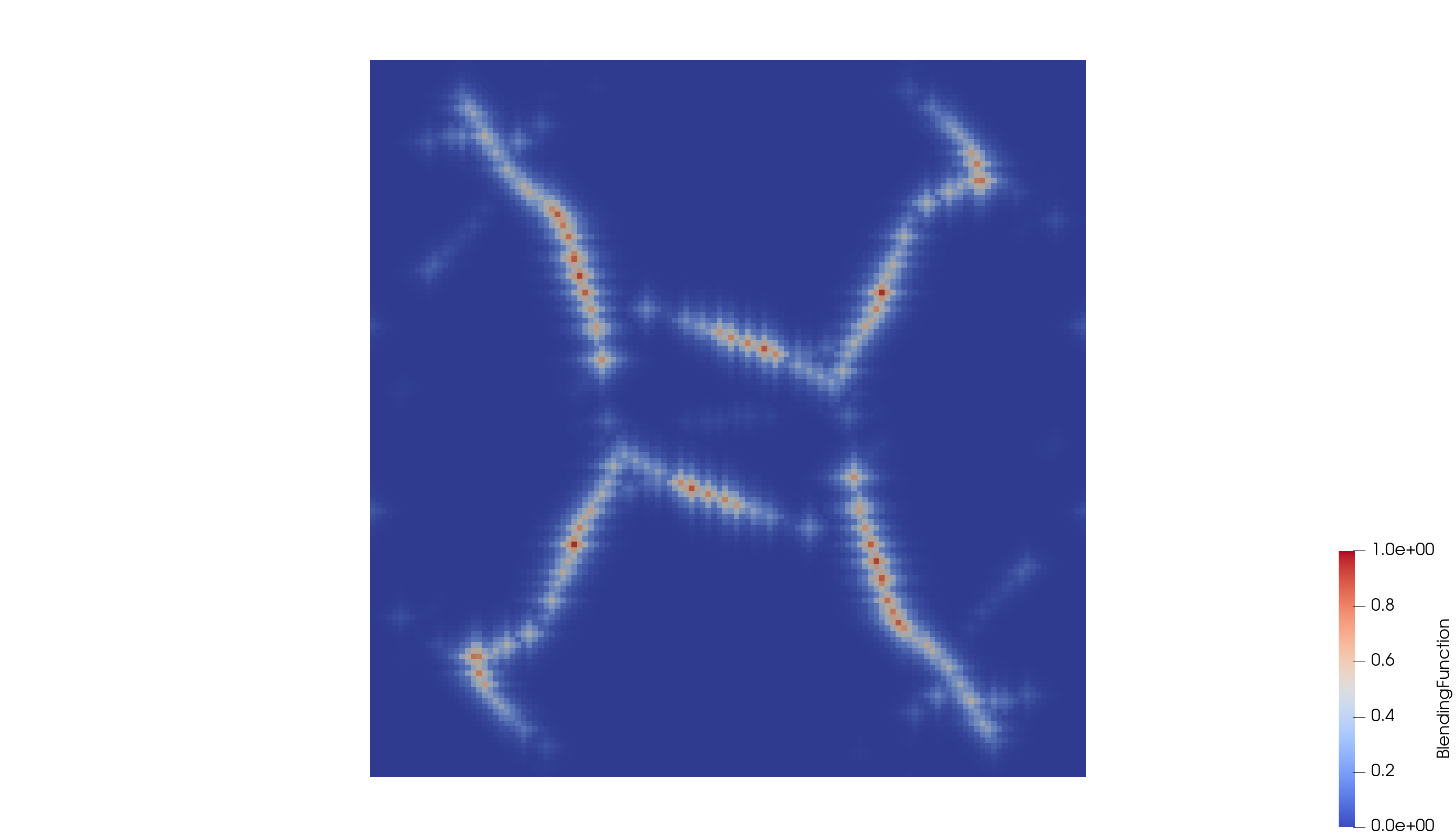}
		\includegraphics[trim=700 100 700 100 ,clip,width=0.36\linewidth]{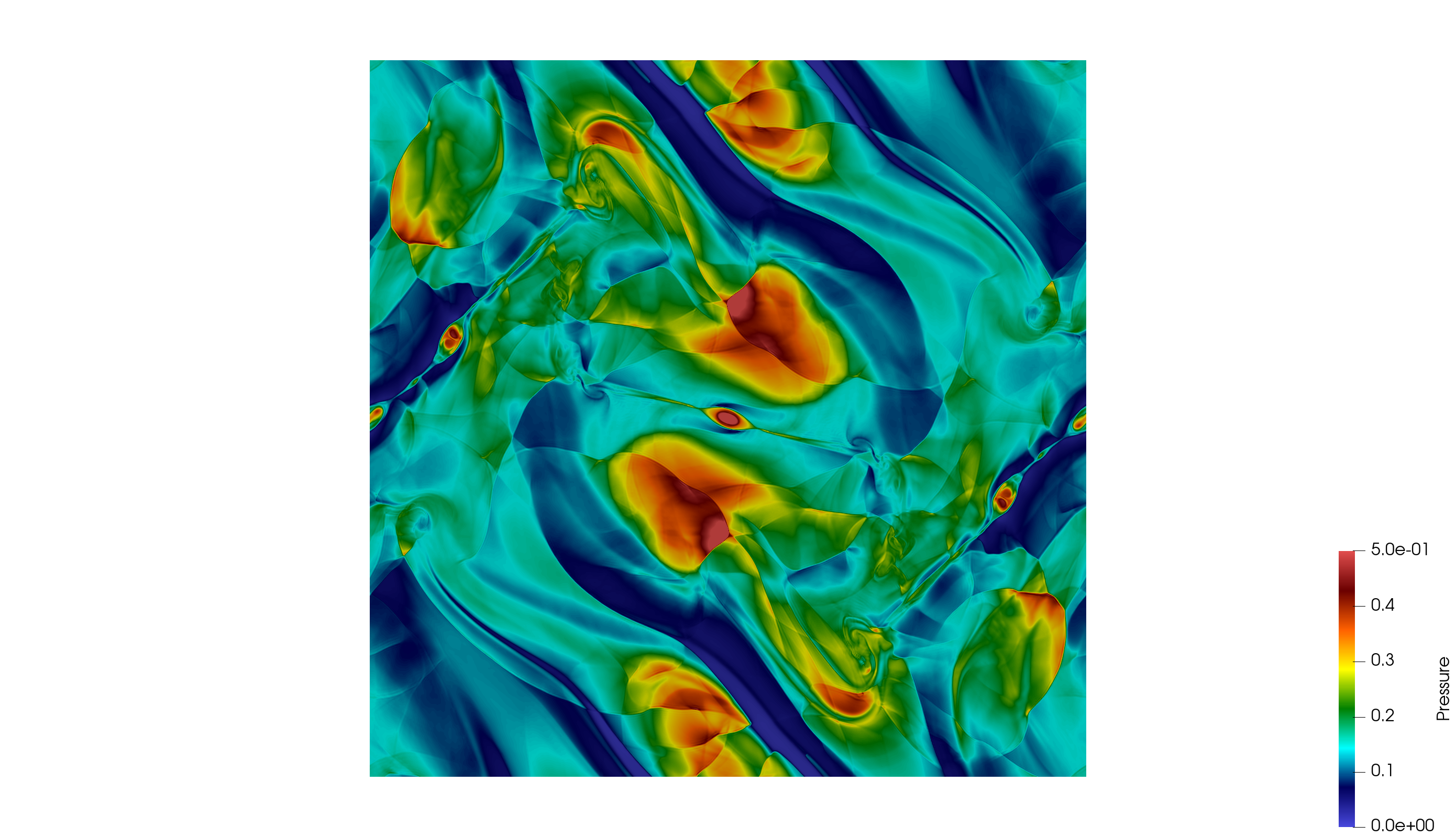}
		\includegraphics[trim=700 100 700 100 ,clip,width=0.36\linewidth]{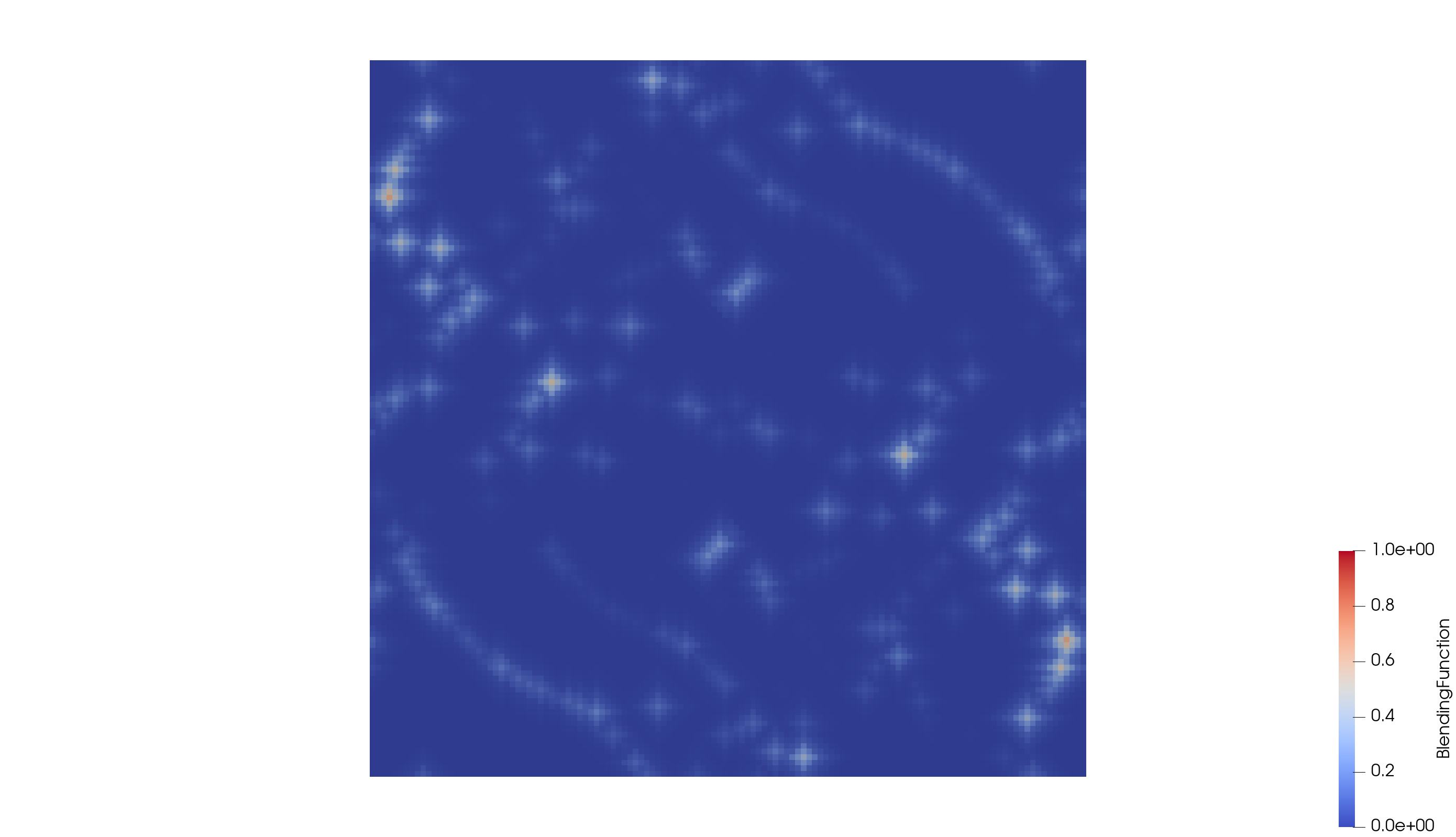}
	\caption{Evolution of the Orszag-Tang vortex problem with the TVD-ES shock capturing method. 
	We show the pressure and the blending coefficient for the simulation with $1024^2$ degrees of freedom and $N=7$ for $t=0.25$ (top), $t=0.50$ (middle), $t=0.75$ (bottom).
	The blending coefficient is computed with the indicator described in Section \ref{sec:Indicator} using the gas pressure as indicator quantity, $\epsilon=p$.}
	\label{fig:OT1024_N7}
\end{figure}

Figure \ref{fig:OT1024_N7} shows the evolution of the pressure and the blending coefficient, $\alpha$, for the simulation that uses the TVD-ES method with $1024^2$ DOFs and $N=7$.
As in Figure \ref{fig:OT1024_N3}, the shocks are correctly identified by the indicator, and a proportionate amount of stabilization is applied.
Note that the value of the blending coefficient, $\alpha$ is in general lower than in the $N=3$ case, as the polynomials of degree $N=7$ can better represent steep gradients.
It can be clearly seen that in the $N=7$ simulation, more and smaller vortices appear at $t=0.75$ than in the $N=3$ case.
This shows the advantage of using a high polynomial degree.

\begin{figure}[htb]
\centering
\begin{subfigure}[b]{0.9\linewidth}
	\includegraphics[width=\linewidth]{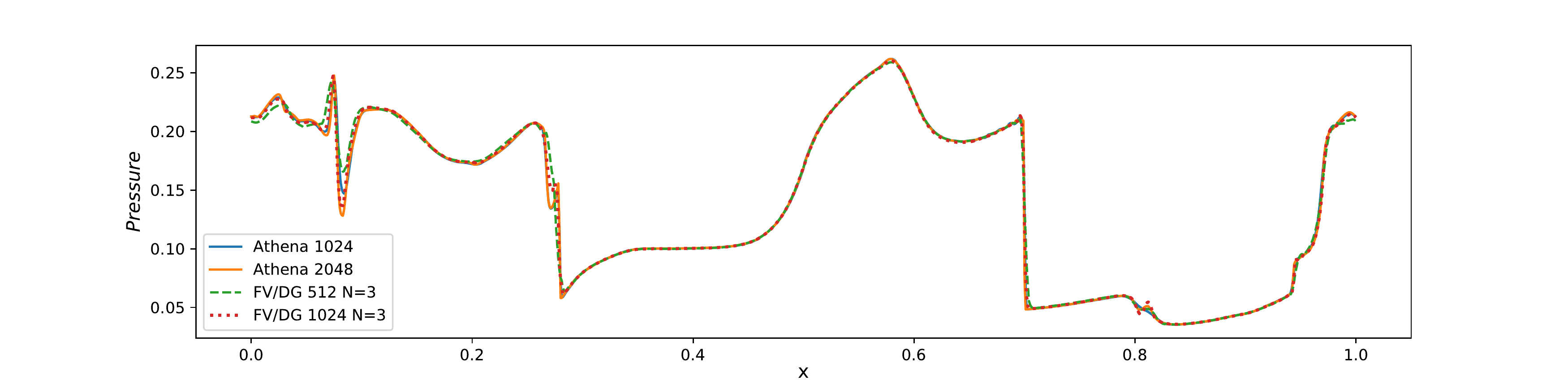}
	\caption{$y=0.3125$}
\end{subfigure}
\begin{subfigure}[b]{0.9\linewidth}
	\includegraphics[width=\linewidth]{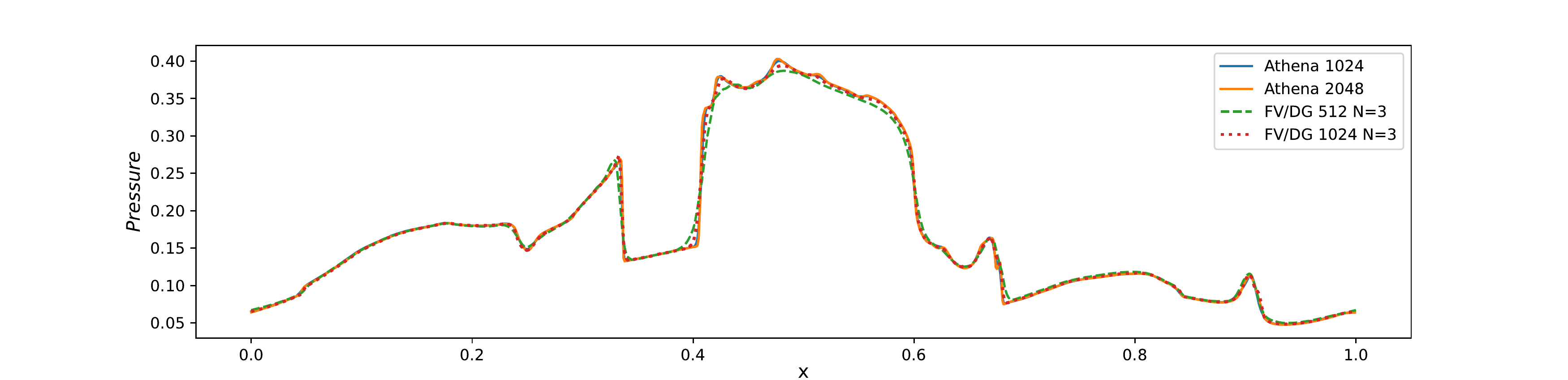}
	\caption{$y=0.4277$}
\end{subfigure}
\caption{Slices of dimensionless pressure for the Orszag-Tang vortex at $t=0.5$ for different resolutions ($N=3$) and comparison with the Athena solver.
The legend shows the number of degrees of freedom per direction.}
\label{fig:OT_DataLinesN3}
%
\begin{subfigure}[b]{0.9\linewidth}
	\includegraphics[width=\linewidth]{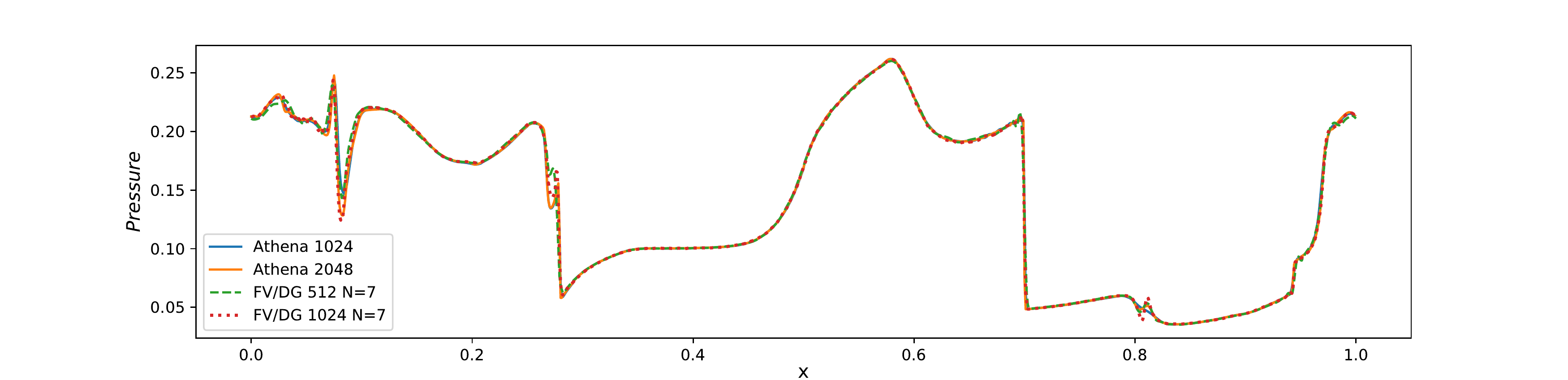}
	\caption{$y=0.3125$}
\end{subfigure}
\begin{subfigure}[b]{0.9\linewidth}
	\includegraphics[width=\linewidth]{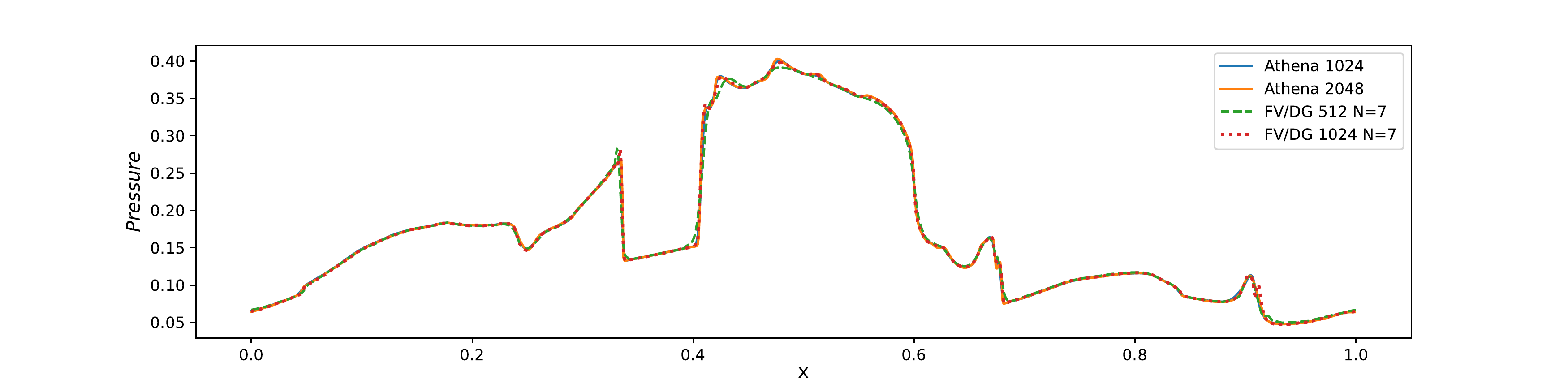}
	\caption{$y=0.4277$}
\end{subfigure}
\caption{Slices of dimensionless pressure for the Orszag-Tang vortex at $t=0.5$ for different resolutions ($N=7$) and comparison with the Athena solver.
The legend shows the number of degrees of freedom per direction.}
\label{fig:OT_DataLinesN7}
\end{figure}

In Figures \ref{fig:OT_DataLinesN3} and  \ref{fig:OT_DataLinesN7}, we show the pressure along two slice cuts, at $y=0.3125$ and $y=0.4277$, for different resolutions that were tested with the TVD-ES shock-capturing method.
In both figures, we also show the pressure on the slice cuts that is obtained with the Athena code\footnote{We used Athena public version available at \url{https://github.com/PrincetonUniversity/athena-public-version}.} \cite{Stone2008} using $1024^2$ and $2048^2$ DOFs and a Finite Volume method with a second order reconstruction procedure on a uniform grid.
Athena is an open-source solver that uses the constrained transport technique to ensure the divergence-free condition on the magnetic field.
In this particular example, we used the Rusanov Riemann solver that is implemented in Athena, such that the results are comparable.

In Figures \ref{fig:OT_DataLinesN3} and  \ref{fig:OT_DataLinesN7}, we can observe that our method converges to the Athena solution as the resolution is increased and that some features are better captured when using higher-order DGSEM method with $N=7$.

%

Finally, in Figure \ref{fig:OT_Entropy}, we show the evolution of the total mathematical entropy, $S_{\Omega}$, over time for the Orszag-Tang vortex test.
As expected, the schemes fulfill the second law of thermodynamics.
Since this case has no dissipation through viscosity/resistivity, the mathematical entropy stays constant at the beginning, when the solution is smooth, and then decreases, when shocks appear in the domain.

\begin{figure}[htb]
\centering
\includegraphics[width=0.6\linewidth]{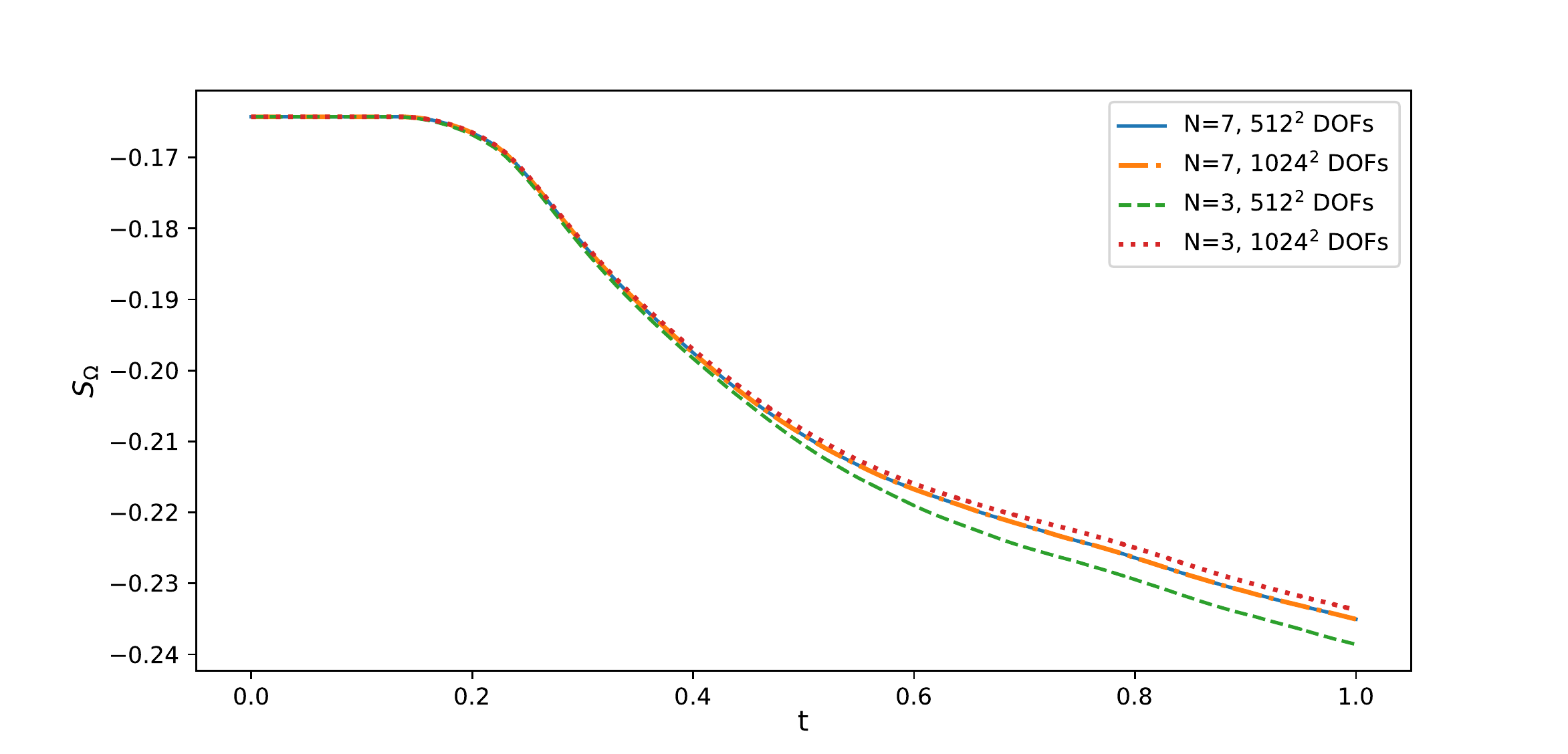}

\caption{Evolution of the total mathematical entropy, $S_{\Omega}$, over time for the Orszag-Tang vortex test.}
\label{fig:OT_Entropy}
\end{figure}

\pagebreak

\subsection{GEM Reconnection Challenge}

The GEM (Geospace Environmental Modeling) reconnection challenge was originally proposed by \citet{Birn2001} as a way of testing the robustness and accuracy of MHD codes and of comparing different MHD models (e.g. resistive MHD, Hall MHD, etc).

Magnetic reconnection refers to the phenomenon where oppositely directed magnetic field lines break and reconnect in a plasma, altering the magnetic field topology.
This process occurs in solar flares, in the Earth's magnetosphere and in plasma confinement devices, such as Tokamaks \cite{Helander2002,Ono2012}.
Magnetic reconnection can only occur in resistive plasmas, because in ideal (non-resistive) plasmas the magnetic field is \textit{frozen-in} to the plasma due to Ohm's law \cite{Ciuca2020}.
Besides the purely resistive effects, the rate of magnetic reconnection is governed by the Hall effect
\cite{Birn2001,Mignone2012,Sousa2015,Ciuca2020}.

Even though the resistive GLM-MHD equations are not the best model to simulate magnetic reconnection processes, as they do not include the Hall effect, our goal with this test, rather than an accurate description of the physical phenomenon, is twofold. 
First, we show that our proposed shock-capturing methods can be used when resistive terms are present in the MHD system.
Second, we reproduce the reconnection flux rates that are obtained with other resistive MHD codes.

The initial condition for the GEM reconnection challenge is a stationary current sheet,
\begin{align*}
\rho(x,y,t=0) &= \mathrm{sech}^2(y/l)+0.2 
& p(x,y,t=0) &=\frac{\rho B_0^2}{2} \\
v_1(x,y,t=0)  &= 0
& v_2(x,y,t=0) &= 0 \\
B_1(x,y,t=0)  &= B_0 \tanh(y/l) + B_1'
& B_2(x,y,t=0) &= B_2' \\
\psi(x,y,t=0), &= 0 &&
\end{align*}
where the magnetic field is perturbed with \cite{Birn2001}
\begin{align*}
B_1' &= -0.1 \frac{\pi}{L_y} \sin \left(\frac{\pi y}{L_y} \right) \cos \left( \frac{2 \pi x}{L_x} \right) \\
B_2' &= 0.1 \frac{2 \pi}{L_x} \sin \left( \frac{2 \pi x}{L_x} \right)  \cos \left(\frac{\pi y}{L_y} \right).
\end{align*}
This perturbation introduces a small magnetic island on the periodic boundary that triggers the reconnection process.

Without the FV stabilization, the entropy stable DGSEM crashes with this setup due to positivity issues.
We use $\epsilon=\rho p$ as the indicator quantity, as it showed to provide the necessary robustness for this test, and run the simulation from $t_0=0$ until $t_f=100$ with two different resolutions ($512 \times 256$ DOFs, and $1024 \times 512$ DOFs) and two different values for the resistivity ($\resistivity = 10^{-3}$ and $\resistivity = 5 \times 10^{-3}$).
The Navier-Stokes viscosity is set to $\viscosity = 0$ in this test in agreement with the literature on the topic \cite{Birn2001,Mignone2012,Sousa2015,Ciuca2020}.

\begin{figure}[htb!]
	\centering
	\includegraphics[trim=800 1450 800 0 ,clip,width=0.4\linewidth]{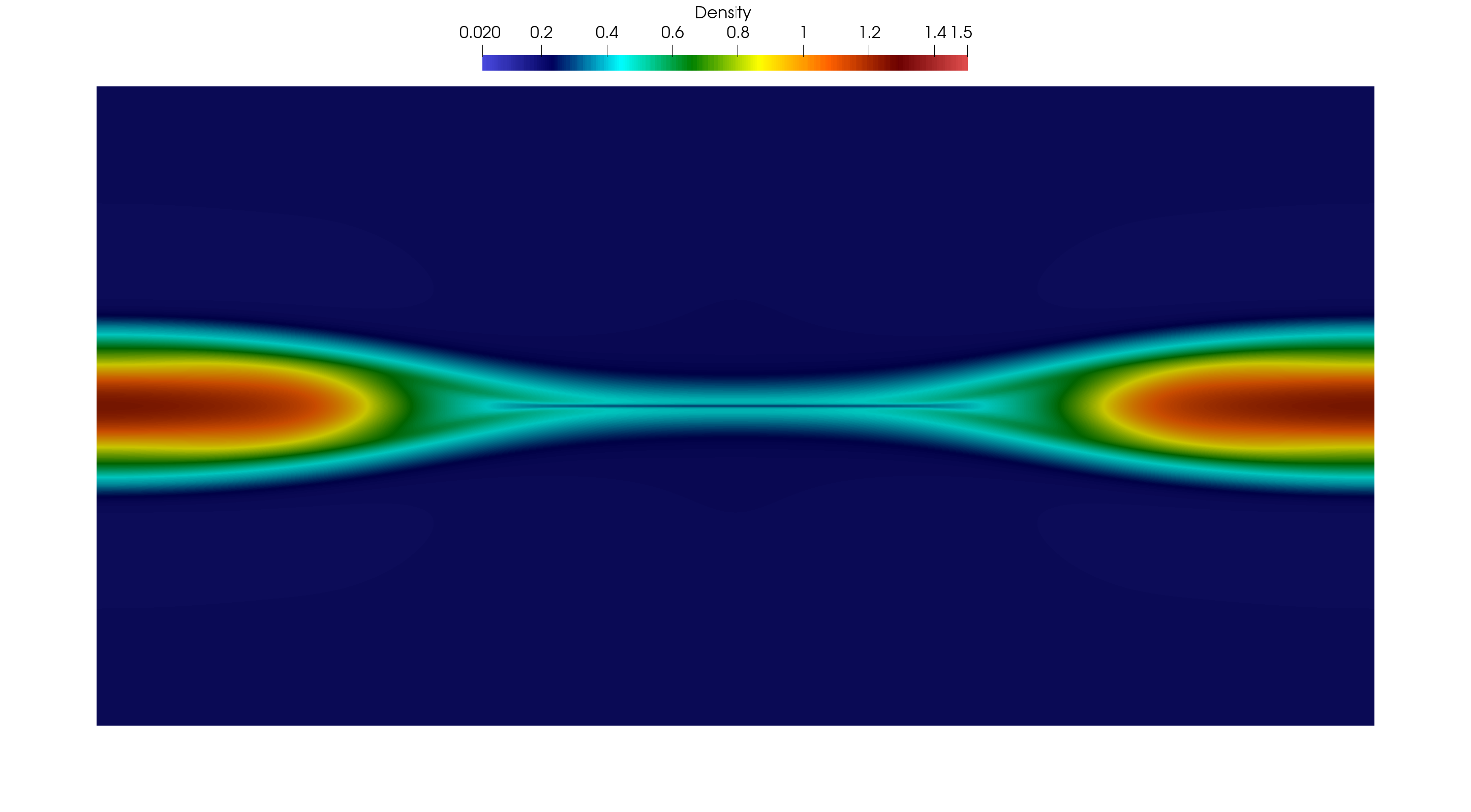}
	\includegraphics[trim=800 1450 800 0 ,clip,width=0.4\linewidth]{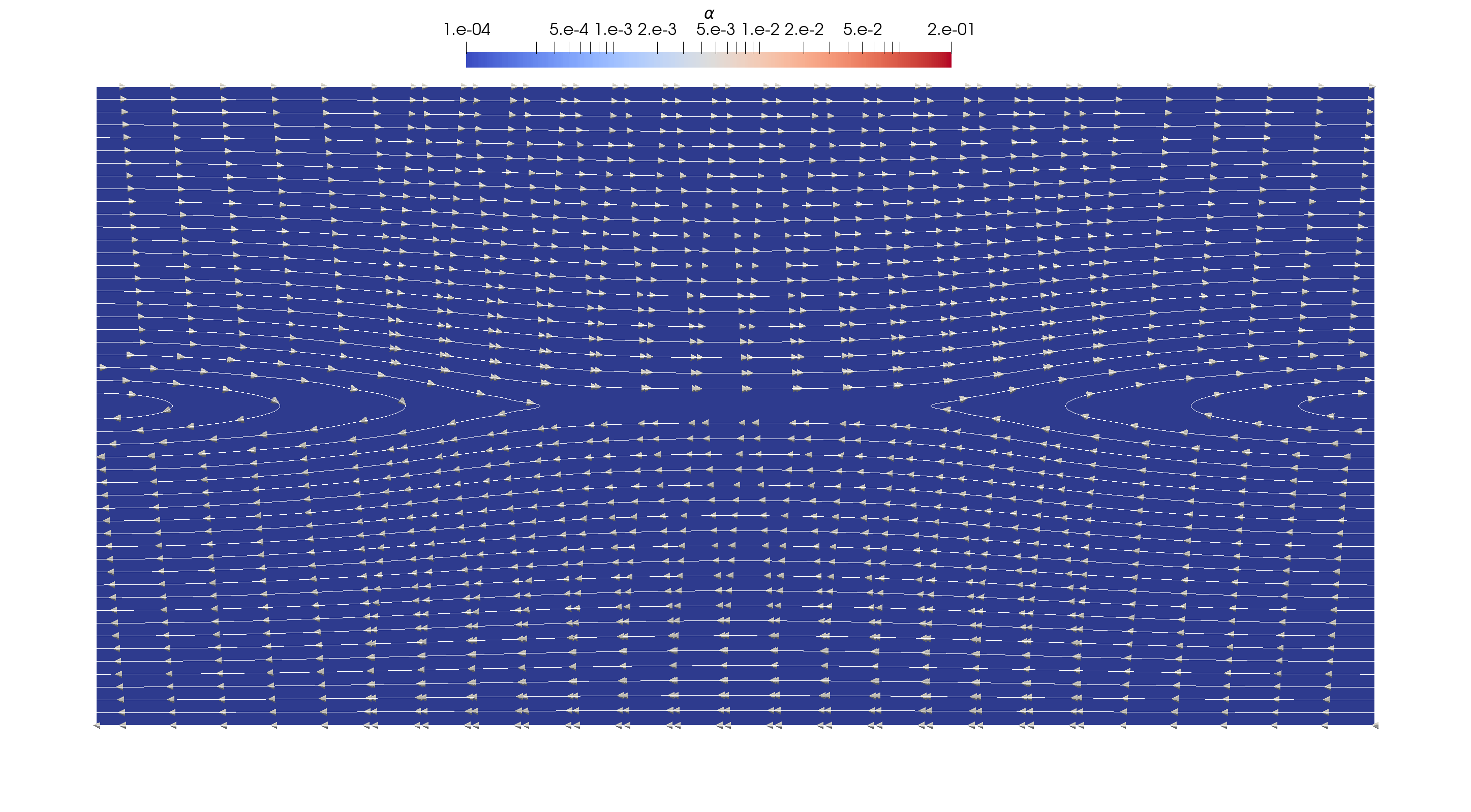}
		\includegraphics[trim=150 150 150 150 ,clip,width=0.4\linewidth]{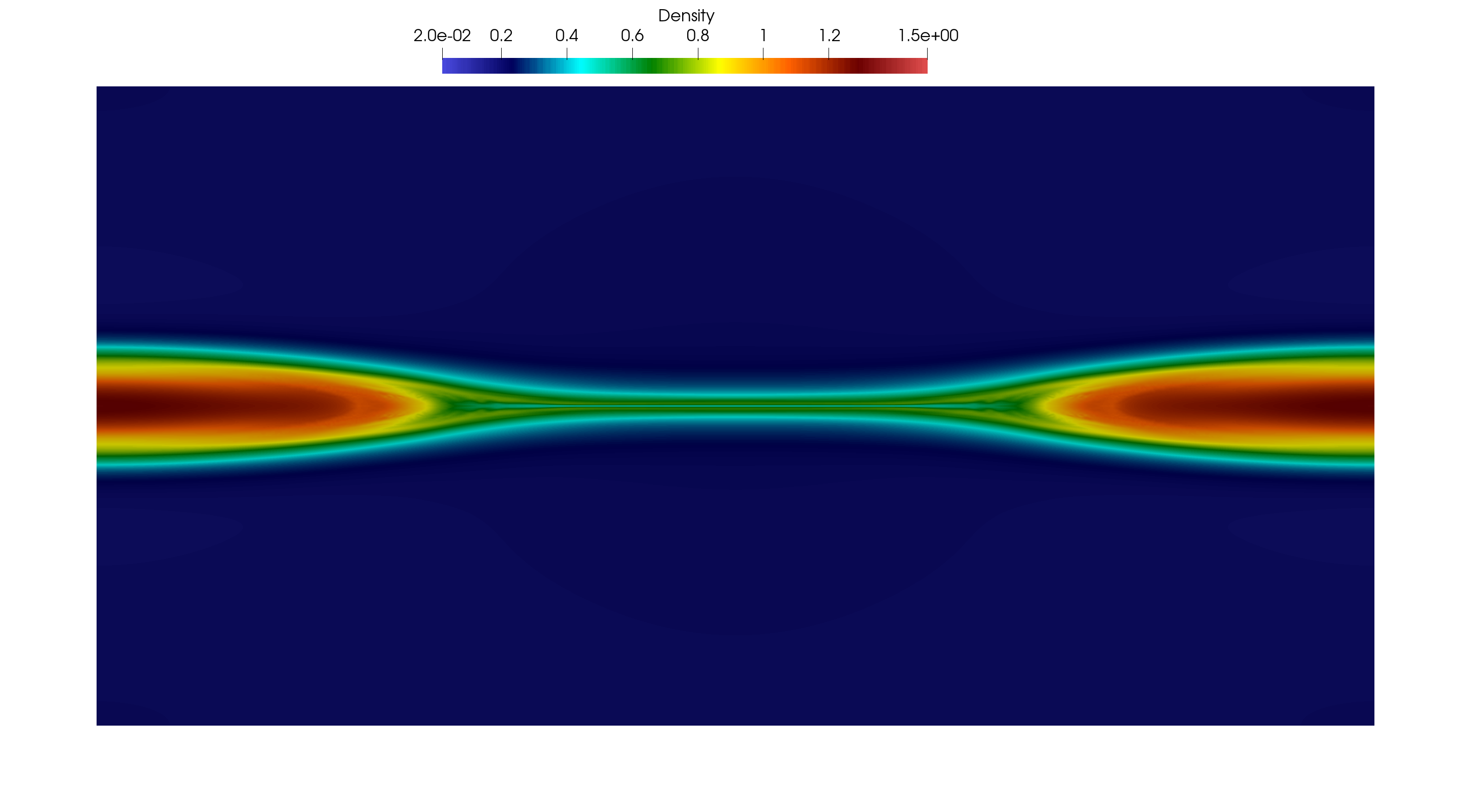}
		\includegraphics[trim=150 150 150 150 ,clip,width=0.4\linewidth]{figs/03_GEM-Reconnection/DOF1024_N7_eta_1e-3_t_050_Alpha.png}
		\includegraphics[trim=150 150 150 150 ,clip,width=0.4\linewidth]{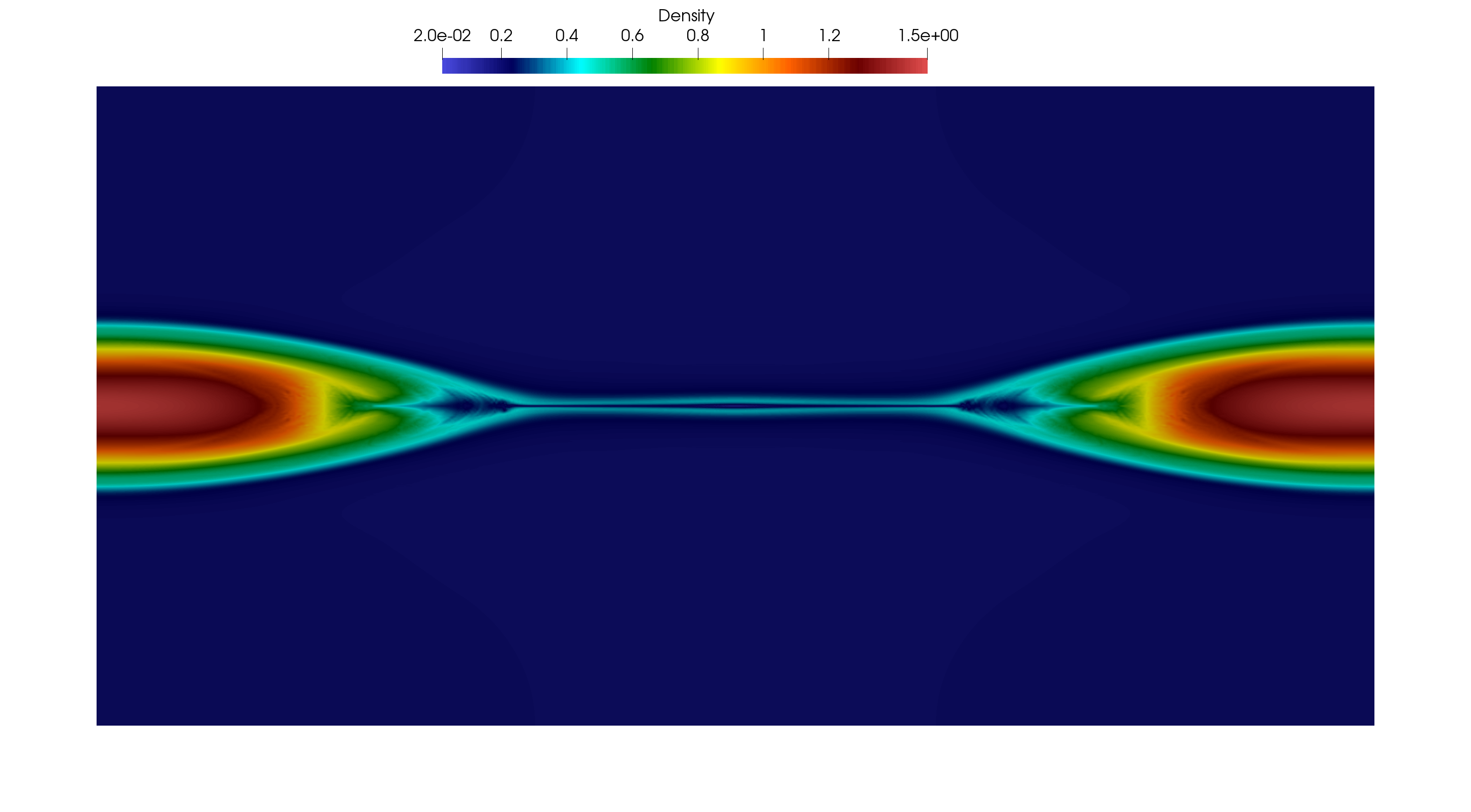}
		\includegraphics[trim=150 150 150 150 ,clip,width=0.4\linewidth]{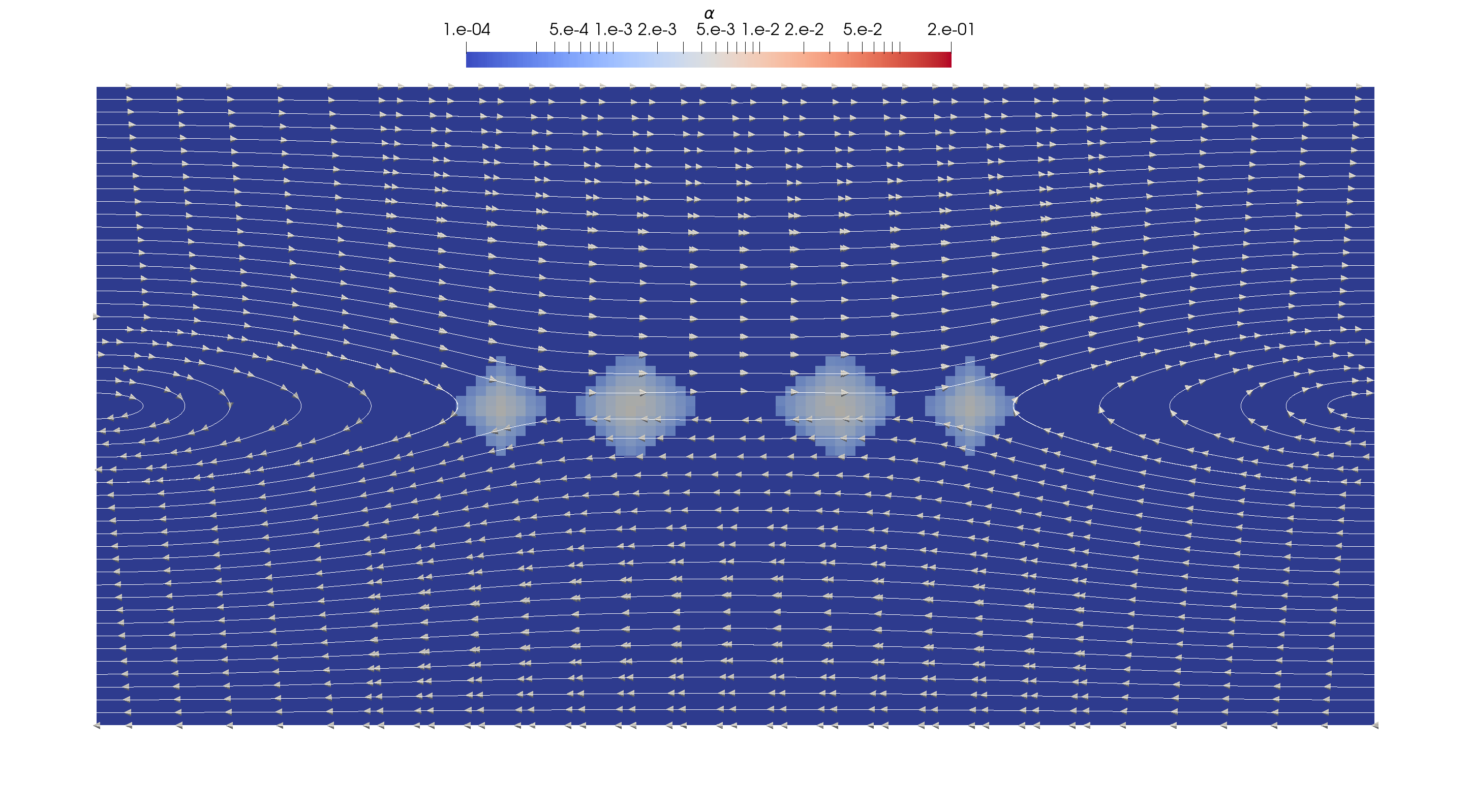}
	\caption{Evolution of the GEM reconnection challenge test ($\resistivity = 10^{-3}$) with the TVD-ES shock capturing method. 
	We show the density, the blending coefficient, and the magnetic field lines for the simulation with $1024 \times 512$ degrees of freedom and $N=7$ for $t=50$ (top) and $t=100$ (bottom).
	We use the indicator of Section \ref{sec:Indicator} with $\epsilon=\rho p$.}
	\label{fig:GEM1024_N7_1}
\end{figure}

Figures \ref{fig:GEM1024_N7_1} and \ref{fig:GEM1024_N7_5} show the density, the blending coefficient, and the magnetic field lines as the simulation with the TVD-ES method and $N=7$ advances for $\resistivity=10^{-3}$ and $\resistivity=5 \times 10^{-3}$, respectively.
As expected, the reconnected flux region is larger when the resistivity of the medium is higher.
As the simulation advances, some stabilization is required in the regions where the reconnection process occurs, and sometimes also along the reconnected magnetic field lines (not visible in these pictures), as magnetosonic and Alfvén waves travel away from the reconnection spots.
Note that only a small amount of blending is necessary to stabilize the simulation, but as mentioned before, without the FV stabilization the simulation crashes.

\begin{figure}[htb!]
	\centering
	\includegraphics[trim=800 1450 800 0 ,clip,width=0.4\linewidth]{figs/03_GEM-Reconnection/DOF1024_N7_eta_5e-3_t_050_Density.png}
	\includegraphics[trim=800 1450 800 0 ,clip,width=0.4\linewidth]{figs/03_GEM-Reconnection/DOF1024_N7_eta_1e-3_t_050_Alpha.png}
		\includegraphics[trim=150 150 150 150 ,clip,width=0.4\linewidth]{figs/03_GEM-Reconnection/DOF1024_N7_eta_5e-3_t_050_Density.png}
		\includegraphics[trim=150 150 150 150 ,clip,width=0.4\linewidth]{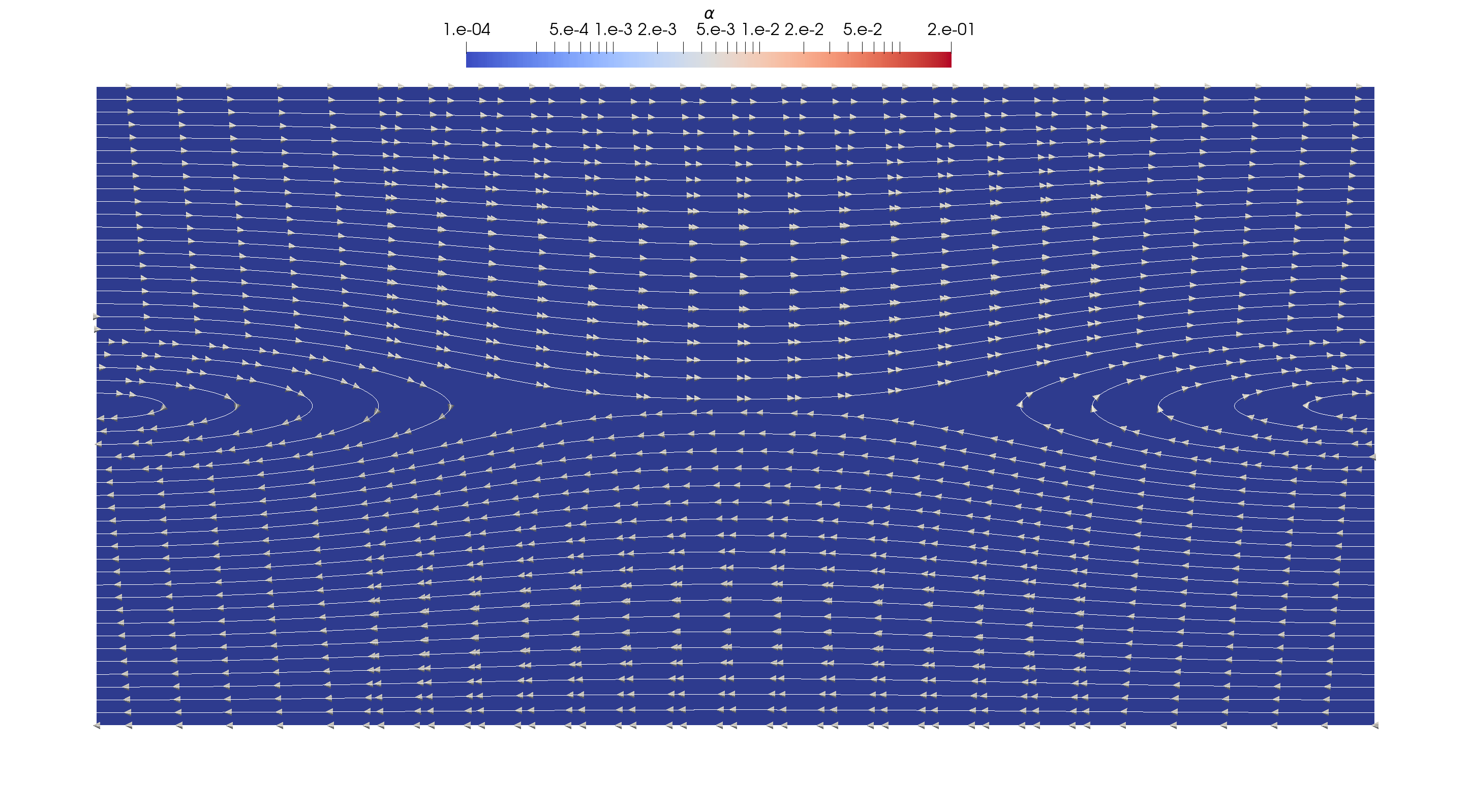}
		\includegraphics[trim=150 150 150 150 ,clip,width=0.4\linewidth]{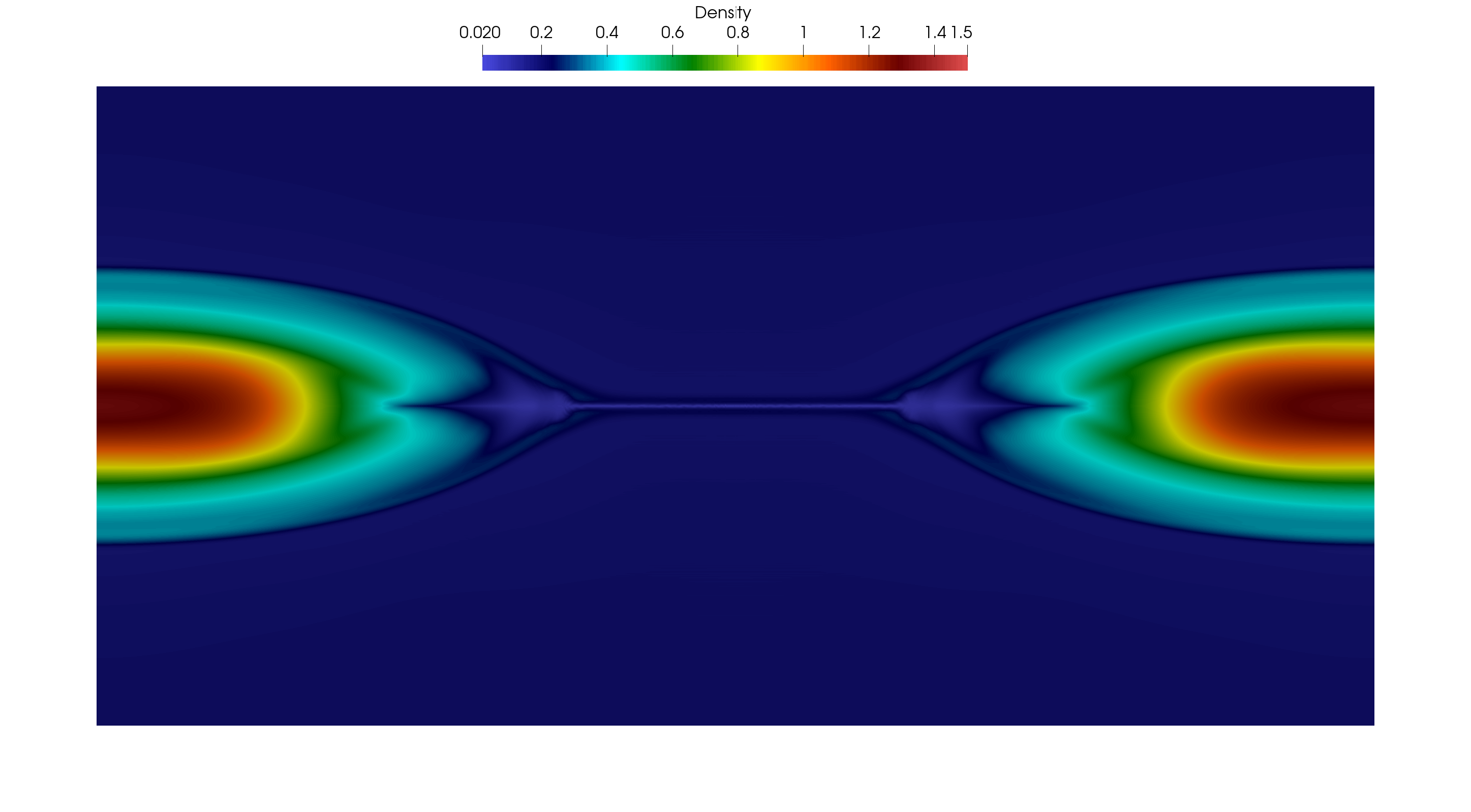}
		\includegraphics[trim=150 150 150 150 ,clip,width=0.4\linewidth]{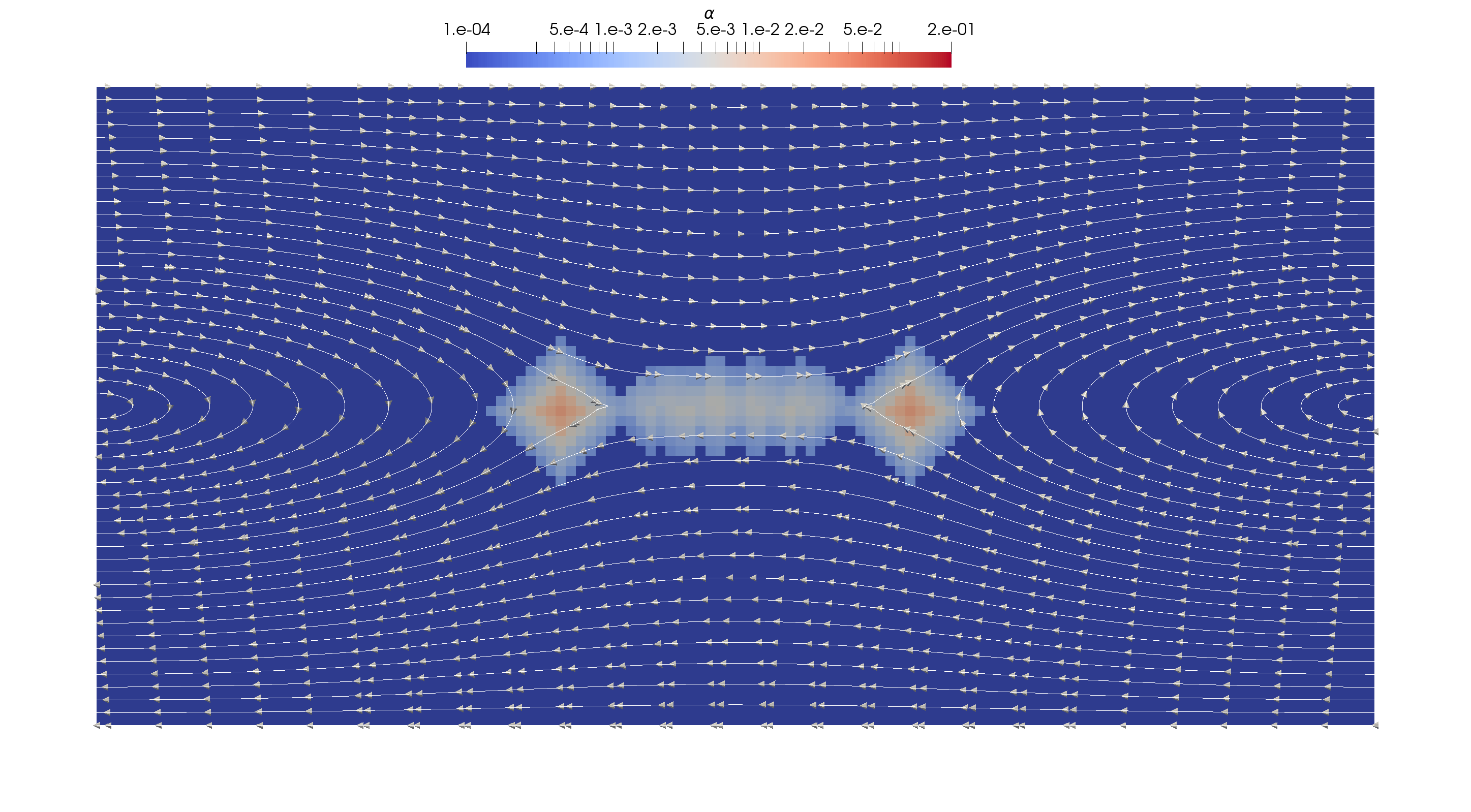}
	\caption{Evolution of the GEM reconnection challenge test ($\resistivity = 5 \times 10^{-3}$) with the TVD-ES shock capturing method. 
	We show the density, the blending coefficient, and the magnetic field lines for the simulation with $1024 \times 512$ degrees of freedom and $N=7$ for $t=50$ (top) and $t=100$ (bottom).
	We use the indicator of Section \ref{sec:Indicator} with $\epsilon=\rho p$.}
	\label{fig:GEM1024_N7_5}
\end{figure}

Finally, Figure \ref{fig:GEM_ReconnectedF} shows the evolution of reconnected flux, $\phi$, as a function of time for the two resistivities studied and different resolutions.
The reconnected flux is computed as the total magnetic field in the $y$ component over the $x$ axis \cite{Birn2001,Sousa2015},
\begin{equation}
\phi(t) = \frac{1}{2} \int_{-L_x/2}^{L_x/2} |B_2(x,y=0,t)| \d x.
\end{equation}

As can be seen in Figure \ref{fig:GEM_ReconnectedF}, the $N=3$ TVD-ES scheme produces higher reconnection rates than the $N=7$ TVD-ES scheme for the same number of degrees of freedom.
The higher reconnection rates are a consequence of the higher numerical resistivity, which is in turn a consequence of the lower polynomial degree, but also of the modal shock sensor that we employ.

For the highest value of resistivity tested, $\resistivity = 5 \times 10^{-3}$, Figure \ref{fig:GEM_ReconnectedF_5} shows a comparison of the evolution of the reconnected flux obtained with our numerical schemes and the numerical results of \citet{Birn2001}.
We only show the comparison for $t \le 40$ since we could not find data in the literature for $t > 40$.
It can be observed that the reconnection rate predicted with the FV/DGSEM method evolves similarly as in \cite{Birn2001}, but that our schemes are slightly less resistive than the reference.

Finally, in Figure \ref{fig:GEM_Entropy}, we show the evolution of the total mathematical entropy, $S_{\Omega}$, over time for the GEM reconnection challenge test.
As expected, the schemes fulfill the second law of thermodynamics.
Since this case has dissipation through resistivity, the mathematical entropy decreases monotonically from the beginning of the simulation.
Moreover, it can be observed that the setup with the higher resistivity has a higher entropy dissipation rate.

\begin{figure}[htb]
\centering
\begin{subfigure}[b]{0.45\linewidth}
	\includegraphics[width=\linewidth]{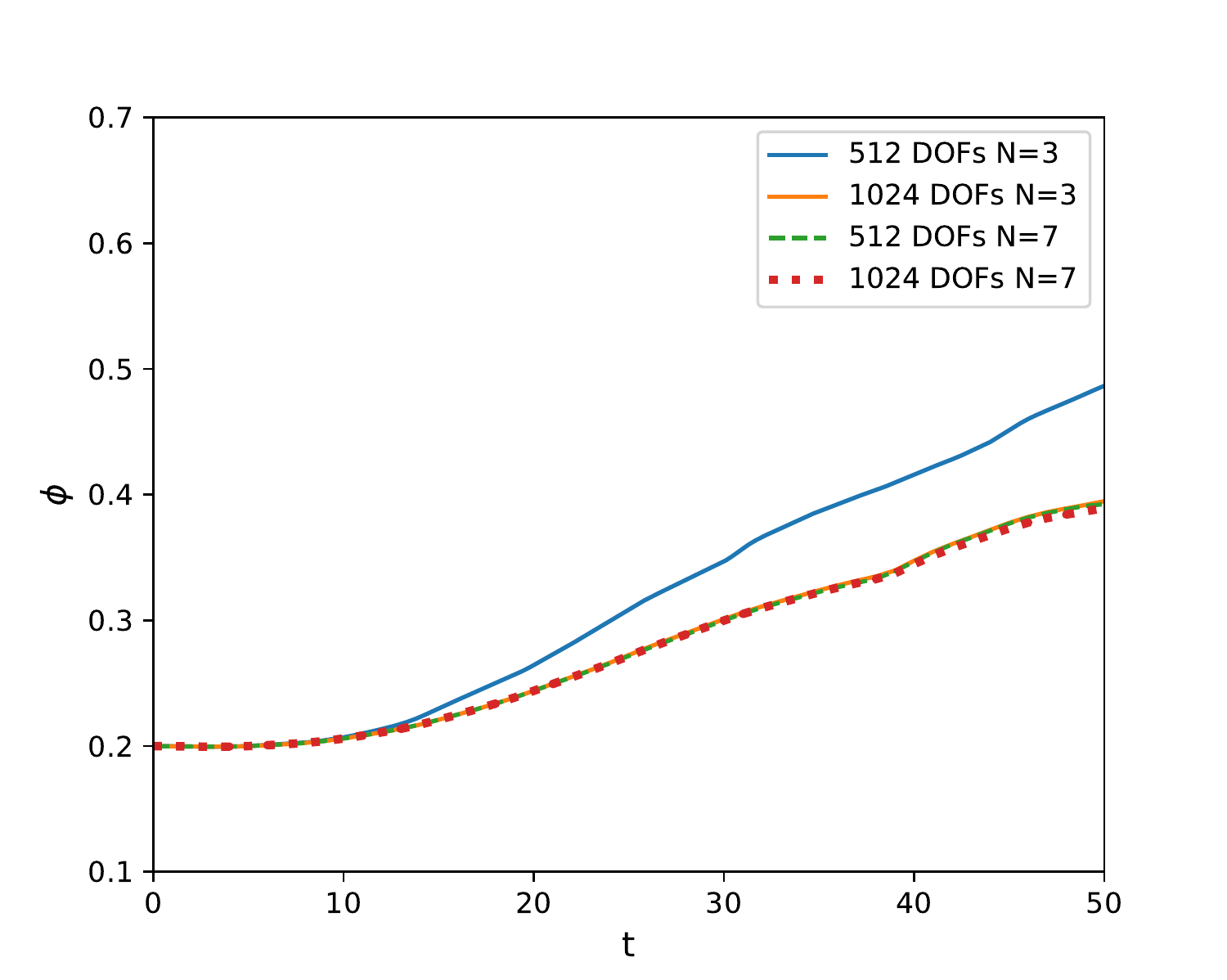}
	\caption{$\resistivity = 10^{-3}$}
\end{subfigure}
\begin{subfigure}[b]{0.45\linewidth}
	\includegraphics[width=\linewidth]{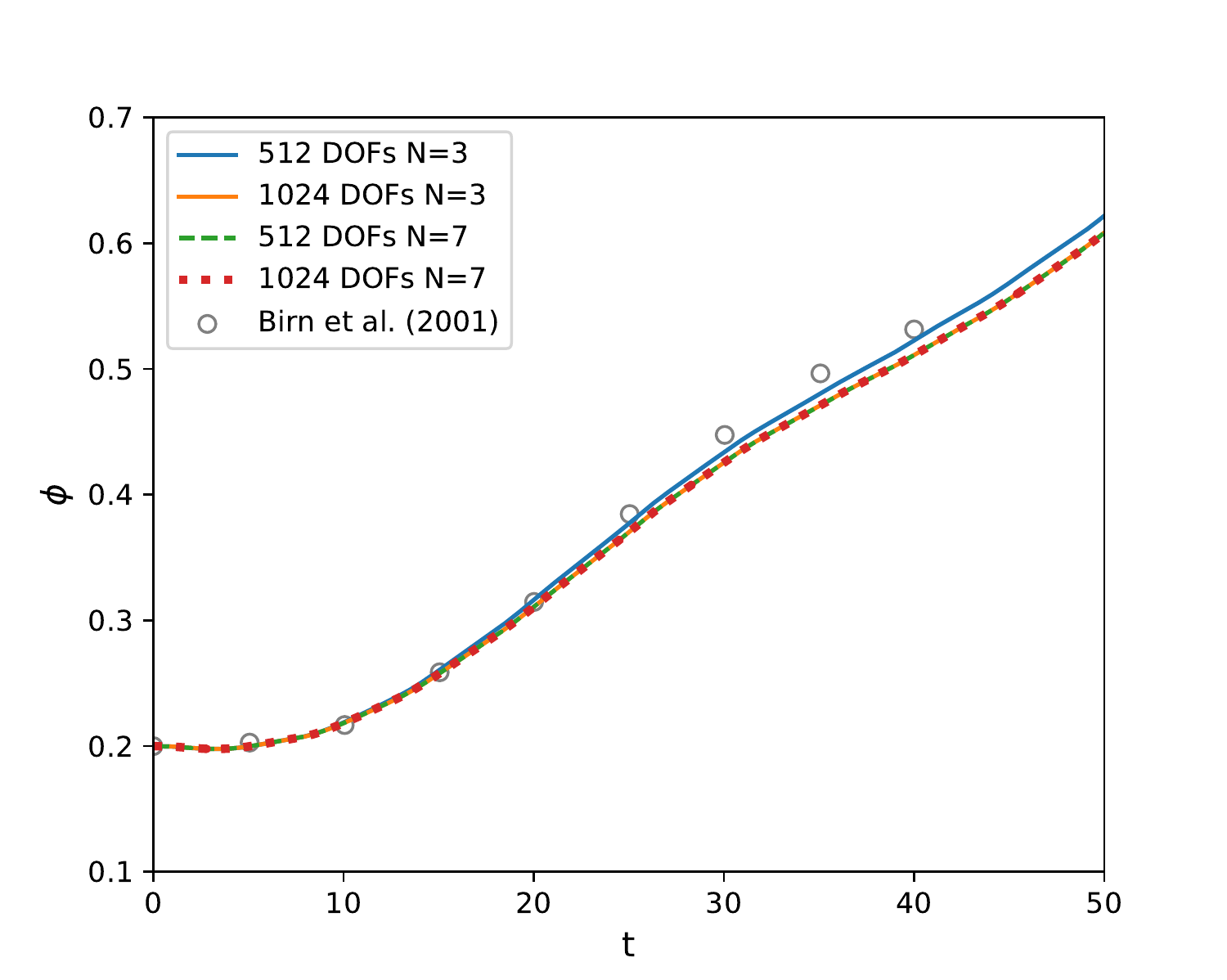}
	\caption{$\resistivity = 5 \times 10^{-3}$}
	\label{fig:GEM_ReconnectedF_5}
\end{subfigure}
\caption{Reconnected flux for the GEM reconnection challenge as a function of time obtained with the FV/DGSEM method and comparison with the data of \citet{Birn2001}.
The legend shows the number of degrees of freedom in the $x$ direction.}
\label{fig:GEM_ReconnectedF}
\centering
\includegraphics[width=0.7\linewidth]{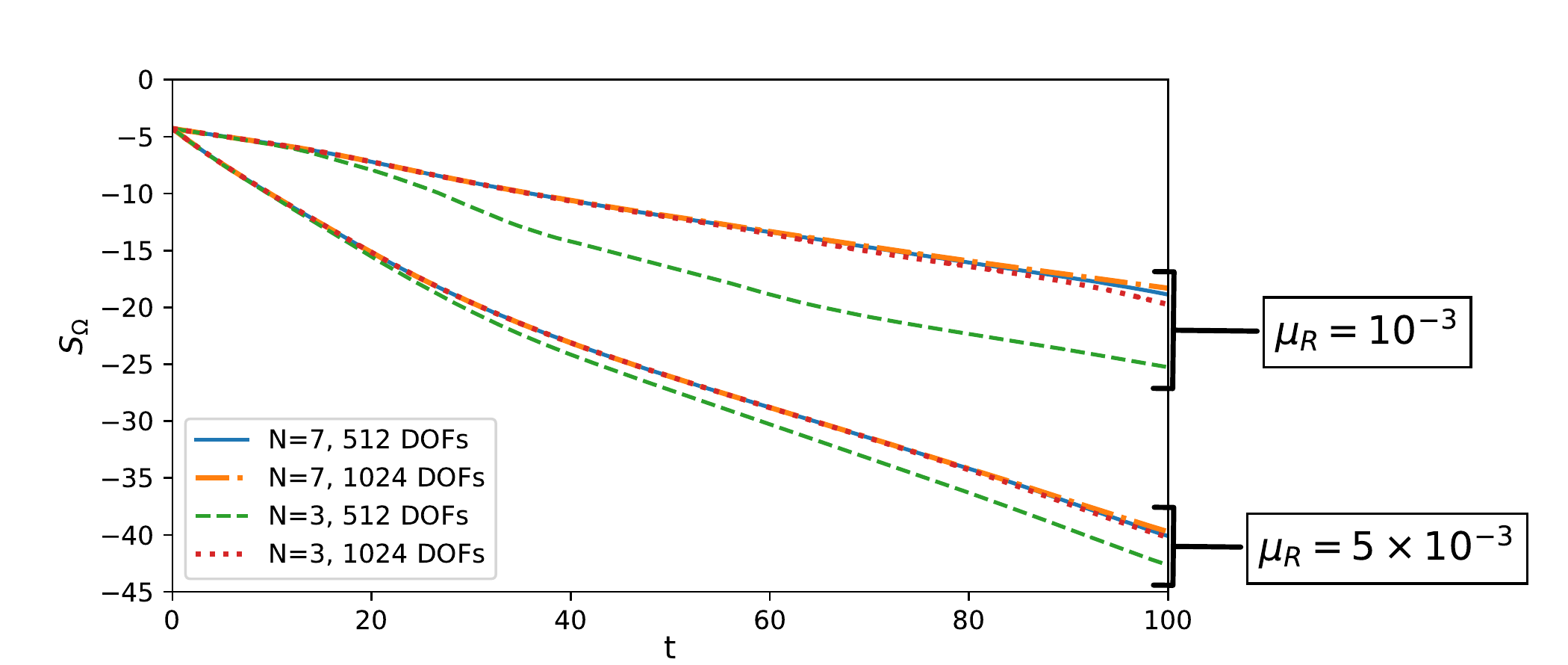}

\caption{Evolution of the total mathematical entropy, $S_{\Omega}$, over time for the GEM reconnection challenge test.
The legend shows the number of degrees of freedom in the $x$ direction.}
\label{fig:GEM_Entropy}
\end{figure}

\subsection{Io's Interaction with its Plasma Torus}

As a final test, we apply our hybrid FV/DGSEM method to simulate a space physics problem: the interaction of Jupiter's moon Io with its plasma torus.

Io is the most volcanically active body of the solar system, which is embedded in Jupiter's magnetic field, the largest and most powerful planetary magnetosphere of the solar system.
Due to its strong volcanic activity, Io expels ions and neutrals, which are in turn ionized by ultraviolet and electron impact ionization, forming a plasma torus around Jupiter \cite{Kivelson2003,Saur2004}.
As Io moves inside the plasma torus, elastic collisions of ions and neutrals inside its atmosphere generate a magnetospheric disturbance that propagates away from Io along the background magnetic field lines at the Alfvén wave speed \cite{Kivelson2003,Saur2004}.
This phenomenon creates a pair of Alfvén current tubes that are commonly called \textit{Alfvén wings}, which have been observed by several flybys \cite{Kivelson1997}.

A number of numerical studies have successfully described the Io/plasma torus interaction using compressible MHD models.
For instance, \citet{Saur2004} studied the physical phenomenon using a multi-fluid MHD model.
\citet{Jacobsen2007} used a single-fluid MHD model to study the effect of the density profile of the plasma torus and the Jovian ionosphere on the geometry of the Alfvén wings.
Blöcker et al. \cite{Blocker2017,Blocker2018} also used a single-fluid MHD model to analyze the effect of volcanic plumes in Io's atmosphere on the magnetic field iteractions.
Recently, \citet{Bohm2019} showed that an entropy stable high-order DG discretization of the ideal MHD equations encounters positivity issues in this problem, and simulated the Io/plasma torus phenomenon using an artificial dissipation-based shock capturing method.
In this section, we show that our proposed hybrid FV/DGSEM discretization is robust enough to handle the steep gradients that appear in this setup.

Following \cite{Blocker2017,Blocker2018,Jacobsen2007}, we use a simplified description of the physical problem, where neutrals, relativistic, visco-resistive, and Hall effects are neglected.
For simplicity, we follow the approach of \cite{Blocker2017,Blocker2017,Jacobsen2007}, which models the moon using a source term that represents a gas cloud that exchanges momentum and energy with the plasma torus\footnote{A more accurate representation of the plasma interaction with a planetary body can be obtained with non-conducting inner boundaries \cite{Duling2014}. However, the simplified version with the gas cloud is enough for the purpose of this test.}.
In summary, we use a modified version of the ideal GLM-MHD model \eqref{eq:idealGLM-MHD},
\begin{equation}
\partial_t \state{u} + \Nabla \cdot \blocktensor{f}^a + \noncon = \state{r}_c,
\end{equation}
where the source term is used to model the neutral-ion collision that takes place in Io's atmosphere \cite{Blocker2016,Blocker2017,Blocker2018},
\begin{equation}
\state{r}_c = \left( 0, -\varpi \rho \vec{v}, -\varpi k, \vec{0}, 0 \right)^T,
\end{equation}
$\varpi$ is the ion-neutral collision frequency and $k$ is the total energy exchange in the moon's atmosphere, modeled as
\begin{equation}
k = \rho E - \frac{1}{2 \mu_0} \left( \norm{\vec{B}}^2  + \psi^2 \right).
\end{equation}

In our model, we locate Io in the origin of our computational domain and define the collision frequency as \cite{Bohm2019},
\begin{equation}
\varpi=
\begin{cases}
\varpi_{\text{in}}, & r \le R_{\text{Io}} \\
\varpi_{\text{in}} \exp \left( \frac{R_{\text{Io}} - r}{d} \right), & \text{otherwise},
\end{cases}
\end{equation}
where $r=\norm{\vec{x}}$ is the distance to the origin, $R_{\text{Io}}$ is the radius of Io, and $d=150/1820$ \cite{Bohm2019} is dilatation factor that models how Io's atmosphere becomes thinner away from the moon's surface.

Since this problem contains a broad scale of orders of magnitude, we perform a non-dimensionalization following \cite{Bohm2019}.
The characteristic quantities are $R_{\text{Io
}}=1.82 \times 10^{6}$ m, the radius of Io, $\rho_{\infty}=7.02 \times 10^{-17}$ kg$/$m$^{3}$, the density of the plasma torus \cite{Jacobsen2007,Blocker2018,Blocker2017}, $V_{\infty}=56 \times 10^3$ m$/$s, the orbital velocity of Io, and, as stated at the beginning of Section \ref{sec:Results}, $\mu_0=1$, which implies that we use the magnetic permeability of empty space as a characteristic quantity, $\mu_{0,\infty}=1.26\times 10^{-6}$ N$/$A$^2$.

Table \ref{tab:IoParams} contains the parameters used for the simulation in SI units and their corresponding non-dimensional values that are computed with the characteristic quantities.
Since in our simplified model the background magnetic field in $z$, $B_3$, is constant, and taking into account that we want to compare our numerical results with the experimental data taken by the I31 flyby of the Galileo spacecraft \cite{Kivelson1997}, we initialize $B_3$ as the mean measured value along the I31 path.
Note that our non-dimensionalization retains the supersonic Mach number, $\text{Ma} \approx 2$, and the sub-Alvénic magnetic Mach number, $\text{Ma}_m \approx 0.28$, of the physical setup.

\begin{table}[H!]
\caption{Initial condition and other parameters for the Io test. The quantities with subscript $0$ are part of the initial condition.} \label{tab:IoParams}
\begin{tabular}{l|rrrr}
Variable description & Variable name  & Value in SI units & Ref. & Nondim. value \\
\hline
Density of the plasma torus 	
	& $\rho_0$ 				& $7.02 \times 10^{-17}$ kg$/$m$^{3}$ & \cite{Kivelson2003} & $1$
	\\
Velocity in the $x$ direction 
	& $(v_1)_0$ 			& $56 \times 10^3$ m$/$s  & \cite{Kivelson2003} & $1$
	\\
Velocity in the $y$ and $z$ directions 	
	& $(v_2)_0$, $(v_3)_0$ 	& $0$ m$/$s & *\cite{Jacobsen2007} & $0$
	\\
Magnetic field in the $x$ and $y$ directions 	
	& $(B_1)_0$, $(B_2)_0$ 	& $0$ T & *\cite{Jacobsen2007} & $0$
	\\
Magnetic field in the $z$ direction
	& $(B_3)_0$ 	& $-1930$ nT & **\cite{Kivelson1997} & $-3.604$ 
	\\
Pressure in the plasma torus
	& $p_0$			& $34$nPa	& \cite{Kivelson1997} & $0.149$
	\\
Collision frequency 
	& $\varpi_{\text{in}}$		& $4$ Hz	 & *\cite{Bohm2019} & $127.719$
	\\
\hline
\multicolumn{5}{l}{* \  Assumed as in reference.}
\\ 
\multicolumn{5}{l}{** Taken as the mean value along the I31 path.}
\end{tabular}
\end{table}

\begin{figure}[htb]
\centering
\begin{subfigure}[b]{0.35\linewidth}
\includegraphics[trim=0 80 1600 80,clip,width=\linewidth]{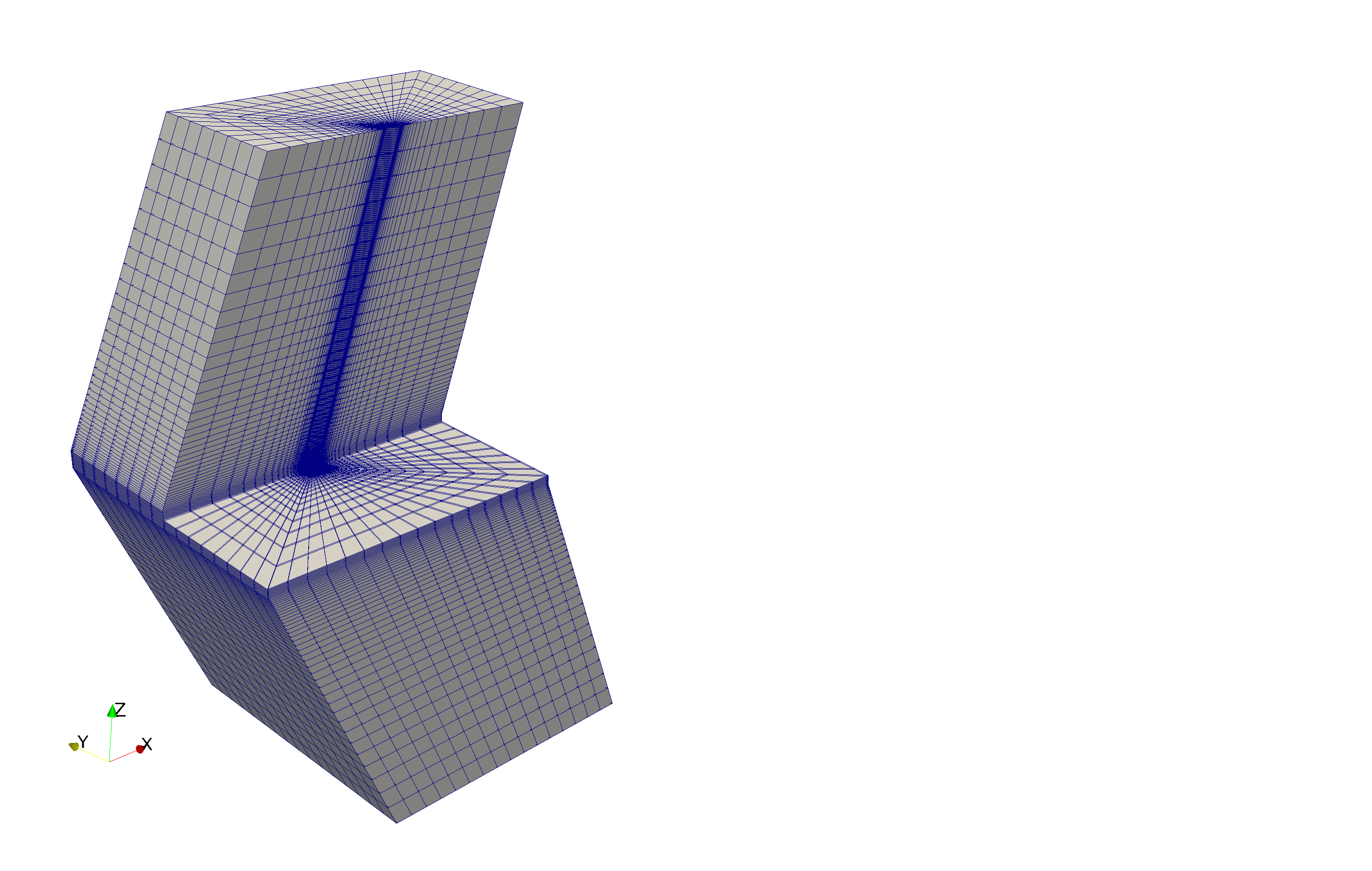}
\caption{Clip of the entire mesh.}
\end{subfigure}
\begin{subfigure}[b]{0.45\linewidth}
\includegraphics[trim=0 0 0 0,clip,width=\linewidth]{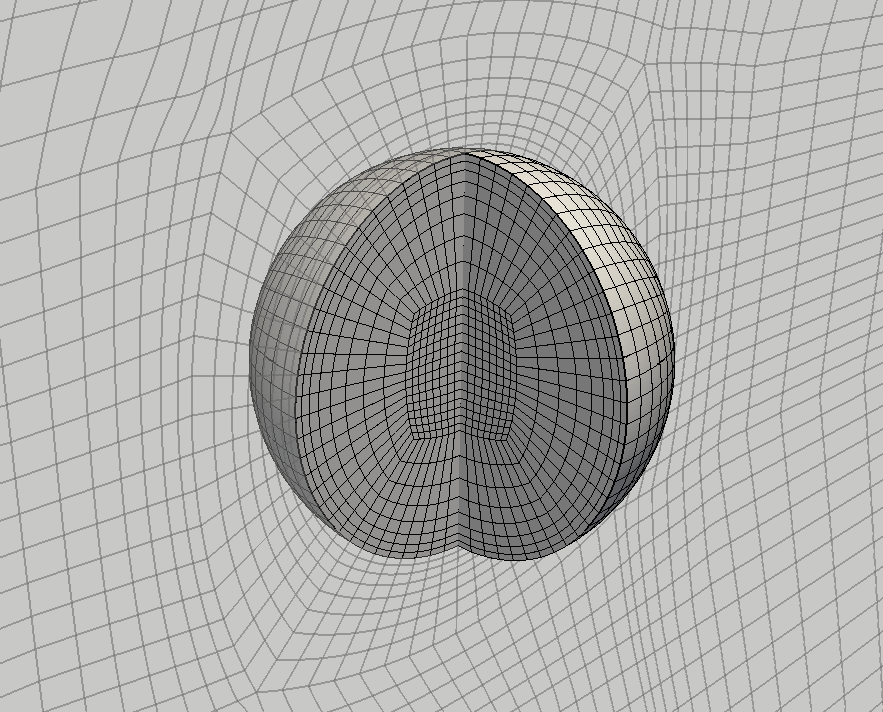}
\caption{Detail of the curvilinear spherical mesh.}
\end{subfigure}
\caption{Mesh used for the simulation of Io's interaction with its plasma torus.}
\label{fig:Io_Setup}
\end{figure}

Our computational domain spans $x \in [-20,34]$, $y \in [-20,20]$, $z \in [-40,40]$.
We use a total of 165888 curvilinear hexahedral elements with a mapping polynomial degree $N_{\text{geo}}=2$.
Figure \ref{fig:Io_Setup} shows the mesh, where
we have refined the regions of interesting flow features and we have employed curvilinear elements to match the spherical surface of Io.
We employ the TVD-ES method and set the polynomial degree to $N=4$, for a total of $20.736$ MDOFs.
All the boundaries of the domain impose the far-field conditions, $\state{u}_0$, with a weak Dirichlet boundary condition.

Following \cite{Blocker2018,Bohm2019}, we run the simulation until before the Alfvén wings touch the upper and lower boundaries of the domain ($t=10$) to avoid problems with the wave reflections.
As we will show, this simulation time is enough to observe the deflection of the Jovian magnetic field and to compare the numerical results with measurements performed during the I31 flyby of the Galileo space craft.

Figure \ref{fig:Io_Contours} shows the sonic Mach number, the magnetic field in the $x$ direction and the blending coefficient for a slice cut at $y=0$ of the three-dimensional domain at non-dimensional time $t=10$.
As can be observed, inside the Alfvén wings the magnetic field gets deflected and the Mach number reduces since the plasma flow gets slowed down.
In the wake that forms behind Io, the density increases significantly, leading to high Mach numbers and shock-like structures.
Moreover, it can be seen that in general a low amount of stabilization is required to stabilize the simulation, predominantly in the atmosphere of Io, where the plasma interaction takes place.

\begin{figure}[htb]
\centering

\begin{subfigure}[b]{0.32\linewidth}
	\includegraphics[trim=0 0 0 0,clip,width=\linewidth]{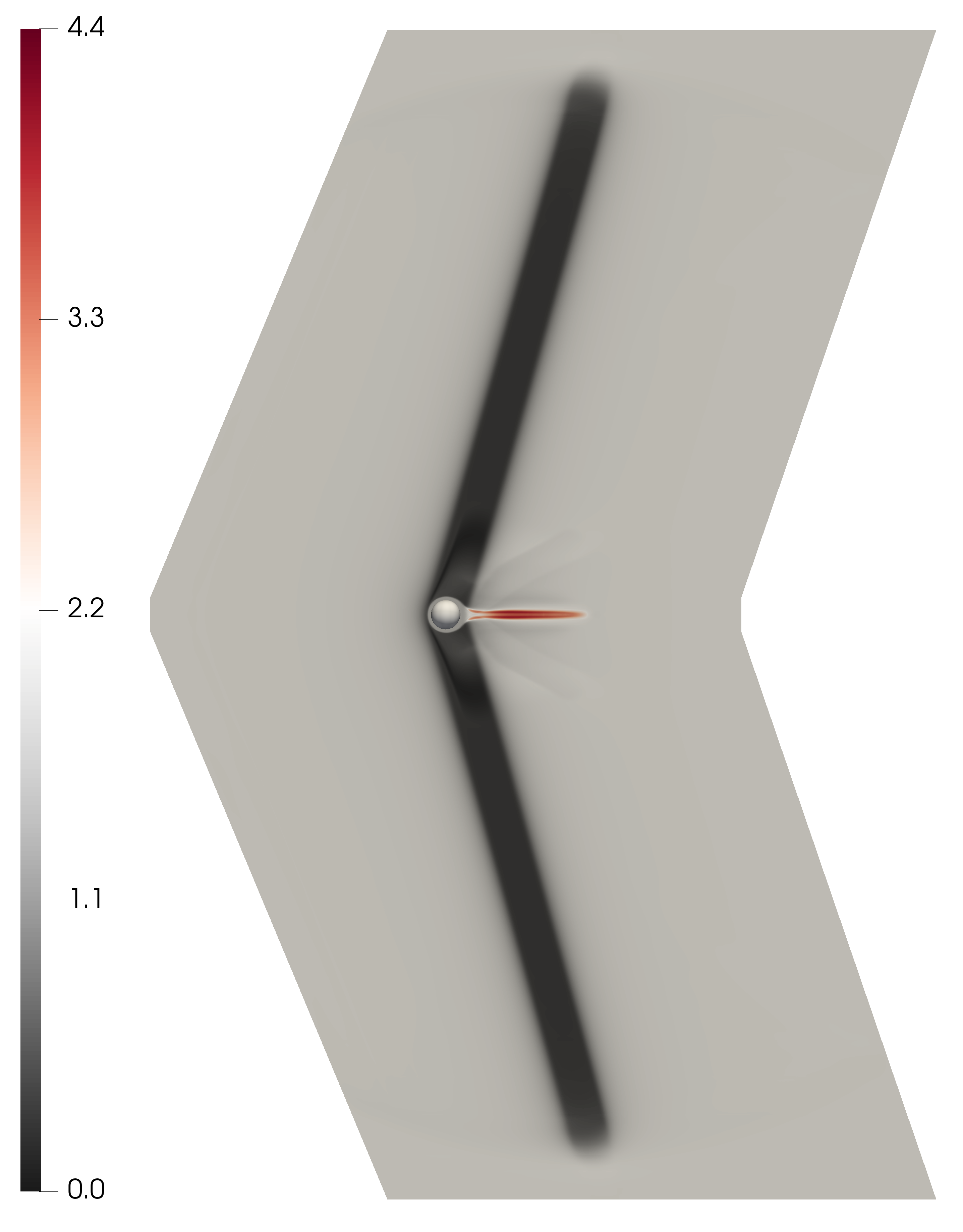}
	\caption{Sonic Mach number}
\end{subfigure}
\begin{subfigure}[b]{0.32\linewidth}
	\includegraphics[trim=0 0 0 0,clip,width=\linewidth]{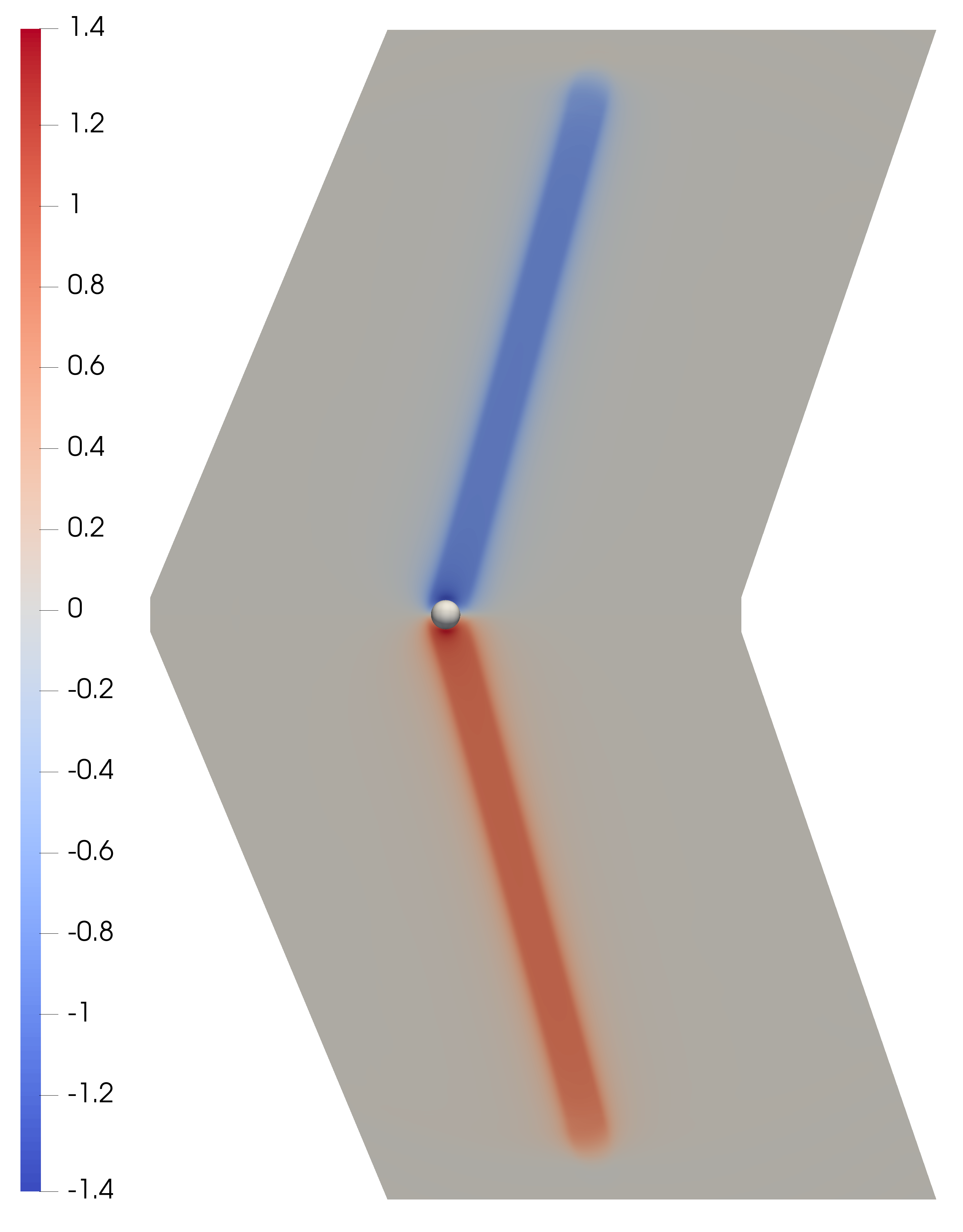}
	\caption{$B_1$}
\end{subfigure}
\begin{subfigure}[b]{0.32\linewidth}
	\includegraphics[trim=0 0 0 0,clip,width=\linewidth]{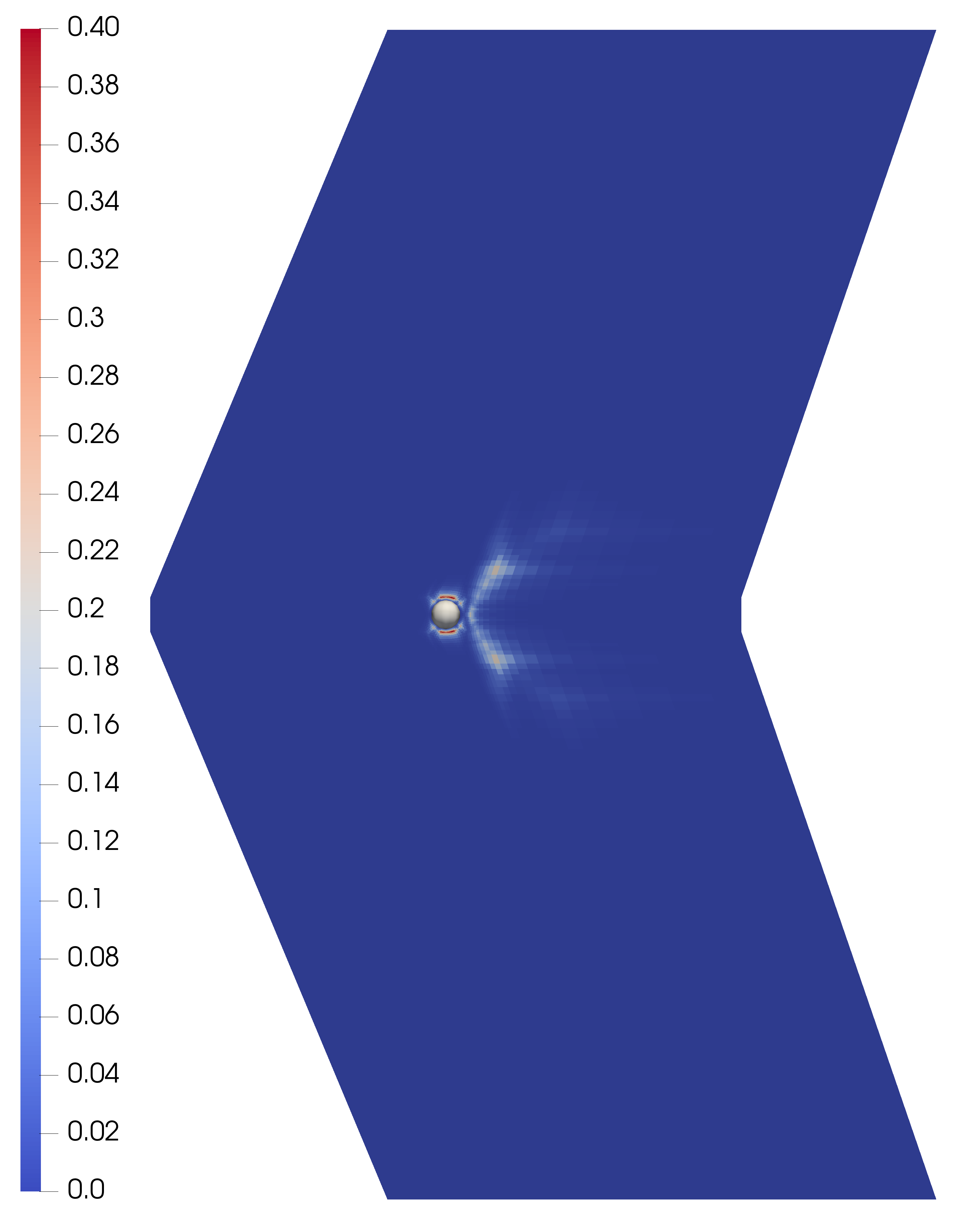}
	\caption{$\alpha$}
\end{subfigure}

\caption{Slice cuts at $y=0$ of the three-dimensional domain at non-dimensional time $t=10$ showing the sonic Mach number, the magnetic field in the $x$ direction and the blending coefficient.}
\label{fig:Io_Contours}
\end{figure}

In Figure \ref{fig:Io_DetailVz_Galileo}, we show a detailed view of the plasma velocity in the $z$ direction in the vicinity of Io on a slice cut at $y=0$, and also the trajectory of the I31 flyby of the Galileo spacecraft.
The zoomed-in view shows the shock-like structures in the wake behind the moon.

\begin{figure}[htb]
\centering
\includegraphics[trim=0 0 0 0,clip,width=0.5\linewidth]{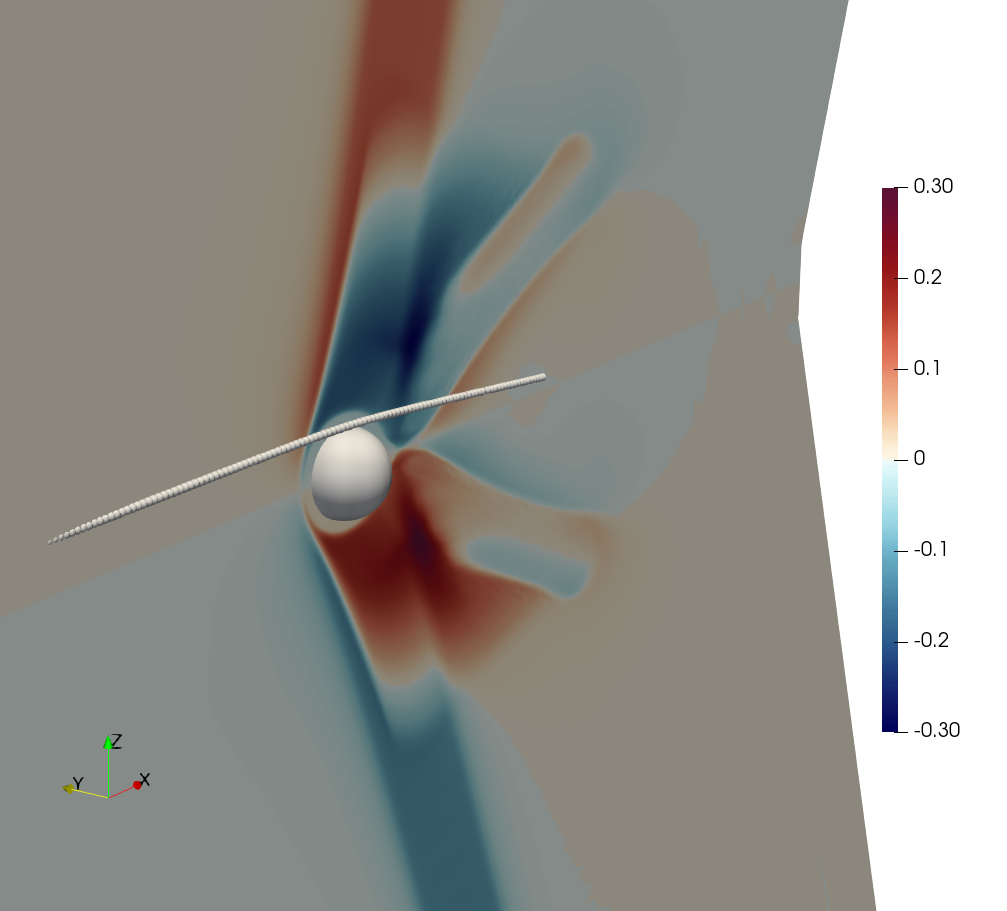}

\caption{Detail of the plasma velocity in $z$ direction ($v_3$) on a slice cut at $y=0$ and representation of the trajectory of the I31 flyby of the Galileo spacecraft with data obtained from \cite{Kivelson1997}.}
\label{fig:Io_DetailVz_Galileo}
\end{figure}

Figure \ref{fig:Io_Magnetometer} shows a comparison of the three magnetic field components along the I31 trajectory predicted by the TVD-ES discretization of the GLM-MHD model (transformed back to SI units) and the I31 magnetometer measurements \cite{Kivelson1997}. 
In the $x$ axis we show the location of the measurements (in Io Phi-Omega -IPHIO- coordinates) and the time of the measurements (in universal time), which were conducted on the sixth of August 2001. Note that the numerical $B_1$ and $B_2$ values are shifted with the mean experimental value to account for the fact that in our setup we supposed that the background magnetic field is zero in the $y$ and $z$ directions.

This test shows that our simplified physical model, together with the hybrid FV/DGSEM TVD-ES discretization, is able to capture some of the relevant features of the Io/plasma torus interaction problem.

\begin{figure}[htb]
\begin{center}
\includegraphics[trim=0 50 0 0,clip,width=\linewidth]{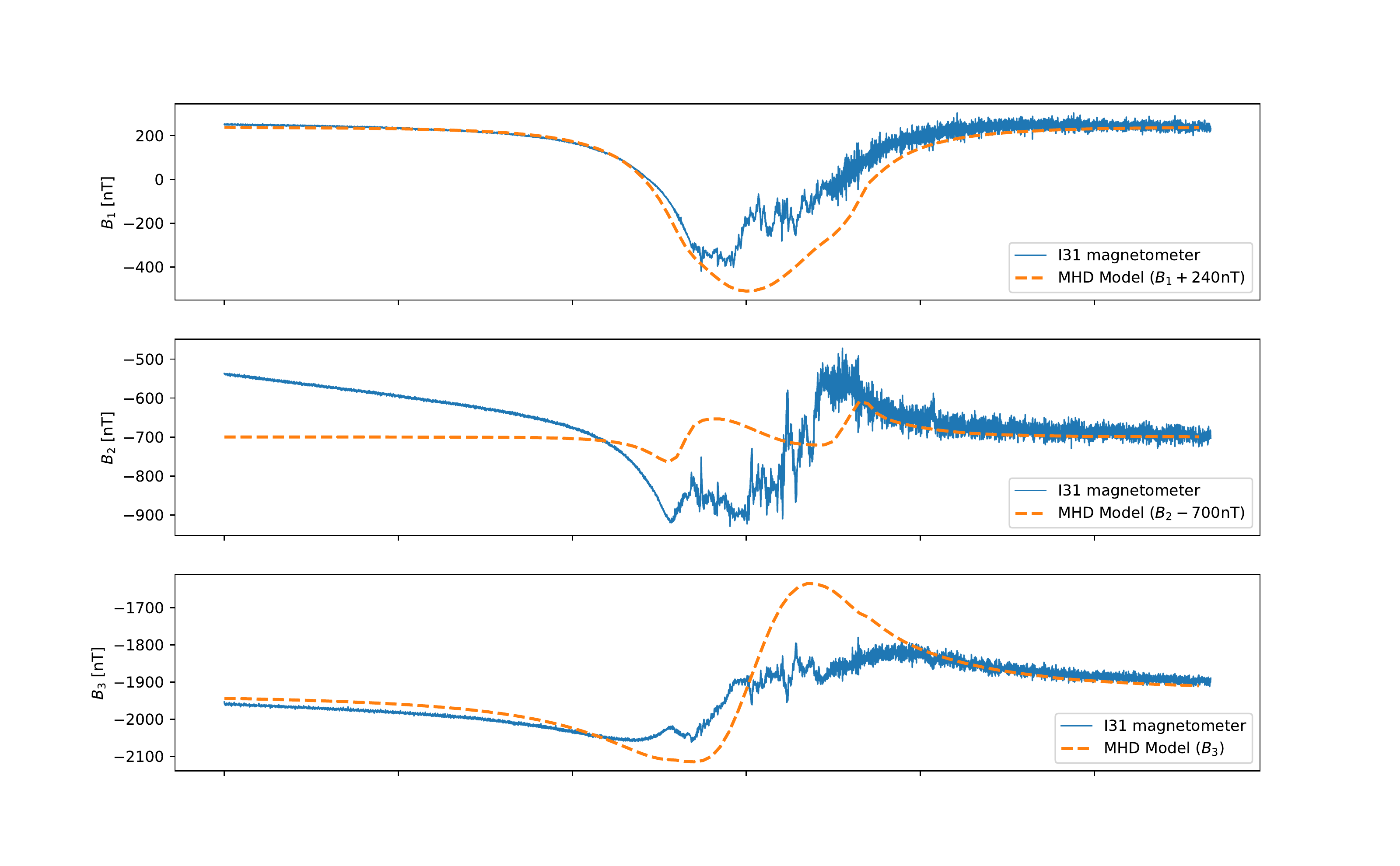}
\begin{tabular}{c|m{1.65cm}m{1.65cm}m{1.65cm}m{1.65cm}m{1.65cm}m{1.65cm}m{0.8cm}}
$t$[UT] & \ 4:25 & \ 4:36 & \ 4:48 & \ 4:59 & \ 5:10 & \ 5:21 & \\
$x$[R$_{\text{Io}}$] & 
-7.52 &
-5.05 &
-2.56 &
-0.01 &
\ 2.54 &
-6.29 &
\\
$y$[R$_{\text{Io}}$] & 
\ 0.07 &
-0.03 &
-0.14 &
-0.24 &
-0.29 &
\ 0.02 & \\
$z$[R$_{\text{Io}}$] & 
\ 0.97 &
\ 1.03 &
\ 1.08 &
\ 1.08 &
\ 0.89 &
\ 1.00 & 
\end{tabular}
\end{center}
\caption{Magnetic field components (in SI units) along the I31 trajectory predicted by the TVD-ES discretization of the GLM-MHD model and the I31 magnetometer measurements \cite{Kivelson1997}. 
In the $x$ axis we show the location of the measurements (in Io Phi-Omega -IPHIO- coordinates) and the time of the measurements (in universal time), which were conducted on the sixth of August 2001. 
Note that the numerical $B_1$ and $B_2$ values are shifted.}
\label{fig:Io_Magnetometer}
\end{figure}

\section{Conclusions}

In this work, we have presented two novel and robust entropy stable shock-capturing methods for DGSEM discretizations of the GLM-MHD equations.
These methods can be applied for the discretization of general systems of conservation laws with or without the addition of non-conservative terms.

The first method, which is an generalization of the method in \cite{Hennemann2020}, uses a first-order finite volume (FV) scheme to stabilize the advective and non-conservative terms of the DGSEM discretization.
The second method uses a reconstruction procedure for the subcell FV scheme that is carefully built to ensure entropy stability.

We have analytically proved that the methods presented in this paper are entropy stable on three-dimensional unstructured curvilinear meshes.
Moreover, we have provided numerical verifications the theoretical properties of the methods and have shown their robustness and accuracy with common benchmark cases and a space physics application.

The methods detailed in this paper are implemented in the open-source framework FLUXO (\url{www.github.com/project-fluxo}).
\section*{Acknowledgments}

This work has received funding from the European Research Council through the ERC Starting Grant “An Exascale aware and Un-crashable Space-Time-Adaptive Discontinuous Spectral Element Solver for Non-Linear Conservation Laws” (Extreme), ERC grant agreement no. 714487 (Gregor J. Gassner and Andrés Rueda-Ramírez).

The authors would like to thank Prof. Joachim Saur and Stephan Schlegel for their helpful insights about the Io case.

\printcredits

\bibliographystyle{model1-num-names}

\bibliography{Biblio}

\begin{thebibliography}{72}
\expandafter\ifx\csname natexlab\endcsname\relax\def\natexlab#1{#1}\fi
\providecommand{\bibinfo}[2]{#2}
\ifx\xfnm\relax \def\xfnm[#1]{\unskip,\space#1}\fi
\bibitem[{Wang et~al.(2013)Wang, Fidkowski, Abgrall, Bassi, Caraeni, Cary,
  Deconinck, Hartmann, Hillewaert, Huynh, Kroll, May, Persson, van Leer,
  Visbal, van Leer, and Visbal}]{Wang2013High}
\bibinfo{author}{Z.~J. Wang}, \bibinfo{author}{K.~Fidkowski},
  \bibinfo{author}{R.~Abgrall}, \bibinfo{author}{F.~Bassi},
  \bibinfo{author}{D.~Caraeni}, \bibinfo{author}{A.~Cary},
  \bibinfo{author}{H.~Deconinck}, \bibinfo{author}{R.~Hartmann},
  \bibinfo{author}{K.~Hillewaert}, \bibinfo{author}{H.~T. Huynh},
  \bibinfo{author}{N.~Kroll}, \bibinfo{author}{G.~May}, \bibinfo{author}{P.-O.
  Persson}, \bibinfo{author}{B.~van Leer}, \bibinfo{author}{M.~Visbal},
  \bibinfo{author}{B.~van Leer}, \bibinfo{author}{M.~Visbal},
\newblock \bibinfo{title}{{High-order CFD methods: current status and
  perspective}},
\newblock \bibinfo{journal}{International Journal for Numerical Methods in
  Fluids} \bibinfo{volume}{72} (\bibinfo{year}{2013})
  \bibinfo{pages}{811--845}.
\bibitem[{Cockburn et~al.(2000)Cockburn, Karniadakis, and Shu}]{Cockburn2000}
\bibinfo{author}{B.~Cockburn}, \bibinfo{author}{G.~E. Karniadakis},
  \bibinfo{author}{C.-W. Shu},
\newblock \bibinfo{title}{{The Development of Discontinuous Galerkin Methods}},
\newblock \bibinfo{journal}{Discontinuous Galerkin Methods}
  \bibinfo{volume}{11} (\bibinfo{year}{2000}) \bibinfo{pages}{3--50}.
\bibitem[{Hindenlang et~al.(2012)Hindenlang, Gassner, Altmann, Beck,
  Staudenmaier, and Munz}]{Hindenlang2012}
\bibinfo{author}{F.~Hindenlang}, \bibinfo{author}{G.~J. Gassner},
  \bibinfo{author}{C.~Altmann}, \bibinfo{author}{A.~Beck},
  \bibinfo{author}{M.~Staudenmaier}, \bibinfo{author}{C.~D. Munz},
\newblock \bibinfo{title}{{Explicit discontinuous Galerkin methods for unsteady
  problems}},
\newblock \bibinfo{journal}{Computers and Fluids} \bibinfo{volume}{61}
  (\bibinfo{year}{2012}) \bibinfo{pages}{86--93}.
\bibitem[{Rivi{\`{e}}re(2008)}]{Riviere2008}
\bibinfo{author}{B.~Rivi{\`{e}}re}, \bibinfo{title}{{Discontinuous Galerkin
  Methods for Solving Elliptic and Parabolic Equations Theory and
  Implementation}}, \bibinfo{publisher}{SIAM}, \bibinfo{year}{2008}.
\bibitem[{Kopriva et~al.(2002)Kopriva, Woodruff, and Hussaini}]{Kopriva2002}
\bibinfo{author}{D.~A. Kopriva}, \bibinfo{author}{S.~L. Woodruff},
  \bibinfo{author}{M.~Y. Hussaini},
\newblock \bibinfo{title}{{Computation of electromagnetic scattering with a
  non-conforming discontinuous spectral element method}},
\newblock \bibinfo{journal}{International Journal for Numerical Methods in
  Engineering} \bibinfo{volume}{53} (\bibinfo{year}{2002})
  \bibinfo{pages}{105--122}.
\bibitem[{Rueda-Ram{\'{i}}rez et~al.(2019{\natexlab{a}})Rueda-Ram{\'{i}}rez,
  Manzanero, Ferrer, Rubio, and Valero}]{RuedaRamirez2019a}
\bibinfo{author}{A.~M. Rueda-Ram{\'{i}}rez}, \bibinfo{author}{J.~Manzanero},
  \bibinfo{author}{E.~Ferrer}, \bibinfo{author}{G.~Rubio},
  \bibinfo{author}{E.~Valero},
\newblock \bibinfo{title}{{A p-multigrid strategy with anisotropic p-adaptation
  based on truncation errors for high-order discontinuous Galerkin methods}},
\newblock \bibinfo{journal}{Journal of Computational Physics}
  \bibinfo{volume}{378} (\bibinfo{year}{2019}{\natexlab{a}})
  \bibinfo{pages}{209--233}.
\bibitem[{Rueda-Ram{\'{i}}rez et~al.(2019{\natexlab{b}})Rueda-Ram{\'{i}}rez,
  Rubio, Ferrer, and Valero}]{RuedaRamirez2019}
\bibinfo{author}{A.~M. Rueda-Ram{\'{i}}rez}, \bibinfo{author}{G.~Rubio},
  \bibinfo{author}{E.~Ferrer}, \bibinfo{author}{E.~Valero},
\newblock \bibinfo{title}{{Truncation Error Estimation in the p-Anisotropic
  Discontinuous Galerkin Spectral Element Method}},
\newblock \bibinfo{journal}{Journal of Scientific Computing}
  \bibinfo{volume}{78} (\bibinfo{year}{2019}{\natexlab{b}})
  \bibinfo{pages}{433--466}.
\bibitem[{Bohm et~al.(2018)Bohm, Winters, Gassner, Derigs, Hindenlang, and
  Saur}]{Bohm2018}
\bibinfo{author}{M.~Bohm}, \bibinfo{author}{A.~R. Winters},
  \bibinfo{author}{G.~J. Gassner}, \bibinfo{author}{D.~Derigs},
  \bibinfo{author}{F.~Hindenlang}, \bibinfo{author}{J.~Saur},
\newblock \bibinfo{title}{{An entropy stable nodal discontinuous Galerkin
  method for the resistive MHD equations. Part I: Theory and numerical
  verification}},
\newblock \bibinfo{journal}{Journal of Computational Physics}
  \bibinfo{volume}{1} (\bibinfo{year}{2018}) \bibinfo{pages}{1--35}.
\bibitem[{Winters et~al.(2018)Winters, Moura, Mengaldo, Gassner, Walch, Peiro,
  and Sherwin}]{Winters2018}
\bibinfo{author}{A.~R. Winters}, \bibinfo{author}{R.~C. Moura},
  \bibinfo{author}{G.~Mengaldo}, \bibinfo{author}{G.~J. Gassner},
  \bibinfo{author}{S.~Walch}, \bibinfo{author}{J.~Peiro},
  \bibinfo{author}{S.~J. Sherwin},
\newblock \bibinfo{title}{{A comparative study on polynomial dealiasing and
  split form discontinuous Galerkin schemes for under-resolved turbulence
  computations}},
\newblock \bibinfo{journal}{Journal of Computational Physics}
  \bibinfo{volume}{372} (\bibinfo{year}{2018}) \bibinfo{pages}{1--21}.
\bibitem[{Fisher et~al.(2013)Fisher, Carpenter, Nordstr{\"{o}}m, Yamaleev, and
  Swanson}]{Fisher2013}
\bibinfo{author}{T.~C. Fisher}, \bibinfo{author}{M.~H. Carpenter},
  \bibinfo{author}{J.~Nordstr{\"{o}}m}, \bibinfo{author}{N.~K. Yamaleev},
  \bibinfo{author}{C.~Swanson},
\newblock \bibinfo{title}{{Discretely conservative finite-difference
  formulations for nonlinear conservation laws in split form: Theory and
  boundary conditions}},
\newblock \bibinfo{journal}{Journal of Computational Physics}
  \bibinfo{volume}{234} (\bibinfo{year}{2013}) \bibinfo{pages}{353--375}.
\bibitem[{Fisher and Carpenter(2013)}]{Fisher2013a}
\bibinfo{author}{T.~C. Fisher}, \bibinfo{author}{M.~H. Carpenter},
\newblock \bibinfo{title}{{High-order entropy stable finite difference schemes
  for nonlinear conservation laws: Finite domains}},
\newblock \bibinfo{journal}{Journal of Computational Physics}
  \bibinfo{volume}{252} (\bibinfo{year}{2013}) \bibinfo{pages}{518--557}.
\bibitem[{Gassner(2013)}]{Gassner2013}
\bibinfo{author}{G.~J. Gassner},
\newblock \bibinfo{title}{{A Skew-Symmetric Discontinuous Galerkin Spectral
  Element Discretization and Its Relation to SBP-SAT Finite Difference
  Methods}},
\newblock \bibinfo{journal}{SIAM Journal on Scientific Computing}
  \bibinfo{volume}{35} (\bibinfo{year}{2013}) \bibinfo{pages}{A1233--A1253}.
\bibitem[{Carpenter et~al.(2014)Carpenter, Fisher, Nielsen, and
  Frankel}]{Carpenter2014}
\bibinfo{author}{M.~H. Carpenter}, \bibinfo{author}{T.~C. Fisher},
  \bibinfo{author}{E.~J. Nielsen}, \bibinfo{author}{S.~H. Frankel},
\newblock \bibinfo{title}{{Entropy stable spectral collocation schemes for the
  Navier-Stokes Equations: Discontinuous interfaces}},
\newblock \bibinfo{journal}{SIAM Journal on Scientific Computing}
  \bibinfo{volume}{36} (\bibinfo{year}{2014}) \bibinfo{pages}{B835--B867}.
\bibitem[{Derigs et~al.(2018)Derigs, Winters, Gassner, Walch, and
  Bohm}]{Derigs2018}
\bibinfo{author}{D.~Derigs}, \bibinfo{author}{A.~R. Winters},
  \bibinfo{author}{G.~J. Gassner}, \bibinfo{author}{S.~Walch},
  \bibinfo{author}{M.~Bohm},
\newblock \bibinfo{title}{{Ideal GLM-MHD: About the entropy consistent
  nine-wave magnetic field divergence diminishing ideal magnetohydrodynamics
  equations}},
\newblock \bibinfo{journal}{Journal of Computational Physics}
  \bibinfo{volume}{364} (\bibinfo{year}{2018}) \bibinfo{pages}{420--467}.
\bibitem[{Persson and Peraire(2006)}]{Persson2006}
\bibinfo{author}{P.-O. Persson}, \bibinfo{author}{J.~Peraire},
\newblock \bibinfo{title}{{Sub-Cell Shock Capturing for Discontinuous Galerkin
  Methods}},
\newblock \bibinfo{journal}{44th AIAA Aerospace Sciences Meeting and Exhibit}
  (\bibinfo{year}{2006}) \bibinfo{pages}{1--13}.
\bibitem[{Kl{\"{o}}ckner et~al.(2011)Kl{\"{o}}ckner, Warburton, and
  Hesthaven}]{Klockner2011}
\bibinfo{author}{A.~Kl{\"{o}}ckner}, \bibinfo{author}{T.~Warburton},
  \bibinfo{author}{J.~S. Hesthaven},
\newblock \bibinfo{title}{{Viscous shock capturing in a time-explicit
  discontinuous Galerkin method}},
\newblock \bibinfo{journal}{Mathematical Modelling of Natural Phenomena}
  \bibinfo{volume}{6} (\bibinfo{year}{2011}) \bibinfo{pages}{57--83}.
\bibitem[{Fernandez et~al.(2018)Fernandez, Nguyen, and Peraire}]{Fernandez2018}
\bibinfo{author}{P.~Fernandez}, \bibinfo{author}{N.-C. Nguyen},
  \bibinfo{author}{J.~Peraire},
\newblock \bibinfo{title}{{A physics-based shock capturing method for
  large-eddy simulation}}  (\bibinfo{year}{2018}).
\bibitem[{Ciucă et~al.(2020)Ciucă, Fernandez, Christophe, Nguyen, and
  Peraire}]{Ciuca2020}
\bibinfo{author}{C.~Ciucă}, \bibinfo{author}{P.~Fernandez},
  \bibinfo{author}{A.~Christophe}, \bibinfo{author}{N.~C. Nguyen},
  \bibinfo{author}{J.~Peraire},
\newblock \bibinfo{title}{{Implicit hybridized discontinuous Galerkin methods
  for compressible magnetohydrodynamics}},
\newblock \bibinfo{journal}{Journal of Computational Physics: X}
  \bibinfo{volume}{5} (\bibinfo{year}{2020}).
\bibitem[{Sonntag and Munz(2014)}]{Sonntag2014}
\bibinfo{author}{M.~Sonntag}, \bibinfo{author}{C.-D. Munz},
\newblock \bibinfo{title}{{Shock Capturing for Discontinuous Galerkin Methods
  using Finite Volume Subcells}},
\newblock in: \bibinfo{booktitle}{Finite Volumes for Complex Applications
  VII-Elliptic, Parabolic and Hyperbolic Problems}, pp.
  \bibinfo{pages}{945--953}.
\bibitem[{Sonntag and Munz(2017)}]{Sonntag2017}
\bibinfo{author}{M.~Sonntag}, \bibinfo{author}{C.~D. Munz},
\newblock \bibinfo{title}{{Efficient Parallelization of a Shock Capturing for
  Discontinuous Galerkin Methods using Finite Volume Sub-cells}},
\newblock \bibinfo{journal}{Journal of Scientific Computing}
  \bibinfo{volume}{70} (\bibinfo{year}{2017}) \bibinfo{pages}{1262--1289}.
\bibitem[{Sonntag(2017)}]{Sonntag2017a}
\bibinfo{author}{M.~Sonntag}, \bibinfo{title}{{Shape derivatives and shock
  capturing for the Navier-Stokes equations in discontinuous Galerkin
  methods}}, Ph.D. thesis, University of Stuttgart, \bibinfo{year}{2017}.
\bibitem[{{N{\'{u}}{\~{n}}ez-de la Rosa} and Munz(2018)}]{Nunez-delaRosa2018}
\bibinfo{author}{J.~{N{\'{u}}{\~{n}}ez-de la Rosa}}, \bibinfo{author}{C.~D.
  Munz},
\newblock \bibinfo{title}{{Hybrid DG/FV schemes for magnetohydrodynamics and
  relativistic hydrodynamics}},
\newblock \bibinfo{journal}{Computer Physics Communications}
  \bibinfo{volume}{222} (\bibinfo{year}{2018}) \bibinfo{pages}{113--135}.
\bibitem[{Markert et~al.(2020)Markert, Gassner, and Walch}]{Markert2020}
\bibinfo{author}{J.~Markert}, \bibinfo{author}{G.~Gassner},
  \bibinfo{author}{S.~Walch},
\newblock \bibinfo{title}{{A Sub-Element Adaptive Shock Capturing Approach for
  Discontinuous Galerkin Methods}},
\newblock \bibinfo{journal}{http://arxiv.org/abs/2011.03338}
  (\bibinfo{year}{2020}).
\bibitem[{Vilar(2019)}]{Vilar2019}
\bibinfo{author}{F.~Vilar},
\newblock \bibinfo{title}{{A posteriori correction of high-order discontinuous
  Galerkin scheme through subcell finite volume formulation and flux
  reconstruction}},
\newblock \bibinfo{journal}{Journal of Computational Physics}
  \bibinfo{volume}{387} (\bibinfo{year}{2019}) \bibinfo{pages}{245--279}.
\bibitem[{Hennemann et~al.(2020)Hennemann, Rueda-Ram{\'{i}}rez, Hindenlang, and
  Gassner}]{Hennemann2020}
\bibinfo{author}{S.~Hennemann}, \bibinfo{author}{A.~M. Rueda-Ram{\'{i}}rez},
  \bibinfo{author}{F.~J. Hindenlang}, \bibinfo{author}{G.~J. Gassner},
\newblock \bibinfo{title}{{A provably entropy stable subcell shock capturing
  approach for high order split form DG for the compressible Euler equations}},
\newblock \bibinfo{journal}{Journal of Computational Physics}
  (\bibinfo{year}{2020}) \bibinfo{pages}{109935}.
\bibitem[{Pazner(2020)}]{Pazner2020}
\bibinfo{author}{W.~Pazner},
\newblock \bibinfo{title}{{Sparse invariant domain preserving discontinuous
  Galerkin methods with subcell convex limiting}},
\newblock \bibinfo{journal}{arXiv preprint arXiv:2004.08503}
  (\bibinfo{year}{2020}).
\bibitem[{Liu et~al.(2018)Liu, Shu, and Zhang}]{Liu2018}
\bibinfo{author}{Y.~Liu}, \bibinfo{author}{C.~W. Shu},
  \bibinfo{author}{M.~Zhang},
\newblock \bibinfo{title}{{Entropy stable high order discontinuous Galerkin
  methods for ideal compressible MHD on structured meshes}},
\newblock \bibinfo{journal}{Journal of Computational Physics}
  \bibinfo{volume}{354} (\bibinfo{year}{2018}) \bibinfo{pages}{163--178}.
\bibitem[{Gassner et~al.(2018)Gassner, Winters, Hindenlang, and
  Kopriva}]{Gassner2018}
\bibinfo{author}{G.~J. Gassner}, \bibinfo{author}{A.~R. Winters},
  \bibinfo{author}{F.~J. Hindenlang}, \bibinfo{author}{D.~A. Kopriva},
  \bibinfo{title}{{The BR1 Scheme is Stable for the Compressible Navier –
  Stokes Equations}}, volume~\bibinfo{volume}{77}, \bibinfo{publisher}{Springer
  US}, \bibinfo{year}{2018}.
\bibitem[{Rueda-Ram{\'{i}}rez et~al.(2020)Rueda-Ram{\'{i}}rez, Ferrer, Kopriva,
  Rubio, and Valero}]{RuedaRamirez2020}
\bibinfo{author}{A.~M. Rueda-Ram{\'{i}}rez}, \bibinfo{author}{E.~Ferrer},
  \bibinfo{author}{D.~A. Kopriva}, \bibinfo{author}{G.~Rubio},
  \bibinfo{author}{E.~Valero},
\newblock \bibinfo{title}{{A statically condensed discontinuous Galerkin
  spectral element method on Gauss-Lobatto nodes for the compressible
  Navier-Stokes equations}},
\newblock \bibinfo{journal}{Journal of Computational Physics}
  (\bibinfo{year}{2020}).
\bibitem[{Ismail and Roe(2009)}]{Ismail2009}
\bibinfo{author}{F.~Ismail}, \bibinfo{author}{P.~L. Roe},
\newblock \bibinfo{title}{{Affordable, entropy-consistent Euler flux functions
  II: Entropy production at shocks}},
\newblock \bibinfo{journal}{Journal of Computational Physics}
  \bibinfo{volume}{228} (\bibinfo{year}{2009}) \bibinfo{pages}{5410--5436}.
\bibitem[{Munz et~al.(2000)Munz, Omnes, Schneider, Sonnendr{\"{u}}cker, and
  Vo{\ss}}]{Munz2000}
\bibinfo{author}{C.~D. Munz}, \bibinfo{author}{P.~Omnes},
  \bibinfo{author}{R.~Schneider}, \bibinfo{author}{E.~Sonnendr{\"{u}}cker},
  \bibinfo{author}{U.~Vo{\ss}},
\newblock \bibinfo{title}{{Divergence Correction Techniques for Maxwell Solvers
  Based on a Hyperbolic Model}},
\newblock \bibinfo{journal}{Journal of Computational Physics}
  \bibinfo{volume}{161} (\bibinfo{year}{2000}) \bibinfo{pages}{484--511}.
\bibitem[{Dedner et~al.(2002)Dedner, Kemm, Kr{\"{o}}ner, Munz, Schnitzer, and
  Wesenberg}]{Dedner2002}
\bibinfo{author}{A.~Dedner}, \bibinfo{author}{F.~Kemm},
  \bibinfo{author}{D.~Kr{\"{o}}ner}, \bibinfo{author}{C.~D. Munz},
  \bibinfo{author}{T.~Schnitzer}, \bibinfo{author}{M.~Wesenberg},
\newblock \bibinfo{title}{{Hyperbolic divergence cleaning for the MHD
  equations}},
\newblock \bibinfo{journal}{Journal of Computational Physics}
  \bibinfo{volume}{175} (\bibinfo{year}{2002}) \bibinfo{pages}{645--673}.
\bibitem[{Powell et~al.(1999)Powell, Roe, Linde, Gombosi, and {De
  Zeeuw}}]{Powell2001}
\bibinfo{author}{K.~G. Powell}, \bibinfo{author}{P.~L. Roe},
  \bibinfo{author}{T.~J. Linde}, \bibinfo{author}{T.~I. Gombosi},
  \bibinfo{author}{D.~L. {De Zeeuw}},
\newblock \bibinfo{title}{{A Solution-Adaptive Upwind Scheme for Ideal
  Magnetohydrodynamics}},
\newblock \bibinfo{journal}{Journal of Computational Physics}
  \bibinfo{volume}{154} (\bibinfo{year}{1999}) \bibinfo{pages}{284--309}.
\bibitem[{Derigs et~al.(2017)Derigs, Winters, Gassner, and Walch}]{Derigs2017}
\bibinfo{author}{D.~Derigs}, \bibinfo{author}{A.~R. Winters},
  \bibinfo{author}{G.~J. Gassner}, \bibinfo{author}{S.~Walch},
\newblock \bibinfo{title}{{A novel averaging technique for discrete
  entropy-stable dissipation operators for ideal MHD}},
\newblock \bibinfo{journal}{Journal of Computational Physics}
  \bibinfo{volume}{330} (\bibinfo{year}{2017}) \bibinfo{pages}{624--632}.
\bibitem[{Bassi and Rebay(1997)}]{Bassi1997}
\bibinfo{author}{F.~Bassi}, \bibinfo{author}{S.~Rebay},
\newblock \bibinfo{title}{{A high-order accurate discontinuous finite element
  method for the numerical solution of the compressible Navier-Stokes
  equations}},
\newblock \bibinfo{journal}{Journal of Computational Physics}
  \bibinfo{volume}{131} (\bibinfo{year}{1997}) \bibinfo{pages}{267--279}.
\bibitem[{Arnold et~al.(2002)Arnold, Brezzi, Cockburn, and Marini}]{Arnold2002}
\bibinfo{author}{D.~N. Arnold}, \bibinfo{author}{F.~Brezzi},
  \bibinfo{author}{B.~Cockburn}, \bibinfo{author}{D.~Marini},
\newblock \bibinfo{title}{{Unified analysis of discontinuous Galerkin methods
  for elliptic problems}},
\newblock \bibinfo{journal}{SIAM J. Numer. Anal.} \bibinfo{volume}{39}
  (\bibinfo{year}{2002}) \bibinfo{pages}{1749--1779}.
\bibitem[{Chandrashekar and Klingenberg(2016)}]{Chandrashekar2016}
\bibinfo{author}{P.~Chandrashekar}, \bibinfo{author}{C.~Klingenberg},
\newblock \bibinfo{title}{{Entropy stable finite volume scheme for ideal
  compressible MHD on 2-D Cartesian meshes}},
\newblock \bibinfo{journal}{SIAM Journal on Numerical Analysis}
  \bibinfo{volume}{54} (\bibinfo{year}{2016}) \bibinfo{pages}{1313--1340}.
\bibitem[{Tadmor(1986)}]{tadmor1986minimum}
\bibinfo{author}{E.~Tadmor},
\newblock \bibinfo{title}{{A minimum entropy principle in the gas dynamics
  equations}},
\newblock \bibinfo{journal}{Applied Numerical Mathematics} \bibinfo{volume}{2}
  (\bibinfo{year}{1986}) \bibinfo{pages}{211--219}.
\bibitem[{Tadmor(1987)}]{tadmor1983entropy}
\bibinfo{author}{E.~Tadmor},
\newblock \bibinfo{title}{{Entropy functions for symmetric systems of
  conservation laws}},
\newblock \bibinfo{journal}{Journal of Mathematical Analysis and Applications}
  \bibinfo{volume}{122} (\bibinfo{year}{1987}) \bibinfo{pages}{355--359}.
\bibitem[{Tadmor(2003)}]{Tadmor2003}
\bibinfo{author}{E.~Tadmor},
\newblock \bibinfo{title}{{Entropy stability theory for difference
  approximations of nonlinear conservation laws and related time-dependent
  problems}},
\newblock \bibinfo{journal}{Acta Numerica} \bibinfo{volume}{12}
  (\bibinfo{year}{2003}) \bibinfo{pages}{451--512}.
\bibitem[{Renac(2019)}]{Renac2019}
\bibinfo{author}{F.~Renac},
\newblock \bibinfo{title}{{Entropy stable DGSEM for nonlinear hyperbolic
  systems in nonconservative form with application to two-phase flows}},
\newblock \bibinfo{journal}{Journal of Computational Physics}
  \bibinfo{volume}{382} (\bibinfo{year}{2019}) \bibinfo{pages}{1--26}.
\bibitem[{Manzanero(2020)}]{Manzanero2020}
\bibinfo{author}{J.~Manzanero}, \bibinfo{title}{{A high-order discontinuous
  Galerkin multiphase flow solver for industrial applications}}, Ph.D. thesis,
  Universidad Polit{\'{e}}cnica de Madrid, \bibinfo{year}{2020}.
\bibitem[{Barth(1999)}]{barth1999numerical}
\bibinfo{author}{T.~J. Barth},
\newblock \bibinfo{title}{{Numerical methods for gasdynamic systems on
  unstructured meshes}},
\newblock in: \bibinfo{booktitle}{An introduction to recent developments in
  theory and numerics for conservation laws}, \bibinfo{publisher}{Springer},
  \bibinfo{year}{1999}, pp. \bibinfo{pages}{195--285}.
\bibitem[{Winters et~al.(2017)Winters, Derigs, Gassner, and
  Walch}]{Winters2017}
\bibinfo{author}{A.~R. Winters}, \bibinfo{author}{D.~Derigs},
  \bibinfo{author}{G.~J. Gassner}, \bibinfo{author}{S.~Walch},
\newblock \bibinfo{title}{{A uniquely defined entropy stable matrix dissipation
  operator for high Mach number ideal MHD and compressible Euler simulations}},
\newblock \bibinfo{journal}{Journal of Computational Physics}
  \bibinfo{volume}{332} (\bibinfo{year}{2017}) \bibinfo{pages}{274--289}.
\bibitem[{Roe and Balsara(1996)}]{Roe1996}
\bibinfo{author}{P.~L. Roe}, \bibinfo{author}{D.~S. Balsara},
\newblock \bibinfo{title}{{Notes on the eigensystem of magnetohydrodynamics}},
\newblock \bibinfo{journal}{SIAM Journal on Applied Mathematics}
  \bibinfo{volume}{56} (\bibinfo{year}{1996}) \bibinfo{pages}{57--67}.
\bibitem[{van Leer(1974)}]{VanLeer1974}
\bibinfo{author}{B.~van Leer},
\newblock \bibinfo{title}{{Towards the ultimate conservative difference scheme.
  II. Monotonicity and conservation combined in a second-order scheme}},
\newblock \bibinfo{journal}{Journal of Computational Physics}
  \bibinfo{volume}{14} (\bibinfo{year}{1974}) \bibinfo{pages}{361--370}.
\bibitem[{Coquel and LeFloch(1996)}]{Coquel1996}
\bibinfo{author}{F.~Coquel}, \bibinfo{author}{P.~G. LeFloch},
\newblock \bibinfo{title}{{An entropy satisfying MUSCL scheme for systems of
  conservation laws}},
\newblock \bibinfo{journal}{Numerische Mathematik} \bibinfo{volume}{74}
  (\bibinfo{year}{1996}) \bibinfo{pages}{1--33}.
\bibitem[{Coquel et~al.(2006)Coquel, Helluy, and Schneider}]{Coquel2006}
\bibinfo{author}{F.~Coquel}, \bibinfo{author}{P.~Helluy},
  \bibinfo{author}{J.~Schneider},
\newblock \bibinfo{title}{{Second-order entropy diminishing scheme for the
  Euler equations}},
\newblock \bibinfo{journal}{International Journal for Numerical Methods in
  Fluids} \bibinfo{volume}{50} (\bibinfo{year}{2006})
  \bibinfo{pages}{1029--1061}.
\bibitem[{Fjordholm et~al.(2012)Fjordholm, Mishra, and Tadmor}]{Fjordholm2012}
\bibinfo{author}{U.~S. Fjordholm}, \bibinfo{author}{S.~Mishra},
  \bibinfo{author}{E.~Tadmor},
\newblock \bibinfo{title}{{Arbitrarily high-order accurate entropy stable
  essentially nonoscillatory schemes for systems of conservation laws}},
\newblock \bibinfo{journal}{SIAM Journal on Numerical Analysis}
  \bibinfo{volume}{50} (\bibinfo{year}{2012}) \bibinfo{pages}{544--573}.
\bibitem[{Biswas and Dubey(2018)}]{Biswas2018}
\bibinfo{author}{B.~Biswas}, \bibinfo{author}{R.~K. Dubey},
\newblock \bibinfo{title}{{Low dissipative entropy stable schemes using third
  order WENO and TVD reconstructions}},
\newblock \bibinfo{journal}{Advances in Computational Mathematics}
  \bibinfo{volume}{44} (\bibinfo{year}{2018}) \bibinfo{pages}{1153--1181}.
\bibitem[{Winters and Gassner(2016)}]{Winters2016}
\bibinfo{author}{A.~R. Winters}, \bibinfo{author}{G.~J. Gassner},
\newblock \bibinfo{title}{{Affordable, entropy conserving and entropy stable
  flux functions for the ideal MHD equations}},
\newblock \bibinfo{journal}{Journal of Computational Physics}
  \bibinfo{volume}{304} (\bibinfo{year}{2016}) \bibinfo{pages}{72--108}.
\bibitem[{Spiteri and Ruuth(2002)}]{Spiteri2002}
\bibinfo{author}{R.~J. Spiteri}, \bibinfo{author}{S.~J. Ruuth},
\newblock \bibinfo{title}{{A new class of optimal high-order
  strong-stability-preserving time discretization methods}},
\newblock \bibinfo{journal}{SIAM Journal on Numerical Analysis}
  \bibinfo{volume}{40} (\bibinfo{year}{2002}) \bibinfo{pages}{469--491}.
\bibitem[{Krais et~al.(2020)Krais, Beck, Bolemann, Frank, Flad, Gassner,
  Hindenlang, Hoffmann, Kuhn, Sonntag, and Munz}]{Krais2019}
\bibinfo{author}{N.~Krais}, \bibinfo{author}{A.~Beck},
  \bibinfo{author}{T.~Bolemann}, \bibinfo{author}{H.~Frank},
  \bibinfo{author}{D.~Flad}, \bibinfo{author}{G.~Gassner},
  \bibinfo{author}{F.~Hindenlang}, \bibinfo{author}{M.~Hoffmann},
  \bibinfo{author}{T.~Kuhn}, \bibinfo{author}{M.~Sonntag},
  \bibinfo{author}{C.~D. Munz},
\newblock \bibinfo{title}{{FLEXI: A high order discontinuous Galerkin framework
  for hyperbolic-parabolic conservation laws}},
\newblock \bibinfo{journal}{Computers \& Mathematics with Applications}
  \bibinfo{volume}{In Press} (\bibinfo{year}{2020}).
\bibitem[{Chan et~al.(2019)Chan, Fernandez, Carpenter, {Del Rey
  Fern{\'{a}}ndez}, and Carpenter}]{chan2019efficient}
\bibinfo{author}{J.~Chan}, \bibinfo{author}{D.~C. D.~R. Fernandez},
  \bibinfo{author}{M.~H. Carpenter}, \bibinfo{author}{D.~C. {Del Rey
  Fern{\'{a}}ndez}}, \bibinfo{author}{M.~H. Carpenter},
\newblock \bibinfo{title}{{Efficient entropy stable Gauss collocation
  methods}},
\newblock \bibinfo{journal}{SIAM Journal on Scientific Computing}
  \bibinfo{volume}{41} (\bibinfo{year}{2019}) \bibinfo{pages}{A2938----A2966}.
\bibitem[{Hindenlang et~al.(2015)Hindenlang, Bolemann, and
  Munz}]{hindenlang2015mesh}
\bibinfo{author}{F.~Hindenlang}, \bibinfo{author}{T.~Bolemann},
  \bibinfo{author}{C.-D. Munz}, \bibinfo{title}{{Mesh Curving Techniques for
  High Order Discontinuous Galerkin Simulations}}, Ph.D. thesis, University of
  Stuttgart, \bibinfo{year}{2015}.
\bibitem[{Orszag and Tang(1979)}]{orszag1979small}
\bibinfo{author}{S.~A. Orszag}, \bibinfo{author}{C.-M. Tang},
\newblock \bibinfo{title}{{Small-scale structure of two-dimensional
  magnetohydrodynamic turbulence}},
\newblock \bibinfo{journal}{Journal of Fluid Mechanics} \bibinfo{volume}{90}
  (\bibinfo{year}{1979}) \bibinfo{pages}{129--143}.
\bibitem[{Stone et~al.(2008)Stone, Gardiner, Teuben, Hawley, and
  Simon}]{Stone2008}
\bibinfo{author}{J.~M. Stone}, \bibinfo{author}{T.~A. Gardiner},
  \bibinfo{author}{P.~Teuben}, \bibinfo{author}{J.~F. Hawley},
  \bibinfo{author}{J.~B. Simon},
\newblock \bibinfo{title}{{Athena: A New Code for Astrophysical MHD}},
\newblock \bibinfo{journal}{The Astrophysical Journal Supplement Series}
  \bibinfo{volume}{178} (\bibinfo{year}{2008}) \bibinfo{pages}{137--177}.
\bibitem[{Birn et~al.(2001)Birn, Drake, Shay, Rogers, Denton, Hesse,
  Kuznetsova, Ma, Bhattacharjee, Otto, and Pritchett}]{Birn2001}
\bibinfo{author}{J.~Birn}, \bibinfo{author}{J.~F. Drake},
  \bibinfo{author}{M.~A. Shay}, \bibinfo{author}{B.~N. Rogers},
  \bibinfo{author}{R.~E. Denton}, \bibinfo{author}{M.~Hesse},
  \bibinfo{author}{M.~Kuznetsova}, \bibinfo{author}{Z.~W. Ma},
  \bibinfo{author}{A.~Bhattacharjee}, \bibinfo{author}{A.~Otto},
  \bibinfo{author}{P.~L. Pritchett},
\newblock \bibinfo{title}{{Geospace Environmental Modeling (GEM) Magnetic
  Reconnection Challenge}},
\newblock \bibinfo{journal}{Journal of Geophysical Research: Space Physics}
  \bibinfo{volume}{106} (\bibinfo{year}{2001}) \bibinfo{pages}{3715--3719}.
\bibitem[{Helander et~al.(2002)Helander, Eriksson, and
  Andersson}]{Helander2002}
\bibinfo{author}{P.~Helander}, \bibinfo{author}{L.~G. Eriksson},
  \bibinfo{author}{F.~Andersson},
\newblock \bibinfo{title}{{Runaway acceleration during magnetic reconnection in
  tokamaks}},
\newblock \bibinfo{journal}{Plasma Physics and Controlled Fusion}
  \bibinfo{volume}{44} (\bibinfo{year}{2002}).
\bibitem[{Ono et~al.(2012)Ono, Tanabe, Yamada, Inomoto, Ii, Inoue, Gi,
  Watanabe, Gryaznevich, Scannell, Michael, and Cheng}]{Ono2012}
\bibinfo{author}{Y.~Ono}, \bibinfo{author}{H.~Tanabe},
  \bibinfo{author}{T.~Yamada}, \bibinfo{author}{M.~Inomoto},
  \bibinfo{author}{T.~Ii}, \bibinfo{author}{S.~Inoue}, \bibinfo{author}{K.~Gi},
  \bibinfo{author}{T.~Watanabe}, \bibinfo{author}{M.~Gryaznevich},
  \bibinfo{author}{R.~Scannell}, \bibinfo{author}{C.~Michael},
  \bibinfo{author}{C.~Z. Cheng},
\newblock \bibinfo{title}{{Ion and electron heating characteristics of magnetic
  reconnection in tokamak plasma merging experiments}},
\newblock \bibinfo{journal}{Plasma Physics and Controlled Fusion}
  \bibinfo{volume}{54} (\bibinfo{year}{2012}).
\bibitem[{Mignone et~al.(2012)Mignone, Zanni, Tzeferacos, {Van Straalen},
  Colella, and Bodo}]{Mignone2012}
\bibinfo{author}{A.~Mignone}, \bibinfo{author}{C.~Zanni},
  \bibinfo{author}{P.~Tzeferacos}, \bibinfo{author}{B.~{Van Straalen}},
  \bibinfo{author}{P.~Colella}, \bibinfo{author}{G.~Bodo},
\newblock \bibinfo{title}{{The PLUTO code for adaptive mesh computations in
  astrophysical fluid dynamics}},
\newblock \bibinfo{journal}{Astrophysical Journal, Supplement Series}
  \bibinfo{volume}{198} (\bibinfo{year}{2012}).
\bibitem[{Sousa et~al.(2015)Sousa, Lin, and Shumlak}]{Sousa2015}
\bibinfo{author}{{\'{E}}.~M. Sousa}, \bibinfo{author}{G.~Lin},
  \bibinfo{author}{U.~Shumlak},
\newblock \bibinfo{title}{{Uncertainty quantification of the gem challenge
  magnetic reconnection problem using the multilevel Monte Carlo method}},
\newblock \bibinfo{journal}{International Journal for Uncertainty
  Quantification} \bibinfo{volume}{5} (\bibinfo{year}{2015})
  \bibinfo{pages}{327--339}.
\bibitem[{Kivelson et~al.(2003)Kivelson, Bagenal, Kurth, Neubauer, Paranicas,
  and Saur}]{Kivelson2003}
\bibinfo{author}{M.~G. Kivelson}, \bibinfo{author}{F.~Bagenal},
  \bibinfo{author}{W.~S. Kurth}, \bibinfo{author}{F.~M. Neubauer},
  \bibinfo{author}{C.~Paranicas}, \bibinfo{author}{J.~Saur},
\newblock \bibinfo{title}{{Magnetospheric interactions with satellites}},
\newblock \bibinfo{journal}{Jupiter: The Planet} \bibinfo{volume}{m}
  (\bibinfo{year}{2003}) \bibinfo{pages}{1--24}.
\bibitem[{Saur et~al.(2004)Saur, Neubauer, and Connerney}]{Saur2004}
\bibinfo{author}{J.~Saur}, \bibinfo{author}{F.~M. Neubauer},
  \bibinfo{author}{J.~E.~P. Connerney}, \bibinfo{title}{{Plasma interaction of
  Io with its plasma torus}}, \bibinfo{year}{2004}.
\bibitem[{Kivelson et~al.(1997)Kivelson, Khurana, Russell, Walker, Joy, and
  Mafi}]{Kivelson1997}
\bibinfo{author}{M.~Kivelson}, \bibinfo{author}{K.~Khurana},
  \bibinfo{author}{C.~Russell}, \bibinfo{author}{R.~Walker},
  \bibinfo{author}{S.~Joy}, \bibinfo{author}{J.~Mafi}, \bibinfo{title}{{GALILEO
  ORBITER AT JUPITER CALIBRATED MAG HIGH RES V1.0,
  GO-J-MAG-3-RDR-HIGHRES-V1.0}}, \bibinfo{type}{Technical Report}, NASA
  Planetary Data System, \bibinfo{year}{1997}.
\bibitem[{Jacobsen et~al.(2007)Jacobsen, Neubauer, Saur, and
  Schilling}]{Jacobsen2007}
\bibinfo{author}{S.~Jacobsen}, \bibinfo{author}{F.~M. Neubauer},
  \bibinfo{author}{J.~Saur}, \bibinfo{author}{N.~Schilling},
\newblock \bibinfo{title}{{Io's nonlinear MHD-wave field in the heterogeneous
  Jovian magnetosphere}},
\newblock \bibinfo{journal}{Geophysical Research Letters} \bibinfo{volume}{34}
  (\bibinfo{year}{2007}) \bibinfo{pages}{1--5}.
\bibitem[{Bl{\"{o}}cker(2017)}]{Blocker2017}
\bibinfo{author}{A.~Bl{\"{o}}cker}, \bibinfo{title}{{Modeling Io's and Europa's
  Plasma Interaction with the Jovian Magnetosphere: Influence of Global
  Atmospheric Asymmetries and Plumes}}, Ph.D. thesis, Universit{\"{a}}t zu
  K{\"{o}}ln, \bibinfo{year}{2017}.
\bibitem[{Bl{\"{o}}cker et~al.(2018)Bl{\"{o}}cker, Saur, Roth, and
  Strobel}]{Blocker2018}
\bibinfo{author}{A.~Bl{\"{o}}cker}, \bibinfo{author}{J.~Saur},
  \bibinfo{author}{L.~Roth}, \bibinfo{author}{D.~F. Strobel},
\newblock \bibinfo{title}{{MHD Modeling of the Plasma Interaction With Io's
  Asymmetric Atmosphere}},
\newblock \bibinfo{journal}{Journal of Geophysical Research: Space Physics}
  \bibinfo{volume}{123} (\bibinfo{year}{2018}) \bibinfo{pages}{9286--9311}.
\bibitem[{Bohm(2019)}]{Bohm2019}
\bibinfo{author}{M.~Bohm}, \bibinfo{title}{{An entropy stable nodal
  discontinuous Galerkin method for the resistive MHD equations}}, Ph.D.
  thesis, University of Cologne, \bibinfo{year}{2019}.
\bibitem[{Duling et~al.(2014)Duling, Saur, and Wicht}]{Duling2014}
\bibinfo{author}{S.~Duling}, \bibinfo{author}{J.~Saur},
  \bibinfo{author}{J.~Wicht},
\newblock \bibinfo{title}{{Consistent boundary conditions at nonconducting
  surfaces of planetary bodies: Applications in a new Ganymede MHD model}},
\newblock \bibinfo{journal}{Journal of Geophysical Research: Space Physics}
  \bibinfo{volume}{119} (\bibinfo{year}{2014}) \bibinfo{pages}{4412--4440}.
\bibitem[{Bl{\"{o}}cker et~al.(2016)Bl{\"{o}}cker, Saur, and
  Roth}]{Blocker2016}
\bibinfo{author}{A.~Bl{\"{o}}cker}, \bibinfo{author}{J.~Saur},
  \bibinfo{author}{L.~Roth},
\newblock \bibinfo{title}{{Europa's plasma interaction with an inhomogeneous
  atmosphere: Development of Alfv{\'{e}}n winglets within the Alfv{\'{e}}n
  wings}},
\newblock \bibinfo{journal}{Journal of Geophysical Research: Space Physics}
  \bibinfo{volume}{121} (\bibinfo{year}{2016}) \bibinfo{pages}{9794--9828}.
\bibitem[{Kopriva(2006)}]{Kopriva2006}
\bibinfo{author}{D.~A. Kopriva},
\newblock \bibinfo{title}{{Metric identities and the discontinuous spectral
  element method on curvilinear meshes}},
\newblock \bibinfo{journal}{Journal of Scientific Computing}
  \bibinfo{volume}{26} (\bibinfo{year}{2006}) \bibinfo{pages}{301--327}.

\end{thebibliography}

\section*{Appendices}
\appendix

\section{Entropy Conservative Flux Function} \label{sec:ECflux}

Using the definition of the numerical non-conservative terms, \eqref{eq:DiamondFluxes}, \citet{Derigs2018} algebraically constructed an entropy conserving numerical flux for the GLM-MHD system:
\begin{equation}\label{eq:ECFlux}
\state{f}_1^{\ec}(\state{u}_L,\state{u}_R) =
\begin{pmatrix} 
\rho^{\ln}_{(L,R)} \avg{v_1} \\
\rho^{\ln}_{(L,R)} \avg{v_1}^2 - \avg{B_1}^2 + \overline{p} + \frac{1}{2} \Big(\avg{B_1 B_1} + \avg{B_2 B_2} + \avg{B_3 B_3}\Big)\\ 
\rho^{\ln}_{(L,R)} \avg{v_1} \avg{v_2} - \avg{B_1} \avg{B_2} \\
\rho^{\ln}_{(L,R)} \avg{v_1} \avg{v_3} - \avg{B_1} \avg{B_3} \\
f_{1,5}^{\ec} \\
c_h \avg{\psi} \\
\avg{v_1}\avg{B_2} - \avg{v_2}\avg{B_1}\\
\avg{v_1}\avg{B_3} - \avg{v_3}\avg{B_1}\\
c_h \avg{B_1} 
\end{pmatrix}
\end{equation}
with 
\begin{equation}\label{eq:ECFlux2}
\begin{split}
f_{1,5}^{\ec} = & f_{1,1}^{\ec}\bigg[\frac{1}{2 (\gamma-1) \beta^{\ln}_{(L,R)}} - \frac{1}{2} \left(\avg{v_1^2} + \avg{v_2^2} + \avg{v_3^2}\right) \bigg] + f_{1,2}^{\ec} \avg{v_1} + f_{1,3}^{\ec} \avg{v_2} + f_{1,4}^{\ec} \avg{v_3} \\
 & + f_{1,6}^{\ec} \avg{B_1} + f_{1,7}^{\ec} \avg{B_2} + f_{1,8}^{\ec} \avg{B_3} + f_{1,9}^{\ec} \avg{\psi} - \frac{1}{2} \big(\avg{v_1 B_1^2}+\avg{v_1 B_2^2}+\avg{v_1 B_3^2}\big)\\
 & + \avg{v_1 B_1} \avg{B_1}+\avg{v_2 B_2} \avg{B_1}+\avg{v_3 B_3} \avg{B_1} - c_h \avg{B_1 \psi}
\end{split}
\end{equation}
and
\begin{equation*}
\overline{p} = \frac{\avg{\rho}}{2\avg{\beta}}.
\end{equation*}

\section{One-Dimensional Derivations}

\subsection{Entropy Balance of the Native LGL FV Method (Proof of Lemma \ref{lemma:EntropyFV})} \label{sec:Proof_FV_1D}
The entropy balance within a subcell is obtained by contracting \eqref{eq:FVMHD_D} with the entropy variables to obtain
\begin{equation} \label{eq:EntropyFV1}
J \omega_j \entVar_j^T \dot{\state{u}}_j   = 
J \omega_j \dot{S}_j =
\entVar^T_j  \numfluxb{f}^a_{(j,j-1)} 
- \entVar^T_j  \numfluxb{f}^a_{(j,j+1)} 
+ \entVar^T_j  \numnonconsD{\Jan}_{(j,j-1)}
- \entVar^T_j  \numnonconsD{\Jan}_{(j,j+1)}.
\end{equation}

We now rearrange \eqref{eq:EntropyFV1}, take advantage of the symmetry of the numerical fluxes, and sum and subtract the following terms to separate the right-hand-side in symmetric and antisymmetric contributions
\begin{align} \label{eq:EntropyFV}
J \omega_j \dot{S}_j =&
\entVar^T_j  \numfluxb{f}^a_{(j,j-1)} 
+ \entVar^T_j  \numnonconsD{\Jan}_{(j,j-1)}
- \entVar^T_j  \numfluxb{f}^a_{(j,j+1)} 
- \entVar^T_j  \numnonconsD{\Jan}_{(j,j+1)} \nonumber\\
&
+ \frac{1}{2} \left(
\entVar^T_{j-1}  \numfluxb{f}^a_{(j-1,j)} 
+ \entVar^T_{j-1}  \numnonconsD{\Jan}_{(j-1,j)} 
- \Psi_{j-1} 
- \entVar^T_{j+1}  \numfluxb{f}^a_{(j+1,j)}  
- \entVar^T_{j+1}  \numnonconsD{\Jan}_{(j+1,j)}  
- \Psi_{j+1}
\right)
- \Psi_j
\nonumber\\
&
- \frac{1}{2} \left(
\entVar^T_{j-1}  \numfluxb{f}^a_{(j-1,j)} 
+ \entVar^T_{j-1}  \numnonconsD{\Jan}_{(j-1,j)} 
- \Psi_{j-1} 
- \entVar^T_{j+1}  \numfluxb{f}^a_{(j+1,j)}  
- \entVar^T_{j+1}  \numnonconsD{\Jan}_{(j+1,j)}  
- \Psi_{j+1}
\right)
+ \Psi_j
\nonumber\\
=&
\numflux{f}^S_{(j-1,j)} -\numflux{f}^S_{(j,j+1)} + \frac{1}{2} \left( r_{(j-1,j)} + r_{(N,R)} \right),
\end{align}
where we used the definition of the numerical entropy flux, \eqref{eq:numEntFlux}, and the entropy production, \eqref{eq:EntropyDissipation}, and took advantage of the antisymmetry of the latter.

\subsection{Entropy Balance of the DGSEM (Proof of Lemma \ref{lemma:EntropyDGSEM})} \label{sec:Proof_DGSEM_1D}

We start by rewriting the volume non-conservative term, \eqref{eq:volNonCons_D1}, using the definition of the numerical non-conservative term, \eqref{eq:volNonCons_D1},
\begin{align} \label{eq:volNonCons_D}
\Jan^{*}_{(j,k)} &= 2 \numnonconsD{\Jan}_{(j,k)} - \Jan_j
\end{align}

The advantage of using a discretization on Legendre-Gauss-Lobatto nodes is that the derivative operator fulfills the SBP property \cite{Gassner2013},
\begin{equation} \label{eq:SBPprop}
\mat{Q} + \mat{Q}^T = \mat{B},
\end{equation}
with the boundary matrix
\begin{equation} \label{eq:Bmatrix}
\mat{B} = \mathrm{diag}([-1, 0, \dots, 0, 1]),
\end{equation}
which will be useful in future derivations.
Furthermore, the SBP property leads to the additional properties,
\begin{align}
	\sum_{k=0}^{N} Q_{jk} &= 0, \label{eq:SBP1}\\
	\sum_{j=0}^{N} Q_{jk} &= \delta_{kN} - \delta_{k0} = B_{kk}. \label{eq:SBP2}
\end{align}

The entropy balance for the advective and non-conservative terms is obtained by contracting \eqref{eq:DGSEMMHDadv2} with the entropy variables and integrating over an element,
\begin{align}
\sum_{j=0}^N \omega_j J \dot{S}^a_j 
  =& \sum_{j=0}^N \entVar^T_j \state{F}_j^{a,\DG}
\\=& - \underbrace{\sum_{j=0}^{N} \entVar^T_j 
\left[
\sum_{k=0}^N 2 Q_{jk} \state{f}^{*}_{(j,k)}
-  \delta_{jN} \state{f}^a_N + \delta_{j0} \state{f}^a_0
+ \sum_{k=0}^N Q_{jk} \Jan^{*}_{(j,k)} 
-  \delta_{jN} \Jan_N + \delta_{j0} \Jan_0
\right]
}_{(a):= \text{ volume terms}} \nonumber\\
&
- \underbrace{\sum_{j=0}^{N} \entVar^T_j  \left[
\delta_{jN}  \numfluxb{f}^a_{(N,R)} 
- \delta_{j0}  \numfluxb{f}^a_{(0,L)} 
+ \delta_{jN} \numnonconsD{\Jan}_{(N,R)} 
- \delta_{j0} \numnonconsD{\Jan}_{(0,L)} 
\right] 
}_{\text{surface terms}} \label{eq:}
\end{align}

Let us first manipulate the volume terms:
\begin{align} \label{eq:EntropyDGSEM_1D}
(a)
	=& \sum_{j=0}^{N} \entVar^T_j \sum_{k=0}^N 2 Q_{jk} \state{f}^{*}_{(j,k)} + \entVar^T_0 \state{f}^a_0 - \entVar^T_N \state{f}^a_N  \\
	& + \sum_{j=0}^{N} \entVar^T_j \sum_{k=0}^N Q_{jk} \Jan^{*}_{(j,k)}
	+ \entVar^T_0 \Jan_0 - \entVar^T_N \Jan_N 
	\\ 
\text{(SBP property \eqref{eq:SBPprop})} \qquad =& \sum_{j=0}^{N} \entVar^T_j \sum_{k=0}^N (B_{jk} - Q_{kj} + Q_{jk}) \state{f}^{*}_{(j,k)} 
+ \entVar^T_0 \state{f}^a_0 - \entVar^T_N \state{f}^a_N  
\\
	& + \frac{1}{2} \sum_{j=0}^{N} \entVar^T_j \sum_{k=0}^N (B_{jk} - Q_{kj} + Q_{jk}) \Jan^{*}_{(j,k)} + \entVar^T_0 \Jan_0 - \entVar^T_N \Jan_N
	\\ 
\text{(def. of $\mat{B}$ \eqref{eq:Bmatrix} \& consistent flux \eqref{eq:consistentProp})} \qquad 
=& 
	\sum_{j,k=0}^N \entVar^T_j (- Q_{kj} + Q_{jk}) \state{f}^{*}_{(j,k)}  
	\\
	& + \frac{1}{2} \sum_{j,k=0}^N \entVar^T_j (- Q_{kj} + Q_{jk}) \Jan^{*}_{(j,k)}  
	+ \frac{1}{2} \left( \entVar^T_0 \Jan_0 - \entVar^T_N \Jan_N \right)
	\\ 
\text{(symm. flux \eqref{eq:conservativeProp} \& re-index)} \qquad =&  
	\sum_{j,k=0}^N Q_{jk} (\entVar^T_j - \entVar^T_k) \state{f}^{*}_{(j,k)} \\
	& + \frac{1}{2} \sum_{j,k=0}^N Q_{jk} \left( \entVar^T_j \Jan^{*}_{(j,k)} - \entVar^T_k \Jan^{*}_{(k,j)} \right)  
	+ \frac{1}{2} \left( \entVar^T_0 \Jan_0 - \entVar^T_N \Jan_N \right)
	\\
\text{(Def. of non-cons. 2-point \eqref{eq:volNonCons_D}}
 \qquad =&
	\sum_{j,k=0}^N Q_{jk} (\entVar^T_j - \entVar^T_k) \state{f}^{*}_{(j,k)} 
	+ \sum_{j,k=0}^N Q_{jk} \left( \entVar^T_j \numnonconsD{\Jan}_{(j,k)} - \entVar^T_k \numnonconsD{\Jan}_{(k,j)} \right)  \\
\text{\& SBP properties \eqref{eq:SBP1},\eqref{eq:SBP2})}
 \qquad
	& + \frac{1}{2} \sum_{k=0}^N \entVar^T_k \Jan_k
	\underbrace{\sum_{j=0}^N Q_{jk}}_{= \delta_{kN} - \delta_{k0}} - \frac{1}{2} 
	\sum_{j=0}^N \entVar^T_j \Jan_j 
	\underbrace{\sum_{k=0}^N Q_{jk}}_{=0} + \frac{1}{2} \left( \entVar^T_0 \Jan_0 - \entVar^T_N \Jan_N \right)
	\\ 
\text{(definition of $r$ \eqref{eq:EntropyDissipation})} \qquad 
=&  
	\sum_{j,k=0}^N Q_{jk} (\Psi_j - \Psi_k - r_{(j,k)}) 
	\\
\text{(SBP properties \eqref{eq:SBP1},\eqref{eq:SBP2})} \qquad
=&  
	 \sum_{j=0}^N \Psi_j \underbrace{\sum_{k=0}^N Q_{jk}}_{=0}
	- \sum_{k=0}^N \Psi_k \underbrace{\sum_{j=0}^N Q_{jk}}_{=\delta_{kN} - \delta_{k0}}
	- \sum_{j,k=0}^N Q_{jk} r_{(j,k)}   
	\\ 
=&  
	 \Psi_0 - \Psi_N - \sum_{j,k=0}^N Q_{jk} r_{(j,k)} 
\end{align}

Now, gathering the volume and surface terms we obtain
\begin{equation}
\sum_{j=0}^N \omega_j J \dot{S}^a_j =
\entVar^T_0  \numfluxb{f}^a_{(0,L)} 
+ \entVar^T_0  \numnonconsD{\Jan}_{(0,L)} 
- \Psi_0 
- \entVar^T_N  \numfluxb{f}^a_{(N,R)} 
- \entVar^T_N  \numnonconsD{\Jan}_{(N,R)} 
+ \Psi_N 
+ \sum_{j,k=0}^N Q_{jk} r_{(j,k)}.
\end{equation}
We now sum and subtract the following outer terms,
\begin{align} \label{eq:ManupulationBorders}
\sum_{j=0}^N \omega_j J \dot{S}^a_j =&
\entVar^T_0  \numfluxb{f}^a_{(0,L)} 
+ \entVar^T_0  \numnonconsD{\Jan}_{(0,L)} 
- \Psi_0 
- \entVar^T_N  \numfluxb{f}^a_{(N,R)} 
- \entVar^T_N  \numnonconsD{\Jan}_{(N,R)} 
+ \Psi_N 
+ \sum_{j,k=0}^N Q_{jk} r_{(j,k)} 
\nonumber\\
&
+ \frac{1}{2} \left(
\entVar^T_L  \numfluxb{f}^a_{(L,0)} 
+ \entVar^T_L  \numnonconsD{\Jan}_{(L,0)} 
- \Psi_L 
- \entVar^T_R  \numfluxb{f}^a_{(R,N)} 
- \entVar^T_R  \numnonconsD{\Jan}_{(R,N)} 
+ \Psi_R
\right)
\nonumber\\
&
- \frac{1}{2} \left(
\entVar^T_L  \numfluxb{f}^a_{(L,0)} 
+ \entVar^T_L  \numnonconsD{\Jan}_{(L,0)} 
- \Psi_L 
- \entVar^T_R  \numfluxb{f}^a_{(R,N)} 
- \entVar^T_R  \numnonconsD{\Jan}_{(R,N)} 
+ \Psi_R
\right),
\end{align}
and simplify using the definitions of the numerical entropy flux, \eqref{eq:numEntFlux}, and the entropy production, \eqref{eq:EntropyDissipation}, to obtain
\begin{equation} \label{eq:EntropyDGSEMadv}
\dot{S}^a =
\numflux{f}^S_{(0,L)} -\numflux{f}^S_{(N,R)} + \frac{1}{2} \left( r_{(L,0)} + r_{(N,R)} \right) + \sum_{j,k=0}^N Q_{jk} r_{(j,k)},
\end{equation}
which completes the proof.

\section{Three-Dimensional Derivations} \label{sec:3D_Derivations}

\subsection{Three-Dimensional DGSEM for Curvilinear Meshes}

\subsubsection{Discretization}

The DGSEM discretization of the advective and non-conservative terms can be extended from 1D to 3D curvilinear meshes using tensor-product basis expansions.
The extended version of \eqref{eq:DGSEMMHDadv2} reads
\begin{equation}
J_{ijk} \omega_{ijk} \dot{\state{u}}^{\DG}_{ijk} = \state{F}^{a,\DG}_{ijk},
\end{equation}
where
\begin{align}  \label{eq:DGSEM_MHD_3D}
\state{F}^{a,\DG}_{ijk} := 
&
{\color{red}
 \omega_{jk} \left( - 2 \sum_{m=0}^{N} Q_{im} \tilde{\state{f}}^{1*}_{(i,m)jk} 
- \sum_{m=0}^{N} Q_{im} \tilde{\Jan}^{1*}_{(i,m)jk} 
-  
  \delta_{i0} \left[ (\blocktensor{f} + \JanVec) \cdot J\vec{a}^1 \right]_{0jk}
+ \delta_{iN} \left[ (\blocktensor{f} + \JanVec)\cdot J\vec{a}^1 \right]_{Njk} \right) 
}
\nonumber\\ 
+& \omega_{ik} \left( - 2\sum_{m=0}^{N} Q_{jm} \tilde{\state{f}}^{2*}_{i(j,m)k}  
- \sum_{m=0}^{N} Q_{jm} \tilde{\Jan}^{2*}_{i(j,m)k} 
- 
  \delta_{j0} \left[ (\blocktensor{f} + \JanVec) \cdot J\vec{a}^2 \right]_{i0k}
+ \delta_{jN} \left[ (\blocktensor{f} + \JanVec) \cdot J\vec{a}^2 \right]_{iNk}  \right) 
\nonumber\\ 
+&\omega_{ij} \left( - 2\sum_{m=0}^{N} Q_{km} \tilde{\state{f}}^{3*}_{ij(k,m)} 
- \sum_{m=0}^{N} Q_{km} \tilde{\Jan}^{3*}_{ij(k,m)} 
- 
  \delta_{k0} \left[ (\blocktensor{f} + \JanVec) \cdot J\vec{a}^3 \right]_{ij0}
+ \delta_{kN} \left[ (\blocktensor{f} + \JanVec) \cdot J\vec{a}^3 \right]_{ijN} \right)
\nonumber\\
+&
{\color{blue} 
\omega_{jk} \left(
  \delta_{i0} \left[\numfluxb{f}^a_{(0,L)jk} + \numnonconsD{\Jan}_{(0,L)jk} \right]
- \delta_{iN} \left[\numfluxb{f}^a_{(N,R)jk} + \numnonconsD{\Jan}_{(N,R)jk} \right] \right)
}
\nonumber\\
+& \omega_{ik} \left(
  \delta_{j0} \left[\numfluxb{f}^a_{i(0,L)k} + \numnonconsD{\Jan}_{i(0,L)k} \right]
- \delta_{jN} \left[\numfluxb{f}^a_{i(N,R)k} + \numnonconsD{\Jan}_{i(N,R)k} \right] \right)
\nonumber\\
+& \omega_{ij} \left(
  \delta_{k0} \left[\numfluxb{f}^a_{ij(0,L)} + \numnonconsD{\Jan}_{ij(0,L)} \right]
- \delta_{kN} \left[\numfluxb{f}^a_{ij(N,R)} + \numnonconsD{\Jan}_{ij(N,R)} \right] \right).
\end{align}

The mapping Jacobians, $J_{ijk}$, which may now be different at each degree of freedom of the element, and the contravariant basis vectors, $\vec{a}^m_{ijk}=\Nabla \xi^m$, define the mapping from reference space to physical space, $(\xi^1,\xi^2,\xi^3) \in [-1,1]^3 \rightarrow (x,y,z) \in \Omega$. 

Moreover, the following new conventions are used:

\begin{itemize}
\item The two- and three-dimensional quadrature weights are defined from the one-dimensional weights as
\begin{equation}
\omega_{ij} := \omega_i \omega_j, ~~~~~ \omega_{ijk} := \omega_i \omega_j \omega_k.
\end{equation}


\item The term $\JanVec$ is defined as
\begin{equation}
\JanVec := \phiMHD \vec{B} + \phiGLM \psi.
\end{equation}

\item The volume numerical two-point fluxes are defined with the metric terms as
\begin{align}
\tilde{\state{f}}^{1*}_{(i,m)jk} &:= \blocktensor{f}^{*}(\state{u}_{ijk}, \state{u}_{mjk}) \cdot \avg{J\vec{a}^1}_{(i,m)jk}, ~~~~~
\tilde{\state{f}}^{2*}_{i(j,m)k} := \blocktensor{f}^{*}(\state{u}_{ijk}, \state{u}_{imk}) \cdot \avg{J\vec{a}^2}_{i(j,m)k}, \nonumber \\
\tilde{\state{f}}^{3*}_{ij(k,m)} &:= \blocktensor{f}^{*}(\state{u}_{ijk}, \state{u}_{ijm}) \cdot \avg{J\vec{a}^3}_{ij(k,m)},
\end{align}
as it is conventionally done in the DGSEM literature \cite{Gassner2018,Manzanero2020}.
They fulfill the symmetry and consistency property.

\item The volume numerical two-point non-conservative terms are defined with the metric terms in accordance with the definitions of \citet{Bohm2018},
\begin{align}
\tilde{\Jan}^{1*}_{(i,m)jk} &:= \phiMHD_{ijk} \vec{B}_{mjk} \cdot \avg{J\vec{a}^1}_{(i,m)jk} + \phiGLM_{ijk} \cdot \left( J\vec{a} \right)_{ijk} \psi_{mjk}, \nonumber\\
\tilde{\Jan}^{2*}_{i(j,m)k} &:= \phiMHD_{ijk} \vec{B}_{imk} \cdot \avg{J\vec{a}^2}_{i(j,m)k} + \phiGLM_{ijk} \cdot \left( J\vec{a} \right)_{ijk} \psi_{imk}, \nonumber\\
\tilde{\Jan}^{3*}_{ij(k,m)} &:= \phiMHD_{ijk} \vec{B}_{ijm} \cdot \avg{J\vec{a}^3}_{ij(k,m)} + \phiGLM_{ijk} \cdot \left( J\vec{a} \right)_{ijk} \psi_{ijm}.
\end{align}
They fulfill the consistency property.

\item Remark that it is possible to express the unit outward-pointing normal vectors in terms of the metric terms on the element boundaries,
\begin{align} \label{eq:NormalVectors}
\vec{n}_{0jk} &= -\frac{(J\vec{a}^1)_{0jk}} {\norm{J\vec{a}^1}_{0jk}}, &
\vec{n}_{i0k} &= -\frac{(J\vec{a}^2)_{i0k}} {\norm{J\vec{a}^2}_{i0k}}, &
\vec{n}_{ij0} &= -\frac{(J\vec{a}^3)_{ij0}} {\norm{J\vec{a}^3}_{ij0}}, \nonumber\\
\vec{n}_{Njk} &= \frac{(J\vec{a}^1)_{Njk}} {\norm{J\vec{a}^1}_{Njk}}, &
\vec{n}_{iNk} &= \frac{(J\vec{a}^2)_{iNk}} {\norm{J\vec{a}^2}_{iNk}}, &
\vec{n}_{ijN} &= \frac{(J\vec{a}^3)_{ijN}} {\norm{J\vec{a}^3}_{ijN}}, 
\end{align}
to obtain a more conventional form of the DGSEM discretization. 
We keep the form in \eqref{eq:DGSEM_MHD_3D}, as it will be useful for future derivations.

\item For convenience, we define the "surface" numerical flux in a general fashion as
\begin{align}
{\numfluxb{f}^a}_{(i,m)jk} &:= \hat{\blocktensor{f}}^a(\state{u}_{ijk}, \state{u}_{mjk}) \cdot \avg{J\vec{a}^1}_{(i,m)jk}, ~~~~~
{\numfluxb{f}^a}_{i(j,m)k} := \hat{\blocktensor{f}}^a(\state{u}_{ijk}, \state{u}_{imk}) \cdot \avg{J\vec{a}^2}_{i(j,m)k}, \nonumber \\
{\numfluxb{f}^a}_{ij(k,m)} &:= \hat{\blocktensor{f}}^a(\state{u}_{ijk}, \state{u}_{ijm}) \cdot \avg{J\vec{a}^3}_{ij(k,m)},
\end{align}
where $\hat{\blocktensor{f}}(\cdot ,\cdot )$ is an approximate solution to the Riemann problem between two states.
Note that on an element boundary, our interpretation of the numerical flux reduces to the conventional solution of the 1D Riemann problem in normal direction multiplied by the discrete surface Jacobian, e.g.
\begin{align*}
{\numfluxb{f}^a}_{(N,R)jk} &= \hat{\blocktensor{f}}^a(\state{u}_{Njk}, \state{u}_{Rjk}) \cdot \avg{J\vec{a}^1}_{(N,R)jk} \\
 &= \hat{\blocktensor{f}}^a(\state{u}_{Njk}, \state{u}_{Rjk}) \cdot \vec{n}_{Njk} \norm{J\vec{a}^1}_{Njk} \\
 &= 
 \numfluxb{f}^a(\state{u}_{Njk}, \state{u}_{Rjk}; \vec{n}_{Njk} )
 \underbrace{\norm{J\vec{a}^1}_{Njk}}_{:= \mathrm{Surface~Jacobian}} ,
\end{align*}
as the metric terms are continuous across element boundaries.
The "surface" numerical fluxes fulfill the symmetry and consistency properties.

\item For convenience, we define the "surface" numerical non-conservative terms in an general fashion,
\begin{equation} \label{eq:numNonCons3D}
\numnonconsD{\Jan}_{(i,m)jk} := 
\frac{1}{2} \left[ (\vec{B}\cdot J\vec{a}^1)_{ijk} + \vec{B}_{mjk} \cdot \avg{J\vec{a}^1}_{(i,m)jk} \right] \phiMHD_{ijk}  +
\phiGLM_{ijk} \cdot (J\vec{a}^1)_{ijk} \avg{\psi}_{(i,m)jk},
\end{equation}
and the other directions in an analogous way.
Note that on an element boundary, our interpretation of the numerical non-conservative term reduces to the definition by \citet{Bohm2018} multiplied by the discrete surface Jacobian,
\begin{equation}
\numnonconsD{\Jan}_{(N,R)jk} = 
\underbrace{\norm{J\vec{a}^1}_{Njk}}_{:= \mathrm{Surface~Jacobian}}
\vec{n}_{Njk} 
\cdot 
\left[
\avg{\vec{B}}_{(N,R)jk}  \phiMHD_{ijk}  +
\phiGLM_{ijk} \avg{\psi}_{(i,m)jk}
\right].
\end{equation}
Clearly, the "surface" numerical non-conservative terms fulfill the consistency property, but not the symmetry property.

Our definition of the "surface" numerical non-conservative terms, \eqref{eq:numNonCons3D}, allows us to mimic the identity \eqref{eq:volNonCons_D} in 3D,
\begin{equation} \label{eq:numNonConsIdent3D}
\tilde{\Jan}^{1*}_{(i,m)jk} =
2\numnonconsD{\Jan}_{(i,m)jk}
- \JanVec_{ijk} \cdot (J\vec{a}^1)_{ijk}.
\end{equation}

\end{itemize}

\subsubsection{Entropy Balance}

We start by defining the numerical entropy flux and the entropy production that are compatible with our high-order DGSEM on 3D curvilinear meshes.

\begin{definition}[Numerical entropy flux for the 3D DGSEM]
The numerical entropy flux from the degree of freedom $ijk$ to $mjk$ is defined as
\begin{equation} \label{eq:numEntFlux_3D} 
\numflux{f}^S_{(i,m)jk} = 
\avg{\entVar}_{(i,m)jk}^T \numfluxb{f}^a_{(i,m)jk} 
+ \frac{1}{2} \entVar^T_{ijk} \numnonconsD{\stateG{\Phi}}_{(i,m)jk}
+ \frac{1}{2} \entVar^T_{mjk} \numnonconsD{\stateG{\Phi}}_{(m,i)jk}
- \avg{J\vec{a}^1}_{(i,m)jk} \cdot \avg{\vec{\Psi}}_{(i,m)jk},
\end{equation}
which fulfills the symmetric conservative property \eqref{eq:conservativeProp}, and can be used with both the volume numerical flux and the surface numerical flux.
\end{definition}

\begin{definition}[Entropy production for the 3D DGSEM]
The entropy production on an interface between the degrees of freedom $j$ and $k$ is defined as
\begin{equation} \label{eq:EntropyDissipation_3D}
r_{(i,m)jk} =
\jump{\entVar}_{(i,m)jk}^T 
\numfluxb{f}^a_{(i,m)jk}
+ \entVar^T_{mjk} \numnonconsD{\stateG{\Phi}}_{(m,i)jk}
- \entVar^T_{ijk} \numnonconsD{\stateG{\Phi}}_{(i,m)jk}
- \avg{J\vec{a}^1}_{(i,m)jk} \cdot \jump{\vec{\Psi}}_{(i,m)jk},
\end{equation}
which can be used with both the volume numerical flux and the surface numerical flux.
\end{definition}

\begin{lemma} \label{lemma:EntropyDGSEM_3D}
The semi-discrete entropy balance of the 3D DGSEM discretization on curvilinear meshes of the GLM-MHD equations, \eqref{eq:DGSEM_MHD_3D}, integrating over an entire element, reads
\begin{align*}
\sum_{i=0}^N \entVar_{ijk}^T \state{F}_{ijk}^{a,\DG}
=&
	+ \sum_{j,k=0}^N \omega_{jk} \left(
	\numflux{f}^S_{(0,L)jk} -\numflux{f}^S_{(N,R)jk} + \frac{1}{2} \left( r_{(L,0)jk} + r_{(N,R)jk} \right) 
	\right)
	\\&
	+ \sum_{i,k=0}^N \omega_{ik} \left(
	\numflux{f}^S_{i(0,L)k} -\numflux{f}^S_{i(N,R)k} + \frac{1}{2} \left( r_{i(L,0)k} + r_{i(N,R)k} \right) 
	\right)
	\\&
	+ \sum_{i,j=0}^N \omega_{ij} \left(
	\numflux{f}^S_{ij(0,L)} -\numflux{f}^S_{ij(N,R)} + \frac{1}{2} \left( r_{ij(L,0)} + r_{ij(N,R)} \right) 
	\right)
	\\
	&
	+ \sum_{i,j,k,m=0}^N \omega_{ijk}
	\left(
	D_{im}  r_{(i,m)jk} + D_{jm} r_{i(j,m)k} + D_{km} r_{ij(k,m)}
	\right),
\end{align*}
where the numerical entropy flux and the entropy production are consistent with the FV definitions,  \eqref{eq:numEntFlux_3D} and \eqref{eq:EntropyDissipation_3D}, respectively.
\end{lemma}

\begin{proof}

Let us first compute the integral of the $\xi^1$ volume terms (marked in red in \eqref{eq:DGSEM_MHD_3D}) along the $\xi^1$ direction (note that we are scaling with $-1/\omega_{jk}$ for convenience):
\begin{align*}
- \frac{1}{\omega_{jk}}\sum_{i=0}^N \entVar_{ijk}^T {\color{red} \state{F}_{\elemDom,ijk}^{a,\DG,\xi^1}}
	=& \sum_{i=0}^{N} \entVar^T_{ijk} \sum_{m=0}^N 2 Q_{im} \contSt{f}^{*}_{(i,m)jk} 
	+ \left( \entVar^T ( \blocktensor{f}^a + \JanVec) \cdot J\vec{a}^1 \right)_{0jk}  \\
	& + \sum_{i=0}^{N} \entVar^T_{ijk} \sum_{m=0}^N Q_{im} \tilde{\Jan}^{*}_{(i,m)jk}
	- \left( \entVar^T ( \blocktensor{f}^a + \JanVec ) \cdot J\vec{a}^1 \right)_{Njk} 
	\\ 
\text{(SBP property \eqref{eq:SBPprop})} \qquad =& \sum_{i=0}^{N} \entVar^T_{ijk} \sum_{m=0}^N (B_{im} - Q_{mi} + Q_{im}) \contSt{f}^{*}_{(i,m)jk} 
	+ \left( \entVar^T ( \blocktensor{f}^a + \JanVec ) \cdot J\vec{a}^1 \right)_{0jk}
\\
	& + \frac{1}{2} \sum_{i=0}^{N} \entVar^T_{ijk} \sum_{m=0}^N (B_{im} - Q_{mi} + Q_{im}) \tilde{\Jan}^{*}_{(i,m)jk} 
	\\
	&- \left( \entVar^T ( \blocktensor{f}^a + \JanVec ) \cdot J\vec{a}^1 \right)_{Njk} 
	\\ 
\text{(def. of $\mat{B}$ \eqref{eq:Bmatrix} \& consistent flux \eqref{eq:consistentProp})} \qquad 
=& 
	\sum_{i,m=0}^N \entVar^T_{ijk} (- Q_{mi} + Q_{im}) \contSt{f}^{*}_{(i,m)jk}  
	+ \frac{1}{2} \left( \entVar^T \JanVec \cdot J\vec{a}^1 \right)_{0jk}   
	\\
	& + \frac{1}{2} \sum_{i,m=0}^N \entVar^T_{ijk} (- Q_{mi} + Q_{im}) \tilde{\Jan}^{*}_{(i,m)jk}  
	- \frac{1}{2} \left( \entVar^T \JanVec \cdot J\vec{a}^1 \right)_{Njk} 
	\\ 
\text{(symm. flux \eqref{eq:conservativeProp} \& re-index)} \qquad =&  
	\sum_{i,m=0}^N Q_{im} (\entVar^T_{ijk} - \entVar^T_{mjk}) \state{f}^{*}_{(i,m)jk} 
	+ \frac{1}{2} \left( \entVar^T \JanVec \cdot J\vec{a}^1 \right)_{0jk} 
	\\
	& + \frac{1}{2} \sum_{i,m=0}^N Q_{im} (\entVar^T_{ijk} \tilde{\Jan}^{*}_{(i,m)jk} - \entVar^T_{mjk} \tilde{\Jan}^{*}_{(m,i)jk})   
	- \frac{1}{2} \left( \entVar^T \JanVec \cdot J\vec{a}^1 \right)_{Njk} 
	\\
\text{(Def. of non-cons. 2-point \eqref{eq:numNonConsIdent3D}}
 \qquad =&
	\sum_{i,m=0}^N Q_{im} (\entVar^T_{ijk} - \entVar^T_{mjk}) \state{f}^{*}_{(i,m)jk} 
	+ \sum_{i,m=0}^N Q_{im} \left( \entVar^T_{ijk} \numnonconsD{\Jan}_{(i,m)jk} - \entVar^T_{mjk} \numnonconsD{\Jan}_{(m,i)jk} \right)  \\
\text{\& SBP properties \eqref{eq:SBP1},\eqref{eq:SBP2})}
 \qquad
	& + \frac{1}{2} \sum_{m=0}^N \left(\entVar^T \JanVec \cdot J\vec{a}^1 \right)_{mjk}
	\underbrace{\sum_{i=0}^N Q_{im}}_{= \delta_{mN} - \delta_{m0}} - \frac{1}{2} 
	\sum_{i=0}^N \left( \entVar^T \JanVec \cdot J\vec{a}^1 \right)_{ijk} 
	\underbrace{\sum_{m=0}^N Q_{im}}_{=0} 
	\\
	&+ \frac{1}{2} \left( \entVar^T \JanVec \cdot J\vec{a}^1 \right)_{0jk}
	- \frac{1}{2} \left( \entVar^T \JanVec \cdot J\vec{a}^1 \right)_{Njk} 
	\\ 
\text{(definition of $r$ \eqref{eq:EntropyDissipation_3D})} \qquad 
=&  
	\sum_{i,m=0}^N Q_{im} \left( \avg{J\vec{a}^1}_{(i,m)jk} \cdot (\vec{\Psi}_{ijk} - \vec{\Psi}_{mjk}) - r_{(i,m)jk} \right) 
	\\
\text{(Split the sum \& SBP properties} \qquad
=&  
	 \frac{1}{2} \sum_{i=0}^N (J\vec{a}^1)_{ijk} \cdot \vec{\Psi}_{ijk} \underbrace{\sum_{m=0}^N Q_{im}}_{=0}
	 +\frac{1}{2} \sum_{i,m=0}^N Q_{im} (J\vec{a}^1)_{mjk} \cdot \vec{\Psi}_{ijk} 
	 \\
\text{ \eqref{eq:SBP1},\eqref{eq:SBP2})} \qquad	
&- \frac{1}{2} \sum_{m=0}^N (J\vec{a}^1)_{mjk} \cdot \vec{\Psi}_{mjk} \underbrace{\sum_{i=0}^N Q_{im}}_{=\delta_{mN} - \delta_{m0}}
	-\frac{1}{2} \sum_{i,m=0}^N Q_{im} (J\vec{a}^1)_{ijk} \cdot \vec{\Psi}_{mjk} \\ 
	&- \sum_{i,m=0}^N Q_{im} r_{(i,m)jk}   
	\\ 
\text{(Simplify and re-index)} \qquad
=&
	 \frac{1}{2} \sum_{i,m=0}^N (Q_{im} - Q_{mi}) (J\vec{a}^1)_{mjk} \cdot \vec{\Psi}_{ijk} 
	 \\
	 &+\frac{1}{2} (J\vec{a}^1)_{0jk} \cdot \vec{\Psi}_{0jk}
	 -\frac{1}{2} (J\vec{a}^1)_{Njk} \cdot \vec{\Psi}_{Njk}
	 - \sum_{i,m=0}^N Q_{im} r_{(i,m)jk}   
\\
\text{(SBP property \eqref{eq:SBPprop})} \qquad
=&
	 {-\frac{1}{2} \sum_{i,m=0}^N B_{im} (J\vec{a}^1)_{mjk} \cdot \vec{\Psi}_{ijk}}
	 + \sum_{i,m=0}^N Q_{im} (J\vec{a}^1)_{mjk} \cdot \vec{\Psi}_{ijk} 
	 \\
	 &+{\frac{1}{2} (J\vec{a}^1)_{0jk} \cdot \vec{\Psi}_{0jk}}
	 -{\frac{1}{2} (J\vec{a}^1)_{Njk} \cdot \vec{\Psi}_{Njk}}
	 - \sum_{i,m=0}^N Q_{im} r_{(i,m)jk}   
\\
\text{(Simplify)} \qquad
=&  
	 \sum_{i,m=0}^N Q_{im} (J\vec{a}^1)_{mjk} \cdot \vec{\Psi}_{ijk}  - \sum_{i,m=0}^N Q_{im} r_{(i,m)jk}   
	 \\
	 &+{ (J\vec{a}^1)_{0jk} \cdot \vec{\Psi}_{0jk}}
	 - { (J\vec{a}^1)_{Njk} \cdot \vec{\Psi}_{Njk}} \numberthis \label{eq:EntropyDGSEM_Volume3D}
\end{align*}

We now compute the integral of the $\xi^1$ surface terms (marked in blue in \eqref{eq:DGSEM_MHD_3D}) along the $\xi^1$ direction,
\begin{align*}
\sum_{i=0}^N \entVar_{ijk}^T 
{\color{blue} \state{F}_{\partial \elemDom,ijk}^{a,\DG,\xi^1}}
	=& \omega_{jk} \left(
  \entVar_{0jk}^T \left[\numfluxb{f}_{(0,L)jk} + \numnonconsD{\Jan}_{(0,L)jk} \right]
- \entVar_{Njk}^T \left[\numfluxb{f}_{(N,R)jk} + \numnonconsD{\Jan}_{(N,R)jk} \right] \right)
\\
\text{(Sum zero)}
	=& \omega_{jk} \left(
	\entVar_{0jk}^T \left[\numfluxb{f}_{(0,L)jk} + 
	\numnonconsD{\Jan}_{(0,L)jk} \right]
	- \entVar_{Njk}^T \left[\numfluxb{f}_{(N,R)jk} +
	\numnonconsD{\Jan}_{(N,R)jk} \right] \right)	
	\\
	&
	+ \frac{\omega_{jk}}{2} 
	\left(
	\entVar^T_{Ljk}  \numfluxb{f}^a_{(L,0)jk} 
	+ \entVar^T_{Ljk}  \numnonconsD{\Jan}_{(L,0)jk} 
	- (J\vec{a}^1)_{0jk} \cdot \vec{\Psi}_{Ljk}
	\right)
	\\
	&
	- \frac{\omega_{jk}}{2} \left(
	\entVar^T_{Ljk}  \numfluxb{f}^a_{(L,0)jk} 
	+ \entVar^T_{Ljk}  \numnonconsD{\Jan}_{(L,0)jk} 
	- (J\vec{a}^1)_{0jk} \cdot \vec{\Psi}_{Ljk} 
	 \right)
	 \\
	&
	+ \frac{\omega_{jk}}{2} 
	\left(
	- \entVar^T_{Rjk}  \numfluxb{f}^a_{(R,N)jk} 
	- \entVar^T_{Rjk}  \numnonconsD{\Jan}_{(R,N)jk} 
	+ (J\vec{a}^1)_{Njk} \cdot \vec{\Psi}_{Rjk}
	\right)
	\\
	 &
	- \frac{\omega_{jk}}{2} 
	\left(
	- \entVar^T_{Rjk}  \numfluxb{f}^a_{(R,N)jk} 
	- \entVar^T_{Rjk}  \numnonconsD{\Jan}_{(R,N)jk} 
	+ (J\vec{a}^1)_{Njk} \cdot \vec{\Psi}_{Rjk}
	\right)
\\
\text{(Use \eqref{eq:numEntFlux_3D},\eqref{eq:EntropyDissipation_3D})}
	=& 
	\omega_{jk} \left(
	\numflux{f}^S_{(0,L)jk} -\numflux{f}^S_{(N,R)jk} + \frac{1}{2} \left( r_{(L,0)jk} + r_{(N,R)jk} \right) 
	+{ (J\vec{a}^1)_{0jk} \cdot \vec{\Psi}_{0jk}}
	-{ (J\vec{a}^1)_{Njk} \cdot \vec{\Psi}_{Njk}} \right),
\numberthis \label{eq:EntropyDGSEM_Surf3D}
\end{align*}
and gather the integral of the volume \eqref{eq:EntropyDGSEM_Volume3D} and surface \eqref{eq:EntropyDGSEM_Surf3D} terms to obtain
\begin{align*}
\sum_{i=0}^N \entVar_{ijk}^T 
\left(
{\color{red} \state{F}_{\elemDom,ijk}^{a,\DG,\xi^1}}+
{\color{blue} \state{F}_{\partial \elemDom,ijk}^{a,\DG,\xi^1}}
\right)
	=&
\omega_{jk} \left(
	\numflux{f}^S_{(0,L)jk} -\numflux{f}^S_{(N,R)jk} + \frac{1}{2} \left( r_{(L,0)jk} + r_{(N,R)jk} \right) 
	+ \sum_{i,m=0}^N Q_{im} r_{(i,m)jk}
	\right)
	\\
	&
	- \omega_{jk} \sum_{i,m=0}^N Q_{im} (J\vec{a}^1)_{mjk} \cdot \vec{\Psi}_{ijk}.
\numberthis \label{eq:EntropyDGSEM_Xi_3D}
\end{align*}

Summing \eqref{eq:EntropyDGSEM_Xi_3D} over the remaining directions ($\xi^2$ and $\xi^3$) and adding the integral in the three directions of the other terms (marked in black in \eqref{eq:DGSEM_MHD_3D}) leads to
\begin{align*}
\sum_{i=0}^N \entVar_{ijk}^T \state{F}_{ijk}^{a,\DG}
=&
	\ \ \  \sum_{j,k=0}^N \omega_{jk} \left(
	\numflux{f}^S_{(0,L)jk} -\numflux{f}^S_{(N,R)jk} + \frac{1}{2} \left( r_{(L,0)jk} + r_{(N,R)jk} \right) 
	\right)
	\\&
	+ \sum_{i,k=0}^N \omega_{ik} \left(
	\numflux{f}^S_{i(0,L)k} -\numflux{f}^S_{i(N,R)k} + \frac{1}{2} \left( r_{i(L,0)k} + r_{i(N,R)k} \right) 
	\right)
	\\&
	+ \sum_{i,j=0}^N \omega_{ij} \left(
	\numflux{f}^S_{ij(0,L)} -\numflux{f}^S_{ij(N,R)} + \frac{1}{2} \left( r_{ij(L,0)} + r_{ij(N,R)} \right) 
	\right)
	\\
	&
	+ \sum_{i,j,k,m=0}^N \omega_{ijk}
	\left(
	D_{im}  r_{(i,m)jk} + D_{jm} r_{i(j,m)k} + D_{km} r_{ij(k,m)}
	\right)
	\\
	&
	- \sum_{i,j,k=0}^N \omega_{ijk} 
	\underbrace{	
	\vec{\Psi}_{ijk} \cdot
	\sum_{m=0}^N
	\left(
	D_{im} (J\vec{a}^1)_{mjk} + D_{jm} (J\vec{a}^2)_{imk} + D_{km} (J\vec{a}^3)_{ijm}
	\right)
	}_{=0 },
\end{align*}
where the last term is equal to zero if the discrete metric identities,
\begin{equation} \label{eq:DiscMetricIdentities}
\left(
\sum_{l=1}^3 \bigpartialderiv{}{\xi^l} (J{a}^l_d)
\right)_{ijk} = 0, ~~~~~~ d \in \{ 1,2,3 \},
\end{equation}
hold for all the nodes of the element, $i,j,k \in \{ 0, \ldots, N \}$.
For example, they can be computed from a discrete curl \cite{Kopriva2006}.

\end{proof}

\subsection{Subcell Native LGL Finite Volumes for Curvilinear Meshes}

\subsubsection{Discretization}

The native LGL FV discretization of the advective and non-conservative terms on 3D curvilinear meshes reads
\begin{align} \label{eq:FV_MHD_3D}
J_{ijk} \omega_{ijk} \dot{\state{u}}^{\FV}_{ijk}
=& \state{F}_{ijk}^{a,\FV} \\
=& \omega_{jk} \left( \numfluxb{f}^a_{(i-1,i)jk} - \numfluxb{f}^a_{(i,i+1)jk} \right)
+ \omega_{ik} \left( \numfluxb{f}^a_{i(j-1,j)k} - \numfluxb{f}^a_{i(j,j+1)k} \right)
+ \omega_{ij} \left( \numfluxb{f}^a_{ij(k-1,k)} - \numfluxb{f}^a_{ij(k,k+1)} \right) 
\nonumber \\
&+ 
\omega_{jk} \left(
  \numnonconsD{\Jan}_{(i,i-1)jk}
- \numnonconsD{\Jan}_{(i,i+1)jk}
\right)
+ \omega_{ik} \left(
  \numnonconsD{\Jan}_{i(j,j-1)k}
- \numnonconsD{\Jan}_{i(j,j+1)k}
\right)
+ \omega_{ij} \left(
  \numnonconsD{\Jan}_{ij(k,k-1)}
- \numnonconsD{\Jan}_{ij(k,k+1)}
\right),
\end{align}
where we define the numerical fluxes and non-conservative terms as
\begin{align}
\numfluxb{f}^a_{(i,m)jk} :=&
\norm{\vn_{(i,m)jk}} \numfluxb{f}^a \left(\state{u}_{ijk}, \state{u}_{mjk}; \frac{\vn_{(i,m)jk}}{\norm{\vn_{(i,m)jk}}} \right) 
\\
\numnonconsD{\Jan}_{(i,m)jk} :=& 
\vn_{(i,m)jk} 
\cdot 
\left[
\avg{\vec{B}}_{(i,m)jk}  \phiMHD_{ijk}  +
\phiGLM_{ijk} \avg{\psi}_{(i,m)jk}
\right],
\end{align}
and we adopt the subcell metrics derived by \citet{Hennemann2020}, which ensure a water-tight subcell FV discretization,
\begin{align} \label{eq:SubcellMetrics}
\vn_{(i,i+1)jk}  &= J\vec{a}^1_{0jk} + \sum_{l=0}^{i} \sum_{m=0}^{N} Q_{lm} (J\vec{a}^1)_{mjk}, &
\vn_{i(j,j+1)k}  &= J\vec{a}^2_{i0k} + \sum_{l=0}^{j} \sum_{m=0}^{N} Q_{lm} (J\vec{a}^2)_{imk},
\nonumber\\
\vn_{ij(k,k+1)}  &= J\vec{a}^3_{ij0} + \sum_{l=0}^{k} \sum_{m=0}^{N} Q_{lm} (J\vec{a}^3)_{ijm}.
\end{align}

We remark that the subcell metrics fulfill the symmetry property, and that they reduce to the normal vectors (scaled by the surface area) on the element boundaries, where they do not point outward, but in the direction of the reference element coordinates.

\subsubsection{Entropy Balance}

We need to define a new numerical entropy flux and a new entropy production that are compatible with our native LGL FV method on 3D curvilinear meshes.

\begin{definition}[Numerical entropy flux for the 3D LGL FV]
The numerical entropy flux from the degree of freedom $ijk$ to $mjk$ is defined as
\begin{equation} \label{eq:numEntFlux_FV_3D} 
\numflux{f}^S_{(i,m)jk} = 
\avg{\entVar}_{(i,m)jk}^T \numfluxb{f}^a_{(i,m)jk} 
+ \frac{1}{2} \entVar^T_{ijk} \numnonconsD{\stateG{\Phi}}_{(i,m)jk}
+ \frac{1}{2} \entVar^T_{mjk} \numnonconsD{\stateG{\Phi}}_{(m,i)jk}
- \vn_{(i,m)jk} \cdot \avg{\vec{\Psi}}_{(i,m)jk},
\end{equation}
which fulfills the symmetric conservative property, \eqref{eq:conservativeProp}.
\end{definition}

\begin{definition}[Entropy production for the 3D LGL FV]
The entropy production on an interface between the degrees of freedom $j$ and $k$ is defined as
\begin{equation} \label{eq:EntropyDissipation_FV_3D}
r_{(i,m)jk} =
\jump{\entVar}_{(i,m)jk}^T 
\numfluxb{f}^a_{(i,m)jk}
+ \entVar^T_{mjk} \numnonconsD{\stateG{\Phi}}_{(m,i)jk}
- \entVar^T_{ijk} \numnonconsD{\stateG{\Phi}}_{(i,m)jk}
- \vn_{(i,m)jk} \cdot \jump{\vec{\Psi}}_{(i,m)jk}.
\end{equation}

\end{definition}

Note that, on an element boundary, the numerical entropy flux, \eqref{eq:numEntFlux_FV_3D}, and the entropy production, \eqref{eq:EntropyDissipation_FV_3D}, of the 3D native LGL FV method are equal to the 3D DGSEM definitions, \eqref{eq:numEntFlux_3D} and \eqref{eq:EntropyDissipation_3D}, respectively.

\begin{lemma} \label{lemma:EntropyFV_3D}
The semi-discrete entropy balance of the 3D LGL FV discretization on curvilinear meshes of the GLM-MHD equations, \eqref{eq:FV_MHD_3D}, for a subcell reads
\begin{align*}
\entVar_{ijk}^T \state{F}_{ijk}^{a,\FV}
=&
	\ \ \  \omega_{jk} \left(
	\numflux{f}^S_{(i-1,i)jk} -\numflux{f}^S_{(i,i+1)jk} + \frac{1}{2} \left( r_{(i-1,i)jk} + r_{(i,i+1)jk} \right) 
	\right)
	\\&
	+ \omega_{ik} \left(
	\numflux{f}^S_{i(j-1,j)k} -\numflux{f}^S_{i(j,j+1)k} + \frac{1}{2} \left( r_{i(j-1,j)k} + r_{i(j,j+1)k} \right) 
	\right)
	\\&
	+ \omega_{ij} \left(
	\numflux{f}^S_{ij(k-1,k)} -\numflux{f}^S_{ij(k,k+1)} + \frac{1}{2} \left( r_{ij((k-1,k)} + r_{ij(k,k+1)} \right) 
	\right),
\end{align*}
where the numerical entropy flux and the entropy production are consistent with the FV definitions,  \eqref{eq:numEntFlux_FV_3D} and \eqref{eq:EntropyDissipation_FV_3D}, respectively.
\end{lemma}

\begin{proof}
After contracting \eqref{eq:FV_MHD_3D} with the entropy variables, we follow the same strategy as in previous proofs, where we sum and subtract terms on each subcell interface.
In this case we have interfaces for the subcells at $i \pm 1$, $j \pm 1$, and $k \pm 1$, so we have
\begin{align*}
\entVar_{ijk}^T \state{F}_{ijk}^{a,\FV}
=& \entVar_{ijk}^T \left[
\omega_{jk} \left( \numfluxb{f}^a_{(i-1,i)jk} - \numfluxb{f}^a_{(i,i+1)jk} \right)
+ \omega_{ik} \left( \numfluxb{f}^a_{i(j-1,j)k} - \numfluxb{f}^a_{i(j,j+1)k} \right)
+ \omega_{ij} \left( \numfluxb{f}^a_{ij(k-1,k)} - \numfluxb{f}^a_{ij(k,k+1)} \right) 
\right.
 \\
&+ 
\left.
\omega_{jk} \left(
  \numnonconsD{\Jan}_{(i,i-1)jk}
- \numnonconsD{\Jan}_{(i,i+1)jk}
\right)
+ \omega_{ik} \left(
  \numnonconsD{\Jan}_{i(j,j-1)k}
- \numnonconsD{\Jan}_{i(j,j+1)k}
\right)
+ \omega_{ij} \left(
  \numnonconsD{\Jan}_{ij(k,k-1)}
- \numnonconsD{\Jan}_{ij(k,k+1)}
\right)
\right]	
	\\
	&
	+ \frac{\omega_{jk}}{2} 
	\left(
	\entVar^T_{(i-1)jk}  \numfluxb{f}^a_{(i-1,i)jk} 
	+ \entVar^T_{(i-1)jk}  \numnonconsD{\Jan}_{(i-1,i)jk} 
	- \vn_{(i-1,i)jk} \cdot \vec{\Psi}_{(i-1)jk}
	+ 2 \vn_{(i-1,i)jk} \cdot \vec{\Psi}_{ijk}
	\right)
	\\
	&
	- \frac{\omega_{jk}}{2} \left(
	\entVar^T_{(i-1)jk}  \numfluxb{f}^a_{(i-1,i)jk} 
	+ \entVar^T_{(i-1)jk}  \numnonconsD{\Jan}_{(i-1,i)jk} 
	- \vn_{(i-1,i)jk} \cdot \vec{\Psi}_{(i-1)jk}
	+ 2\vn_{(i-1,i)jk} \cdot \vec{\Psi}_{ijk} 
	 \right)
	 \\
	&
	\cdots
	 \\
	&
	+ \frac{\omega_{ij}}{2} 
	\left(
	- \entVar^T_{ij(k+1)}  \numfluxb{f}^a_{(k+1,k)ij} 
	- \entVar^T_{ij(k+1)}  \numnonconsD{\Jan}_{(k+1,k)ij} 
	+ \vn_{ij(k,k+1)} \cdot \vec{\Psi}_{ij(k+1)}
	- 2\vn_{ij(k,k+1)} \cdot \vec{\Psi}_{ijk}
	\right)
	\\
	 &
	- \frac{\omega_{ij}}{2} 
	\left(
	- \entVar^T_{ij(k+1)}  \numfluxb{f}^a_{ij(k+1,k)} 
	- \entVar^T_{ij(k+1)}  \numnonconsD{\Jan}_{ij(k+1,k)} 
	+ \vn_{ij(k,k+1)} \cdot \vec{\Psi}_{ij(k+1)}
	- 2\vn_{ij(k,k+1)} \cdot \vec{\Psi}_{ijk}
	\right),
	\numberthis \label{eq:StrategySumSubtract}
\end{align*}
where, for readability, we used the ellipsis ($\ldots$) for the terms that belong to the interfaces $i + 1$, $j \pm 1$, and $k - 1$.
We now simplify using the definitions of the numerical entropy flux \eqref{eq:numEntFlux_FV_3D}, and the entropy production \eqref{eq:EntropyDissipation_FV_3D} to obtain
\begin{align*}
\entVar_{ijk}^T \state{F}_{ijk}^{a,\FV}
=&
	\ \ \  \omega_{jk} \left(
	\numflux{f}^S_{(i-1,i)jk} -\numflux{f}^S_{(i,i+1)jk} + \frac{1}{2} \left( r_{(i-1,i)jk} + r_{(i,i+1)jk} \right) 
	\right)
	\\&
	+ \omega_{ik} \left(
	\numflux{f}^S_{i(j-1,j)k} -\numflux{f}^S_{i(j,j+1)k} + \frac{1}{2} \left( r_{i(j-1,j)k} + r_{i(j,j+1)k} \right) 
	\right)
	\\&
	+ \omega_{ij} \left(
	\numflux{f}^S_{ij(k-1,k)} -\numflux{f}^S_{ij(k,k+1)} + \frac{1}{2} \left( r_{ij((k-1,k)} + r_{ij(k,k+1)} \right) 
	\right) \\
	&+
	\underbrace{
	\vec{\Psi}_{ijk} \cdot 
	\left[
	\omega_{jk} \left(\vn_{(i,i+1)jk} - \vn_{(i-1,i)jk} \right)+
	\omega_{ik} \left(\vn_{i(j,j+1)k} - \vn_{i(j-1,j)k} \right)+
	\omega_{ij} \left(\vn_{ij(k,k+1)} - \vn_{ij(k-1,k)} \right)
	\right]
	}_{(e)=0}.
\end{align*}

The last term is equal to zero if the mesh is watertight at the subcell level.
Replacing the definition of the subcell metrics, \eqref{eq:SubcellMetrics}, in $(e)$ we obtain
\begin{align*}
(e) &=
\vec{\Psi}_{ijk} \cdot 
	\left[
	\omega_{jk} \sum_{m=0} Q_{im} (J\vec{a}^1)_{mjk} +
	\omega_{ik} \sum_{m=0} Q_{jm} (J\vec{a}^2)_{imk} +
	\omega_{ij} \sum_{m=0} Q_{km} (J\vec{a}^3)_{ijm}
	\right]
	\\
&= 
 \omega_{ijk} \vec{\Psi}_{ijk} \cdot 
    \sum_{m=0} D_{im} (J\vec{a}^1)_{mjk} +
			   D_{jm} (J\vec{a}^2)_{imk} +
			   D_{km} (J\vec{a}^3)_{ijm},
\end{align*}
which is again equal to zero if the discrete metric identities of the DGSEM, \eqref{eq:DiscMetricIdentities}, hold.

\end{proof}

\subsection{Hybrid FV/DGSEM Scheme}

The hybrid FV/DGSEM scheme reads
\begin{equation} \label{eq:BlendedScheme3D}
J_{ijk} \omega_{ijk} \dot{\state{u}}_{ijk} = 
(1-\alpha) \state{F}_{ijk}^{a,\DG} 
+ \alpha \state{F}_{ijk}^{a,\FV} 
- \state{F}_{ijk}^{\nu,\DG}.
\end{equation}

Lemma \ref{lemma:WillItBlend3D?} is the extension of Lemma \ref{lemma:WillItBlend?} to 3D.

\begin{lemma} \label{lemma:WillItBlend3D?}
The semi-discrete entropy balance of a discretization scheme for a non-conservative system that is obtained by blending two schemes at the element level,
\begin{equation}  \label{eq:Blending3D}
J_{ijk} \omega_{ijk} \dot{\state{u}}_{ijk} = 
(1-\alpha) \state{F}_{ijk}^{\DG} 
+ \alpha \state{F}_{ijk}^{\FV}, ~~~ \forall i,j,k \in [0,N],
\end{equation}
where $\alpha$ is an element-local blending coefficient, and both of the schemes are of the form
\begin{align}  \label{eq:SchemeToBlend3D}
\state{F}^{m}_{ijk} =
& \state{F}^m_{\elemDom,ijk}
\nonumber\\
+&
\omega_{jk} \left(
  \delta_{i0} \left[\numfluxb{f}^a_{(0,L)jk} + \numnonconsD{\Jan}_{(0,L)jk} \right]
- \delta_{iN} \left[\numfluxb{f}^a_{(N,R)jk} + \numnonconsD{\Jan}_{(N,R)jk} \right] \right)
\nonumber\\
+& \omega_{ik} \left(
  \delta_{j0} \left[\numfluxb{f}^a_{i(0,L)k} + \numnonconsD{\Jan}_{i(0,L)k} \right]
- \delta_{jN} \left[\numfluxb{f}^a_{i(N,R)k} + \numnonconsD{\Jan}_{i(N,R)k} \right] \right)
\nonumber\\
+& \omega_{ij} \left(
  \delta_{k0} \left[\numfluxb{f}^a_{ij(0,L)} + \numnonconsD{\Jan}_{ij(0,L)} \right]
- \delta_{kN} \left[\numfluxb{f}^a_{ij(N,R)} + \numnonconsD{\Jan}_{ij(N,R)} \right] \right),
~~~~~~~~ m=\FV, \DG.
\end{align}
with $\state{F}^m_{\elemDom,ijk}$ being any discretization terms that depend on the inner states of the element, is 
\begin{equation}
\sum_{i,j,k=0}^N J_{ijk} \omega_{ijk} \dot S_{ijk}=
\dot{S}_{\partial \elemDom} 
+ (1-\alpha) \dot{S}^{\DG}_{\elemDom} 
+ \alpha \dot{S}^{\FV}_{\elemDom},
\end{equation}
where $\dot{S}^m_{\elemDom}$ is the entropy production of the scheme $m$ inside the element, which only depends on inner states, and $\dot{S}_{\partial \elemDom}$ gathers the entropy flux and production on the boundaries of the element, which are intrinsic to the choice of the surface numerical flux function and the surface non-conservative term,
\begin{align*}
\dot{S}_{\partial \elemDom} =&
\ \ \  \sum_{j,k=0}^N \omega_{jk} \left(
	\numflux{f}^S_{(0,L)jk} -\numflux{f}^S_{(N,R)jk} + \frac{1}{2} \left( r_{(L,0)jk} + r_{(N,R)jk} \right) 
	\right)
	\\&
	+ \sum_{i,k=0}^N \omega_{ik} \left(
	\numflux{f}^S_{i(0,L)k} -\numflux{f}^S_{i(N,R)k} + \frac{1}{2} \left( r_{i(L,0)k} + r_{i(N,R)k} \right) 
	\right)
	\\&
	+ \sum_{i,j=0}^N \omega_{ij} \left(
	\numflux{f}^S_{ij(0,L)} -\numflux{f}^S_{ij(N,R)} + \frac{1}{2} \left( r_{ij(L,0)} + r_{ij(N,R)} \right) 
	\right).
\end{align*}

\end{lemma}

\begin{proof}
We start by contracting \eqref{eq:SchemeToBlend3D} for a scheme $m$ with the entropy variables and integrate over the element to obtain 
\begin{align} 
\sum_{i,j,k=0}^N \entVar_{ijk}^T \state{F}^{m}_{ijk} =&
\sum_{i,j,k=0}^N \entVar_{ijk}^T \state{F}^{m}_{\elemDom,ijk}
\nonumber\\
+&
\sum_{j,k=0}^N
\omega_{jk} \left(
  \entVar^T_{0jk} \left[\numfluxb{f}^a_{(0,L)jk} + \numnonconsD{\Jan}_{(0,L)jk} \right]
- \entVar^T_{Njk} \left[\numfluxb{f}^a_{(N,R)jk} + \numnonconsD{\Jan}_{(N,R)jk} \right] \right)
\nonumber\\
+& 
\sum_{i,k=0}^N
\omega_{ik} \left(
  \entVar^T_{i0k} \left[\numfluxb{f}^a_{i(0,L)k} + \numnonconsD{\Jan}_{i(0,L)k} \right]
- \entVar^T_{iNk} \left[\numfluxb{f}^a_{i(N,R)k} + \numnonconsD{\Jan}_{i(N,R)k} \right] \right)
\nonumber\\
+& 
\sum_{i,j=0}^N
\omega_{ij} \left(
  \entVar^T_{ij0} \left[\numfluxb{f}^a_{ij(0,L)} + \numnonconsD{\Jan}_{ij(0,L)} \right]
- \entVar^T_{ijN} \left[\numfluxb{f}^a_{ij(N,R)} + \numnonconsD{\Jan}_{ij(N,R)} \right] \right).
\end{align}

We use the strategy of summing and subtracting terms again for all the boundary degrees of freedom of the element (see e.g. \eqref{eq:StrategySumSubtract}) to obtain
\begin{align} 
\sum_{i,j,k=0}^N \entVar_{ijk}^T \state{F}^{m}_{ijk} =
\dot{S}^m_{\elemDom} + \dot{S}_{\partial \elemDom}
\end{align}
where
\begin{align*}
\dot{S}^m_{\elemDom} :=
\sum_{i,j,k=0}^N \entVar_{ijk}^T \state{F}^{m}_{\elemDom,ijk}
&+
\sum_{j,k=0}^N \omega_{jk} \left[
	 { (J\vec{a}^1)_{0jk} \cdot \vec{\Psi}_{0jk}}
	-{ (J\vec{a}^1)_{Njk} \cdot \vec{\Psi}_{Njk}}
\right]
\\
&+
\sum_{i,k=0}^N \omega_{ik} \left[
	 { (J\vec{a}^2)_{i0k} \cdot \vec{\Psi}_{i0k}}
	-{ (J\vec{a}^2)_{iNk} \cdot \vec{\Psi}_{iNk}}
\right]
\\
&+
\sum_{i,j=0}^N \omega_{ij} \left[
	 { (J\vec{a}^3)_{ij0} \cdot \vec{\Psi}_{ij0}}
	-{ (J\vec{a}^3)_{ijN} \cdot \vec{\Psi}_{ijN}}
\right].
\end{align*}

As a result, the entropy balance of the hybrid scheme is
\begin{equation}
\sum_{i,j,k=0}^N J_{ijk} \omega_{ijk} \dot S_{ijk} =
\dot{S}_{\partial \elemDom} 
+ (1-\alpha) \dot{S}^{\DG}_{\elemDom} 
+ \alpha \dot{S}^{\FV}_{\elemDom},
\end{equation}
where we remark that $\dot{S}_{\partial \elemDom}$ is the same for all the schemes of the form \eqref{eq:SchemeToBlend3D}.

\end{proof}

The main consequence of Lemma \ref{lemma:WillItBlend3D?} is that the hybrid scheme is semi-discretely \textit{entropy consistent} with the blended schemes.

\end{document}